%% file: main.tex
\documentclass[10pt,a4paper,sans]{article}
\usepackage[utf8]{inputenc}

\usepackage[square,sort,comma,numbers]{natbib}
\usepackage{graphicx}
\usepackage{xcolor}

\usepackage{mathtools}
\usepackage{amssymb}
\usepackage{amsmath}
\usepackage{amsthm}

\usepackage{caption}
\usepackage{subcaption}

\usepackage{booktabs}
\usepackage{dsfont}

\usepackage{cases}
\usepackage{hyperref}

\usepackage{comment}

\usepackage{epstopdf}

\usepackage[margin=2cm]{geometry}

\usepackage{amsthm}
\makeatletter
\def\th@plain{%
  \thm@notefont{}
  \itshape 
}
\def\th@definition{%
  \thm@notefont{}
  \normalfont 
}
\makeatother

\newtheorem{theorem}{Theorem}

\newtheorem{corollary}{Corollary}
\newtheorem{lemma}{Lemma}

\theoremstyle{definition}
\newtheorem{definition}{Definition}

\theoremstyle{remark}
\newtheorem{remark}{Remark}
\def\NN{{\mathbb N}}
\def\ZZ{{\mathbb Z}}
\def\RR{{\mathbb R}}

\usepackage{tikz}

\allowdisplaybreaks

\title{$d$-dimensional extension of a penalization method for Neumann or Robin boundary conditions: a boundary layer approach and numerical experiments\\(extended version)}
\author{Bouchra Bensiali\thanks{Centrale Casablanca, Complex Systems and Interactions Research Center, Ville Verte, 27182 Bouskoura, Morocco 
  (\href{mailto:bouchra.bensiali@gmail.com}{bouchra.bensiali@gmail.com}).}\\ (Boundary) \and Jacques Liandrat\thanks{Aix-Marseille Univ., CNRS, I2M, UMR7373, Centrale Méditerranée, 13451 Marseille, France 
  (\href{mailto:jacques.liandrat@centrale-med.fr}{jacques.liandrat@centrale-med.fr}).}\\ (Layer)}
\date{\today}

\begin{document}

\maketitle

\begin{abstract}
This paper studies the $d$-dimensional extension of a fictitious domain penalization technique that we previously proposed for Neumann or Robin boundary conditions. We apply Droniou's approach for non-coercive linear elliptic problems to obtain the existence and uniqueness of the solution of the penalized problem, and we derive a boundary layer approach to establish the convergence of the penalization method. The developed boundary layer approach is adapted from the one used for Dirichlet boundary conditions, but in contrast to the latter where coercivity enables a straightforward estimate of the remainders, we reduce the convergence of the penalization method to the existence of suitable supersolutions of a dual problem. These supersolutions are then constructed as approximate solutions of the dual problem using an additional formal boundary layer approach. The proposed approach results in an advection-dominated problem, requiring the use of appropriate numerical methods suitable for singular perturbation problems. Numerical experiments, using upwind finite differences, validate both the convergence rate and the boundary layer thickness, illuminating the theoretical results.
    
\end{abstract}

\tableofcontents

\bigskip
\bigskip
{\bf \large Notations}

\begin{table}[h]
\begin{tabular}{l p{1cm} p{10cm}}
\textbf{Spaces} & &\\
\hline
$C^k(\Omega)$, $k\in\NN$ & & set of functions, defined on $\Omega$, with continuous partial derivatives up to order $k$\\
$C^{k,\alpha}(\Omega)$, $k\in\NN$, $\alpha\in[0,1]$ & &  set of functions in $C^k(\Omega)$ whose partial derivatives of order $k$ are $\alpha$-Hölder continuous, or Lipschitz continuous for $\alpha=1$. In the case $\alpha=0$, $C^{k,0}(\Omega)=C^k(\Omega)$\\
$H^s(\Omega)$, $H^s(\partial\Omega)$ & & Sobolev spaces on $\Omega$, $\partial\Omega$ \\
 & & \\
\textbf{Geometry, see Figure~\ref{fig:domains}}  & & \\
\hline
$\mathcal U$ & & initial domain with boundary $\partial\mathcal U$\\
$\Omega$ & & fictitious domain embedding $\overline{\mathcal U}$, with boundary $\partial\Omega$\\
$\omega$ & & complementary domain $\omega=\Omega\setminus\overline{\mathcal U}$\\
$\tilde{n}$ & & outward unit normal vector on $\partial\mathcal{U}$\\
$\nu_\Omega$ & & outward unit normal vector on $\partial{\Omega}$\\
$n$ & & lifting of $\tilde{n}$ inside $\omega$\\
$\varphi$ & & distance to the boundary $\partial\Omega$\\
$\psi$ & & we assume $n=\nabla\psi$ in $\omega$\\
$\omega_1$ & & subdomain of $\omega$ such that $\varphi \le \eta$ for some constant $\eta>0$\\
& & \\
\end{tabular}
\end{table}

\newpage
\section{Introduction}

\input{sections/intro}

\section{Problem analysis and useful tools}\label{sec:problemform}

\input{sections/problemformulation}

\section{Existence and uniqueness}\label{sec:existenceuniqueness}

\input{sections/existenceuniqueness}

\section{Convergence analysis}\label{sec:convergenceBL}

\input{sections/convergenceBL}

\section{Numerical experiments}\label{sec:numerics}

\input{sections/numerics}

\section{Generalization}\label{sec:Generalisation}

\input{sections/generalisation}

\section{Conclusion}

\input{sections/conclu}

\bibliographystyle{plain}
\bibliography{references}

\appendix
\input{sections/appendix-Charac}

\input{sections/appendixDroniouChecking}

\input{sections/appendixspecialcases}

\input{sections/appendixDual}

\end{document}

%% file: sections/intro.tex
Volume penalization methods are numerical analysis methods that have been widely developed during the last 30 years in order to impose boundary conditions to the solution of a given partial differential equation. Generally (see Figure 1), an initial problem raised on an open set $\mathcal U$ with a boundary condition on the  boundary $\partial \mathcal U$ is replaced by another problem (called the penalized problem)  raised on a larger domain $\Omega$ including $\overline{\mathcal U}$. Hopefully, the new problem raised on $\Omega$ is easier to approximate numerically and one can recover from its numerical resolution an approximation of the initial problem. The motivations for doing so are usually one of the following: $\mathcal U$ can be a time varying domain that is difficult to approximate at each time step of a discretization; $\mathcal U$ can be a complex domain and a calculation on a simpler domain is desirable. Generally,   on one hand, the penalized problem involves a penalization term depending on a penalization parameter $\varepsilon$ and on the other hand, it appears immediately that extensions of various terms of the initial equation should be performed on the space $\omega$, complementing $\overline{ \mathcal U}$ in $\Omega$. These questions are the first ones to be answered when one define a penalization method. Moreover,  the mathematical questions connected to this approach deals with the solvability of the penalized problem and the convergence of its solution toward the solution of the initial problem in a sense to be precised, when the penalization parameter $\varepsilon$ goes to zero.

Reviews of different volume penalization methods , with applications in  fluid mechanics are available in (\cite{Pe02,MI05,Sch15}). 
A large amount of works deal with Dirichlet boundary conditions (\cite{ABF99,BookBoyerFabrie,angot1999analysis},...); fewer results concern Neumann or Robin boundary conditions in multidimension (\cite{RAB07,thirumalaisamy2022handling,sakurai2019volume},...).

This paper is devoted to the mathematical analysis of a volume penalization methos for Robin/Neumann boundary conditions in multidimension. Formally, this approach has been initially developped and analyzed for the univariate situation in (\cite{Bensiali2014} and used for numerical applications in \cite{Sch15}). Specific approach for the analysis in multidimension is necessary since explicit resolutions possible in the univariate case are no more available.

We will focalize on the following elliptic problem: 

\begin{equation}\label{eq:reactiondiffusion}
    \begin{dcases}
    \text{Find $u\in H^1(\mathcal{U})$ such that}\\
    -\Delta u+u=f & \text{in $\mathcal{U}$}\\
    \frac{\partial u}{\partial\tilde{n}}+\alpha u=\tilde{g} & \text{on $\partial\mathcal U$}, 
    \end{dcases}
\end{equation}
where $\mathcal{U}$ is a bounded Lipschitz domain of $\RR^d$, $f\in L^2(\mathcal U)$, $\tilde{g}\in H^{1/2}(\partial\mathcal U)$, $\alpha\ge 0$ and $\tilde{n}$ is the outward unit normal vector on $\partial\mathcal U$.
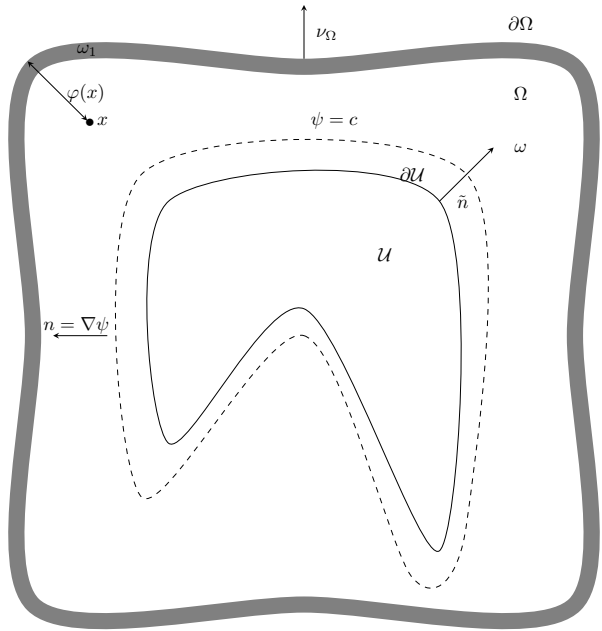
\begin{figure}[htb]
\centering
\resizebox{0.5\textwidth}{!}{
\begin{tikzpicture}

\draw [gray] [line  width=0.3cm] plot [smooth cycle] coordinates {(0,0) (0,5) (0,10) (5,10) (10,10)(10,5) (10,0) (5,0)};
\draw [black] plot [smooth cycle] coordinates {(2.5,3) (5,5.5) (7.5,1) (7.5,7.5) (2.5,7.5)};
\draw [dashed]  plot [smooth cycle] coordinates {(2.,2.) (5,5) (7,.5)  (8,1.5) (8,8) (2,8)};

\draw [-stealth] (5,10.15) --(5,11.15);
\node [right] at (5.1,10.6)  {$\nu_\Omega$};
\draw [-stealth] (7.5,7.5) --(8.5,8.5);
\node [right] at (7.7, 7.5)  {$\tilde{n}$};
\draw [-stealth] (1.37,5) --(.37,5);
\node  at (.8, 5.2)  {$n=\nabla \psi$};
\draw [stealth-stealth] (-.1,10.1) --(1,9);
\node [right] at (.5, 9.5)  {$\varphi(x)$};
\node [right] at (1.06,9) {$x$};
\node at (1.05,8.95) {\textbullet};
\node [right] at (5, 9)  {$\psi= c$};

\node at (9,9.5) {$\Omega$};
\node at (6.5,6.5) {$\mathcal{U}$};
\node at (7,8) {$ \partial \mathcal{U}$};
\node at (9,10.8) {$ \partial \Omega$};
\node at (1,10.3) {$\omega_1$};
\node at (9,8.5) {$\omega$};

\end{tikzpicture}
}
\caption{Domains and notations}
\label{fig:domains}
\end{figure}

The penalized problem, obtained generalizing the formulation of (\cite{Bensiali2014}), where $\varepsilon>0$ is called the penalization parameter, reads:

\begin{equation}\label{eq:reactiondiffusionpenalized}
    \begin{dcases}
    \text{Find $u_\varepsilon\in H^1_0(\Omega)$ such that}\\
    -\Delta u_\varepsilon+u_\varepsilon+\frac{\chi}\varepsilon\bigl(\nabla u_\varepsilon\cdot n+\alpha u_\varepsilon - g\bigr)= (1-\chi)f & \text{in $\Omega$},
    \end{dcases}
\end{equation}
where $\Omega$ is a bounded Lipschitz domain of $\RR^d$ ($\Omega \supset \overline{\mathcal U}$), $\chi$ is the characteristic function of the extended domain $\omega:=\Omega\setminus \overline{\mathcal U}$, $g$ is a lifting of $\tilde{g}$ such that $g\in H^1(\omega)$ and $n$ is a lifting of $\tilde{n}$ such that $n\in H^1(\omega)$ (we will assume for instance $\tilde{n}\in H^{1/2}(\partial\mathcal U)$). Moreover, we assume that there exists a function $\psi$ such that $\partial \mathcal U=\{x\in \RR^d/\psi(x)=0\}$ and $\tilde{n}=\nabla \psi$. On Figure \ref{fig:domains} one can see $\varphi(x)$ the distance to the boundary $\partial \Omega$ and a domain $\omega_1$ (in grey) defined as a subdomain of $\Omega$ such that for some $\delta>0$, to be precised later, $\forall x \in \omega_1, \varphi(x) \leq \delta$.

Our analysis uses two main ingredients: A result of existence and uniqueness in the framework of elliptic PDE suffering from a lack of coercivity (\cite{droniou2002}, \cite{dronioupota2002}) and a boundary layer approach to prove convergence of the penalization solution  following  \cite{BookBoyerFabrie} when the penalization parameter goes to zero.

The paper is therefore organized as follows: We present in section \ref{sec:problemform} the problem we are interested in and we recall useful results and mathematical tools. Section \ref{sec:existenceuniqueness} is devoted to the establishment of existence an uniqueness of the solution $u_\varepsilon$ of the penalized problem, while section \ref{sec:convergenceBL} provides the convergence towards the solution $u$ of the initial problem. Numerical experiments are reported in section \ref{sec:numerics} before concluding remarks. We have gathered into appendices some proofs, specific results in the univariate and spherical symmetric cases, and technical developments.

To give a taste of the penalization method, Figure~\ref{fig:resultsellipse} shows the results obtained using FreeFEM software for an ellipse $\mathcal U$ defined by $\psi(x,y)=\frac{x^2}{a^2}+\frac{x^2}{b^2}-1\le 0$, with $a=2$, $b=1$, that we embed inside a rectangle ${\Omega={]-}(a+e),a+e{[}\times{]-}(b+e),b+e{[}}$, with $e=0.5$. The outward normal vector on $\partial\mathcal U$ satisfies $\tilde{n}=\frac{\nabla\psi}{|\nabla\psi|}$ on $\partial\mathcal{U}$ and we extend it to $\omega$ by $n=\frac{\nabla \psi}{|\nabla\psi|}$ (note that another choice of $\psi$ as the distance $d$  to the ellipse would have yielded a normalized vector $n=\nabla d$ at each point of $\omega$, but the analytic expression of the distance to the ellipse is not known, and we believe that our approach would still work for $n$ of the form $\frac{\nabla \psi}{|\nabla\psi|}$ instead of $\nabla\psi$). Figure~\ref{fig:vectorstreamplotellipse} shows the extended vector $n$ along with the corresponding integral curves. This example allows to introduce the different parts that we will encounter in this paper:
\begin{itemize}
    \item In the initial domain $\mathcal{U}$, the solution of the penalized problem converges, when $\varepsilon\to 0$, towards the solution of the initial problem, while in the complementary domain $\omega$, far from the boundary, the solution of the penalized problem converges towards the solution of
the advection-reaction equation
    \begin{equation}
    \begin{dcases}
    \nabla W\cdot n+\alpha W=g & \text{in $\omega$}\\
    W=u & \text{on $\partial\mathcal U$}.
    \end{dcases}
    \end{equation}
In the considered case (Figure~\ref{fig:resultsellipse}), since $\alpha=0$ and $g=0$, this means that $W$ is constant along the integral curves of $n$, which is indeed the case in Figure~\ref{fig:resultsellipse}, for $\varepsilon$ small enough. The convergence analysis is carried out in Section~\ref{sec:convergenceBL}.
    \item A boundary layer is present in the neighborhood of $\partial\Omega$, in order to satisfy $u_\varepsilon\in H^1_0(\Omega)$ (Sections~\ref{sec:convergenceBL} and~\ref{sec:numerics}).
    \item The use of standard finite elements can produce stability issues if $h$ is not small enough with respect to $\varepsilon$, thus the use of suitable numerical methods, e.g. upwind finite differences, in Section~\ref{sec:numerics}.
\end{itemize}


\begin{figure}[h!]
	\centering
	\centering
	\begin{subfigure}[b]{0.52\textwidth}
		\centering
		\includegraphics[width=0.65\textwidth]{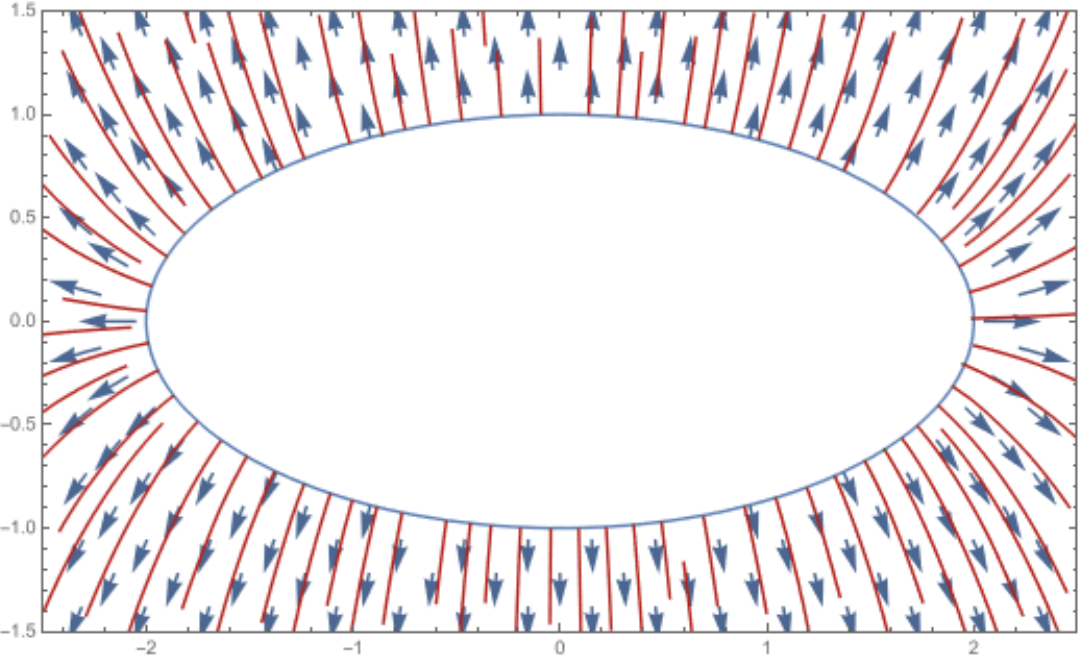}
		\vspace{1cm}
		\caption{Vector and stream plots of $n$}
		\label{fig:vectorstreamplotellipse}
		
	\end{subfigure}
	\hfill
	\begin{subfigure}[b]{0.46\textwidth}
		\centering
		\includegraphics[width=\textwidth]{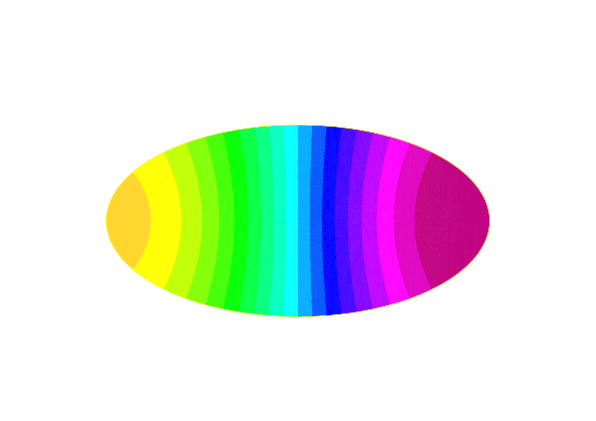}
		\caption{Numerical solution of the initial problem}
		\label{fig:exactsolellipse}
	\end{subfigure}\\
	\begin{subfigure}[b]{0.48\textwidth}
		\centering
		\includegraphics[width=\textwidth]{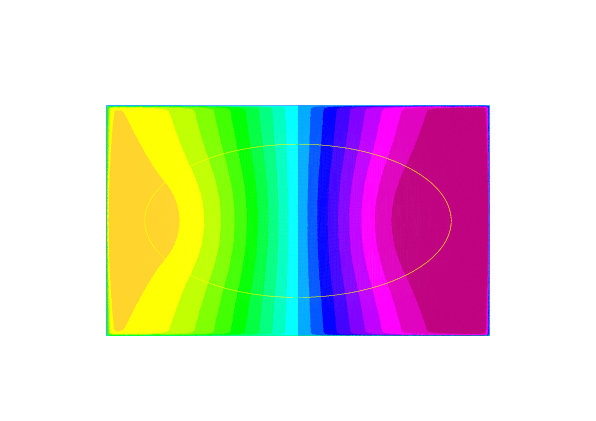}
		\caption{Numerical solution of the penalized problem (fitted mesh)}
		\label{fig:numsolellipse}
	\end{subfigure}
	\hfill
	\begin{subfigure}[b]{0.48\textwidth}
		\centering
		\includegraphics[width=\textwidth]{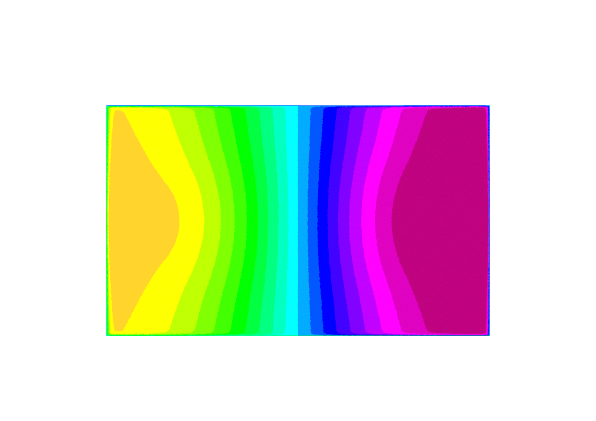}
		\caption{Numerical solution of the penalized problem (unfitted mesh)}
	\end{subfigure}
	\caption{Comparison between the numerical solution of the penalized problem~\eqref{eq:reactiondiffusionpenalized} and the numerical solution of the initial problem~\eqref{eq:reactiondiffusion} for $\alpha = 0$, $f(x,y)=\cos(x)\sin(y)$ and $\tilde{g}(x,y)=0$. We took the penalization parameter $\varepsilon=10^{-2}$ and the mesh size $h =0.04$ (fitted mesh) and $h=0.02$ (unfitted mesh).}
	\label{fig:resultsellipse}
\end{figure}

%% file: sections/problemformulation.tex
In this section, we present some analysis of the problem, as well as a recall of results and useful tools that will be used in the next sections, especially to establish the convergence of the penalization method. Some of these results are proved if no reference is given.

\subsection{Reaction-diffusion equation with Neumann or Robin boundary conditions; some analysis}

We first start with existence, uniqueness and stability analysis of the initial problem \eqref{eq:reactiondiffusion}.

\subsubsection{Existence, uniqueness and stability}

\begin{theorem}[Well-posedness]\label{th:wellposedness} (see for instance~\cite{E10})
If $\mathcal{U}$ is a connected, bounded, Lipschitz domain of $\RR^d$, if  $f\in L^2(\mathcal U)$, $\tilde{g}\in H^{1/2}(\partial\mathcal U)$ and $\alpha\ge 0$, then, there exists a unique solution $u\in H^1(\mathcal U)$ of~\eqref{eq:reactiondiffusion}, and it satisfies, for some $C>0$ depending only on $\mathcal U$ and $\alpha$,
\begin{equation}
\|u\|_{H^1}\le C(\|f\|_{L^2}+\|\tilde{g}\|_{H^{1/2}}).
\end{equation}

\end{theorem}

\subsubsection{Elliptic regularity}

We say that $\Omega$ is a domain of class $C^{k,\alpha}$ if its boundary $\partial\Omega$ is locally the graph of a function of class $C^{k,\alpha}$ such that $\Omega$ is situated on one single side of this graph.

\begin{theorem}[Elliptic regularity for Neumann or Robin boundary conditions]\label{th:ellipticregularity} 
Let $\mathcal U$ and $\alpha$ be as in Theorem~\ref{th:wellposedness}. If for $k\ge 0$, $\mathcal U$ is of class $C^{k+1,1}$, $f\in H^{k}(\mathcal U)$ and $\tilde{g}\in H^{k+1/2}(\partial\mathcal U)$, then the solution $u$ of~\eqref{eq:reactiondiffusion} satisfies $u\in H^{k+2}(\mathcal U)$.

\end{theorem}

\begin{proof}
We use the elliptic regularity property for the Laplace-Neumann problem (\cite{BookBoyerFabrie}, \cite{E10}) and a bootstrapping argument. Indeed, the following theorem holds for Neumann boundary conditions.

\begin{theorem}[Elliptic regularity for the Laplace-Neumann problem]\label{th:ellipticregularityNeumann}    (\cite{BookBoyerFabrie}, \cite{E10})
Let $\mathcal U$ be a connected, bounded, Lipschitz domain of $\RR^d$. Let $f_v\in L^2(\mathcal U)$ and $f_b\in H^{-1/2}(\partial\mathcal U)$ satisfying the compatibility condition
\begin{equation}\label{eq:compatibilitycondition}
    \int_{\mathcal U} f_v+\langle f_b,1\rangle_{H^{-1/2},H^{1/2}}=0.
\end{equation}
There exists a unique solution $u\in H^1(\mathcal U)\cap L^2_0(\mathcal U)=\{v\in H^1(\mathcal U), \ \int_{\mathcal U} v =0\}$ of
\begin{equation}
\begin{dcases}
    -\Delta u=f_v & \text{in $\mathcal U$}\\
    \frac{\partial u}{\partial\nu}=f_b & \text{on $\partial\mathcal U$},
\end{dcases}
\end{equation}
and it satisfies, for some $C>0$ depending only on $\mathcal U$,
\begin{equation}
    \|u\|_{H^1}\le C (\|f_v\|_{L^2}+\|f_b\|_{H^{-1/2}}).
\end{equation}
Moreover, if for $k\ge 0$, $\mathcal U$ is of class $C^{k+1,1}$, $f_v\in H^k$ and $f_b\in H^{k+1/2}(\partial\mathcal U)$ then $u\in H^{k+2}(\mathcal U)$ and we have
\begin{equation*}
    \|u\|_{H^{k+2}}\le C (\|f_v\|_{H^k}+\|f_b\|_{H^{k+1/2}}).
\end{equation*}
\end{theorem}

We use a proof by induction on $k\ge0$.

\begin{itemize}
    \item We first prove the result for $k=0$: We assume $\mathcal U$ of class $C^{1,1}$, $f\in H^0(\mathcal U)$ and $\tilde{g}\in H^{1/2}(\partial\Omega)$. The solution $u\in H^1(\mathcal{U})$ of~\eqref{eq:reactiondiffusion} is solution to the following Laplace-Neumann problem
    \begin{equation}\label{eq:neumannproblem}
    \begin{dcases}
    -\Delta u_*=f-u=f_v & \text{in $\mathcal U$}\\
    \frac{\partial u_*}{\partial\nu}=\tilde{g}-\alpha u=f_b & \text{on $\partial\mathcal U$}.
    \end{dcases}
    \end{equation}
    Since the previous problem admits a solution, the compatibility condition~\eqref{eq:compatibilitycondition} is necessarily satisfied, and using Theorem~\ref{th:ellipticregularityNeumann}, there exists a unique solution $u_*\in H^1(\mathcal U)$ such that $\int_{\mathcal U} u_*=0$, and $u=u_*+c$ where ${c=\frac1{\mu(\mathcal U)}\int_{\mathcal U} u}$ is a constant. Using the elliptic regularity for the Neumann-Laplace problem, since $f_v\in H^0(\mathcal U)$ and $f_b\in H^{1/2}(\partial\mathcal U)$, we obtain $u_*\in H^2(\mathcal U)$, thus $u\in H^2(\mathcal U)$.

    \item We assume the proposition true for a given $k\ge 0$ and prove it holds for $k+1$: We assume $\mathcal U$ of class $C^{k+2,1}$, $f\in H^{k+1}(\mathcal U)$ and $\tilde{g}\in H^{k+1+1/2}(\partial\Omega)$. By the induction hypothesis, the solution $u$ of~\eqref{eq:reactiondiffusion} satisfies $u\in H^{k+2}(\mathcal{U})$. Moreover, $u$ is solution to the Laplace-Neumann problem~\eqref{eq:neumannproblem}, therefore using again the elliptic regularity for the Laplace-Neumann problem (Theorem~\ref{th:ellipticregularityNeumann}), we obtain $u\in H^{k+3}(\mathcal U)$, since $f_v\in H^{k+1}(\mathcal U)$ and $f_b\in H^{k+1+1/2}(\partial\mathcal U)$. 

    \item The proposition thus holds for all $k\ge 0$, by induction.
\end{itemize}
\end{proof}

\subsubsection{Maximum principle}

\begin{theorem}[Maximum principle for Neumann or Robin boundary conditions]\label{th:maximumprinciple}
Let $\mathcal U$, $f$, $\tilde{g}$ and $\alpha$ be as in Theorem~\ref{th:wellposedness}.  If $f\ge 0$ and $\tilde{g}\ge 0$ almost everywhere, then the solution to~\eqref{eq:reactiondiffusion} satisfies $u\ge 0$ almost everywhere in $\mathcal U$.
\end{theorem}

\begin{proof}
The proof is identical to the one for the Laplace-Dirichlet problem \cite{AE23}, \cite{BookBrezis}. We use the weak formulation of~\eqref{eq:reactiondiffusion} with the test function $v=u^-=\max(-u,0) \in H^1(\mathcal U)$. We obtain
\begin{align*}
    \int_{\mathcal U} \nabla u\cdot\nabla u^-+\int_{\mathcal U} u u^-+\alpha\int_{\partial\mathcal U} u u^-&=\int_{\mathcal U} f u^-+\int_{\partial\mathcal U} \tilde{g} u^- \\
    -\int_{\mathcal U} \mathds{1}_{u<0} |\nabla u|^2-\int_{\mathcal U} \mathds{1}_{u<0} |u|^2-\alpha\int_{\partial\mathcal U} \mathds{1}_{u<0} |u|^2&=\int_{\mathcal U} f u^-+\int_{\partial\mathcal U} \tilde{g} u^-\\
    -\int_{\mathcal U} |\nabla u^-|^2-\int_{\mathcal U} |u^-|^2-\alpha\int_{\partial\mathcal U} |u^-|^2&=\int_{\mathcal U} f u^-+\int_{\partial\mathcal U} \tilde{g} u^-\\
    -\|u^-\|^2_{H^1}-\alpha\int_{\partial\mathcal U} |u^-|^2&=\int_{\mathcal U} f u^-+\int_{\partial\mathcal U} \tilde{g} u^-.
\end{align*}
The right hand side is positive by the positivity of $f$ and $\tilde{g}$, while the first hand side is negative. Thus all the terms vanish, in particular $\|u^-\|_{H^1}=0$, thus $u^-=0$, that is $u\ge 0$ almost everywhere in $\mathcal U$.
\end{proof}

\begin{corollary}\label{cor:maximumprinciplebis}
Let $\mathcal U$, $f$, $\tilde{g}$ and $\alpha$ be as in Theorem~\ref{th:wellposedness}. If $f\ge c>0$ (where $c$ is a constant) and $\tilde{g}\ge \alpha c\ge 0$ almost everywhere, then the solution to~\eqref{eq:reactiondiffusion} satisfies $u\ge c>0$ almost everywhere in $\mathcal U$.
\end{corollary}

\begin{proof}
    $U=u-c$ is solution to the following reaction-diffusion problem
    \begin{equation}\label{eq:maxprinicple}
    \begin{dcases}
    \text{Find $U\in H^1(\mathcal{U})$ such that}\\
    -\Delta U+U=f-c & \text{in $\mathcal{U}$}\\
    \frac{\partial U}{\partial\tilde{n}}+\alpha U=\tilde{g}-\alpha c & \text{on $\partial\mathcal U$}, 
    \end{dcases}
\end{equation}
and the conclusion $U\ge 0$ a.e. follows using Theorem~\ref{th:maximumprinciple}. 
\end{proof}

\subsection{Mathematical tools}

\subsubsection{The space $H_{\mathrm{div}}$}

We recall the definition of the space $H_{\mathrm{div}}(\Omega)$ and some of its properties that will be crucial for our study. The proofs can be found in \cite{BookBoyerFabrie}.

\begin{definition}
Let $\Omega$ be a Lipschitz bounded domain of $\RR^d$. We introduce the space
$$H_\mathrm{div}(\Omega)=\{u\in (L^2(\Omega))^d, \ \mathrm{div}\,u\in L^2(\Omega)\},$$
equipped with the graph norm $u\mapsto (\|u\|_{L^2}^2+\|\mathrm{div}\,u\|_{L^2}^2)^{1/2}$.
\end{definition}

\begin{theorem}\label{th:ippHdiv}
    There exists a continuous trace operator $\gamma_\nu$ from $H_\mathrm{div}(\Omega)$ into $H^{-1/2}(\partial\Omega)$ such that $\gamma_\nu(u)=u\cdot \nu$ for any $u\in(C^\infty_c(\overline{\Omega}))^d$, with $\nu$ the outward normal to $\Omega$. Moreover the following Stokes (integration by parts) formula holds, for any $u\in H_\mathrm{div}(\Omega)$ and any $w\in H^1(\Omega)$,
    \begin{equation}\label{eq:ippHdiv}
    \int_\Omega u\cdot \nabla w +\int_\Omega w\, \mathrm{div}\, u=\langle \gamma_\nu (u),\gamma_0(w) \rangle_{H^{-1/2},H^{1/2}},
    \end{equation}
    where $\gamma_0$ stands for the classical trace operator from $H^1(\Omega)$ to $H^{1/2}(\partial\Omega)$.
\end{theorem}

\subsubsection{Sobolev embeddings}\label{sec:sobolevembeddings}

The Sobolev inclusions that will be useful for our study concern inclusions between $H^k$ and $C^m$ spaces \cite{BookBrezis}.

\begin{theorem}[Sobolev embeddings]
    Let $\Omega$ be a Lipschitz open set of $\RR^d$, and $s\in\NN$ such that $2s>d$. Let $k\in\NN$ such that $s-k>d/2$. Then, the following continuous inclusion holds
    \begin{equation*}
        H^s(\Omega)\hookrightarrow C^k(\overline{\Omega}).
    \end{equation*}
\end{theorem}

\begin{corollary}
    Let $\Omega$ be a Lipschitz open set of $\RR^d$. Then the following Sobolev embedding holds, for all $k\in\NN$,
        \begin{equation*}
        H^{k+\lfloor \frac{d}2\rfloor+1}(\Omega)\hookrightarrow C^k(\overline{\Omega}).
    \end{equation*}
    For $d=2$ or $3$, one has
    \begin{equation*}
        H^{k+2}(\Omega)\hookrightarrow C^k(\overline{\Omega}).
    \end{equation*}
\end{corollary}

\subsubsection{Boundary regularity and distance to the boundary}\label{sec:distanceboundary}

For any $x\in\Omega$, we define $\delta(x)$  the signed distance from $x$ to the boundary, which is defined by
\begin{equation}
    \delta(x)=
    \begin{dcases}
        d(x,\partial\Omega) & \text{for $x\in\overline{\Omega}$}\\
        -d(x,\partial\Omega) & \text{for $x\notin\Omega$}.
    \end{dcases}
\end{equation}
$\delta$ is Lipschitz continuous on $\RR^d$. In addition, the regularity of $\Omega$ translates on the regularity of $\delta$ \cite{BookBoyerFabrie}.

\begin{theorem}[Regularity of the signed distance function]\label{th:distanceregularity}
    If $\Omega$ is a domain of $\RR^d$ of class $C^{k+1,1}$, $k\ge 0$, with compact boundary, then there exists $\gamma>0$ such that $\delta$ is $C^{k+1,1}$ in
    \begin{equation*}
        \mathcal{O}_\gamma=\{x\in\Omega, \ \delta(x)<\gamma\}.
    \end{equation*}
\end{theorem}

\subsubsection{Solution to advection-reaction equation via the method of characteristics}

In our proof of the convergence of the penalization method by a boundary layer approach (Section~\ref{sec:convergenceBL}), we will encounter, as expected formally, first order partial differential equations of advection-reaction of the following type:

\begin{equation}\label{eq:advectionreaction}
    \begin{dcases}
    \text{Find $W\in C^1(\overline{\omega})$ such that}\\
    \nabla W\cdot n+\alpha W=g & \text{in $\omega$}\\
    W=V & \text{on $\partial\mathcal U$},
    \end{dcases}
\end{equation}
where the domains $\mathcal U$ and $\omega$ are as in Figure~\ref{fig:domains}.

This type of first order steady state partial differential equations was studied for instance in~\cite{lions2006perturbations,bardos1970problemes,ern2021finite}. 
Here,  we choose  to make suitable assumptions that allow to tackle this kind of equations using the method of characteristics.

We  therefore  suppose that there exists  a function $\psi \in C^1$  such that $n=\nabla\psi$ in $\omega$, $\psi=0$ on $\partial\mathcal{U}$, $\psi>0$ in $\omega$ and $|\nabla\psi|\ge \lambda>0$ in $\omega$. We also assume $n \cdot \nu_\Omega>0$ on $\partial\Omega$ where $\nu_\Omega$ is the outward unit normal vector to $\partial\Omega$ where $\Omega=\mathcal{U}\cup \partial\mathcal U \cup \omega$.  Under these assumptions, we can establish the following theorem.

\begin{theorem}[Hyperbolic regularity]\label{th:hyperbolicregularitybis}
Let $\mathcal U$ be a bounded domain of class $C^k$ for $k\ge 1$. If $\psi\in C^{k+1}(\overline{\omega})$, $\beta\in C^k(\overline{\omega})$, $g\in C^{k}(\overline{\omega})$ and $V\in C^k(\partial\mathcal U)$, then there exists a unique solution $W$ of
\begin{equation}\label{eq:advectionreactionbis}
    \begin{dcases}
    \text{Find $W\in C^1(\overline{\omega})$ such that}\\
    \nabla W\cdot n+\beta(x) W=g & \text{in $\omega$}\\
    W=V & \text{on $\partial\mathcal U$},
    \end{dcases}
\end{equation}
and $W\in C^k(\overline{\omega})$.
\end{theorem}

The proof is reported in Appendix~\ref{sec-append-Charac}.

%% file: sections/existenceuniqueness.tex
In this section, we study the existence and uniqueness of the solution of the penalized problem~\eqref{eq:reactiondiffusionpenalized}.

\subsection{Weak formulation}

The weak formulation of the penalized problem~\eqref{eq:reactiondiffusionpenalized} reads
\begin{equation}\label{eq:penalizedproblemweak}
    \begin{dcases}
    \text{Find  $u_\varepsilon\in H^1_0(\Omega)$ such that $\forall \varphi\in H^1_0(\Omega)$}\\
    \int_\Omega (\nabla u_\varepsilon\cdot\nabla \varphi+u_\varepsilon\varphi)+\frac{1}\varepsilon \int_{\omega}(\nabla u_\varepsilon\cdot n\, \varphi+\alpha u_\varepsilon\varphi)=\int_\mathcal U f\varphi +\frac1\varepsilon\int_{\omega}g\varphi,
    \end{dcases}
\end{equation}
where we recall $f\in L^2(\mathcal U)$, $g\in H^1(\omega)$ and $n\in H^1(\omega)$. We will see, in Subsection~\ref{subsec:existenceuniqueness} proving the existence and uniqueness, that these assumptions, particularly on $n$, make the weak formulation well defined in dimensions $2$ or $3$. We will add the assumption $n\in L^d(\omega)$ in dimension $d\ge 4$ in order that the term $\int_{\omega} \nabla u_\varepsilon\cdot n\, \varphi$ be defined, see~\cite{dronioupota2002}.

\subsection{Lack of coercivity}

\begin{lemma}[Lack of coercivity]\label{lem:lackofcoercivity}
     Let $\mathcal U$ and $\Omega$ be bounded Lipschitz domains of $\RR^d$ (with $\Omega \supset \overline{\mathcal U}$), $\chi$ the characteristic function of the extended domain $\omega:=\Omega\setminus \overline{\mathcal U}$, $f\in L^2(\mathcal U)$, $\alpha\ge 0$, $g\in H^1(\omega)$ and $n\in H^1(\omega)$ (for $d\in\{2,3\}$) or $n\in L^d(\omega)$ for $d\ge4$. The bilinear form associated with the penalized problem~\eqref{eq:penalizedproblemweak} is not always coercive for $\varepsilon$ small enough.
\end{lemma}

\begin{proof}
The bilinear form associated with the penalized problem reads:
\begin{equation}
    a(u,v)=\int_\Omega (\nabla u \cdot\nabla v+u v)+\frac{1}\varepsilon \int_{\omega}(\nabla u \cdot n\, u+\alpha u v)
\end{equation}
for $u,v\in H^1_0(\Omega)$. Thus
\begin{equation}
    a(u,u)=\int_\Omega |\nabla u |^2+\int_\Omega u^2+\frac{1}\varepsilon \int_{\omega}\nabla u \cdot n\, u+\frac{\alpha}\varepsilon \int_{\omega} u^2.
\end{equation}
Using Green's theorem (when it is valid, for instance if $n\in C^1(\overline{\omega})$), for all $\varphi\in H^1_0(\Omega)$,
\begin{equation}\label{eq:green}
    \int_{\omega} \nabla u_\varepsilon\cdot n\, \varphi=-\int_{\omega}( u_\varepsilon\nabla\varphi\cdot n+u_\varepsilon \,\mathrm{div}(n)\,\varphi)- \int_{\partial\mathcal U} u_\varepsilon \varphi
\end{equation}
thus
\begin{equation}
    \int_{\omega}\nabla u \cdot n\, u=-\int_{\omega}u\,  \nabla u \cdot n -\int_{\omega} u^2 \mathrm{div}(n) - \int_{\partial\mathcal U} u^2
\end{equation}
then
\begin{align*}
    a(u,u)&=\int_\Omega |\nabla u |^2+\int_\Omega u^2-\frac{1}{2\varepsilon} \int_{\omega}u^2 \mathrm{div}(n) -\frac{1}{2\varepsilon} \int_{\partial\mathcal U} u^2+\frac{\alpha}\varepsilon \int_{\omega} u^2\\
    &=\int_\Omega |\nabla u |^2+\int_\Omega u^2+\frac{1}{2\varepsilon} \biggl[\int_{\omega} \bigl(2\alpha-\mathrm{div}(n)\bigr) u^2  -\int_{\partial\mathcal U} u^2\biggr]\\
    &=\|u\|_{H^1}^2+\frac{1}{2\varepsilon} \biggl[\int_{\omega} \bigl(2\alpha-\mathrm{div}(n)\bigr) u^2  - \int_{\partial\mathcal                
                                                           U} u^2\biggr].
\end{align*}
We are going to show that the dominant term (in front of $\frac1\varepsilon$) can be negative for a given $u\in H^1_0(\Omega)$, which will contradict the coercivity of the bilinear form $a$.

Assuming $\mathrm{div}(n)\in L^\infty(\omega)$ (which holds if $n\in C^1(\overline{\omega})$), one has
\begin{align*}
    \int_{\omega} \bigl(2\alpha-\mathrm{div}(n)\bigr) u^2-\int_{\partial\mathcal U} u^2&\le C \|u\|^2_{L^2(\omega)}- \int_{\partial\mathcal U} u^2\\
    &\le C \|u\|^2_{L^2(\omega)}- \|\gamma_0(u)\|^2_{L^2(\partial\omega)},
\end{align*}
and the result of non-coercivity is related to the non-continuity of the trace operator from $(H^1(\omega),\|\cdot\|_{L^2(\omega)})$  to $(L^2(\partial\omega),\|\cdot\|_{L^2(\partial\omega)})$. Indeed, let us show the following in our case (domains $\mathcal U$, $\omega$ and $\Omega$ as in Fig.~\ref{fig:domains}):
    \begin{align*}
        \forall C>0, \exists u\in H^1_0(\Omega),  \ \|\gamma_0(u)\|_{L^2(\partial\omega)} > C \|u\|_{L^2(\omega)}.
    \end{align*}

Let us assume $\omega$ of class $C^{1,1}$ then using Theorem~\ref{th:distanceregularity}, there exists a neighborhood of $\partial\omega$ in which the distance to the boundary is $C^{1,1}$. Let us define the following sequence of functions, for $j$ big enough:
\begin{equation}
    f_j(x)=\begin{dcases}
    \text{ solution to}\begin{dcases}
            -\Delta v=0 & \text{in $\mathcal U$} \\
            v=1 & \text{on $\partial\mathcal U$}
        \end{dcases} &\text{if } x\in\mathcal U \\
        \frac{1}{1+j d(x,\partial\mathcal U)} & \text{if } x\in \omega_j=\Bigl\{x\in\omega, \ d(x,\partial\mathcal U) < \frac1{\sqrt{j}}\Bigr\}\\
        \text{ solution to}\begin{dcases}
            -\Delta v=0 & \text{in $\omega\setminus\overline{\omega_j}$} \\
            v=0 & \text{on $\partial\Omega$}\\
            v=\frac1{1+\sqrt{j}} &  \text{on $\partial\omega_j$}
        \end{dcases} &\text{otherwise.}
    \end{dcases} 
\end{equation}
Clearly, $f_j\in H^1_0(\Omega)$, since ${f_j}_{|\mathcal U}\in H^1(\mathcal U)$, ${f_j}_{|\omega\setminus\overline\omega_j}\in H^1(\omega\setminus\overline\omega_j)$, and ${f_j}_{|\omega_j}\in H^1(\omega_j)$ (since $d(.,\partial\mathcal U) \in C^{1,1}(\omega_j)$ and Lipschitz continuous thus $f_j$ and its first order partial derivatives are bounded, allowing to show $f_j$ and $\frac{\partial f_j}{\partial x_i}$  in $L^2(\omega_j)$), and we have by construction continuity of the traces through the interfaces $\partial\mathcal U$ and $\partial\omega_j$. In addition, we have
\begin{equation}
\begin{dcases}
    f_j(x)\mathds{1}_{\omega_j}(x) \to 0 \quad \text{everywhere in  $\omega$}\\
    \forall j\in\NN, 0\le f_j(x)\mathds{1}_{\omega_j}(x) \le 1\quad \text{everywhere in  $\omega$}
\end{dcases}
\end{equation}
thus, by Lebesgue's dominated convergence theorem, we get
\begin{equation}
    \|f_j\|^2_{L^2(\omega_j)} \to 0.
\end{equation}

In $\omega\setminus\overline{\omega_j}$, if we assume $\omega$ of class $C^{m+2,1}$ for $m\in\NN$, so that $\omega\setminus\overline{\omega_j}$ is also of class $C^{m+2,1}$ then $f_j\in H^{m+2}(\omega\setminus\overline{\omega_j})$ by elliptic regularity and $f_j\in C^{2}(\overline{\omega\setminus\overline{\omega_j}})$ if $m>\frac{d}2$. In this case, we have using the weak maximum principle~\cite{evans2022partial}, $0\le f_j(x)\mathds{1}_{\omega\setminus\overline{\omega_j}}(x)\le \frac1{1+\sqrt{j}}$, hence
\begin{equation}
\begin{dcases}
    f_j(x)\mathds{1}_{\omega\setminus\overline{\omega_j}}(x) \to 0 \quad \text{everywhere in  $\omega$}\\
    \forall j\in\NN, 0\le f_j(x)\mathds{1}_{\omega\setminus\overline{\omega_j}}(x) \le 1\quad \text{everywhere in $\omega$}
\end{dcases}
\end{equation}
thus, again by Lebesgue's dominated convergence theorem, we get
\begin{equation}
\|f_j\|^2_{L^2(\omega\setminus\overline{\omega_j})} \to 0.
\end{equation}

If
    \begin{align*}
        \exists C>0, \forall u\in H^1_0(\Omega),  \ \|\gamma_0(u)\|_{L^2(\partial\omega)} \le C \|u\|_{L^2(\omega)},
    \end{align*}
then we would have for $u=f_j$
\begin{align}
    \mu(\partial\mathcal U)=\|\gamma_0(f_j)\|^2_{L^2(\partial\omega)} \le C \|f_j\|^2_{L^2(\omega)}\to0
\end{align}
which is in contradiction with $\mathcal U$ bounded of class $C^{m+2,1}$.

In conclusion, we exhibited special cases for $\mathcal U$, $\Omega$ and $n$ (regularity assumptions) for which the bilinear form associated to the penalized problem is not coercive for $\varepsilon$ small enough.
\end{proof}

\subsection{Existence and uniqueness using Droniou's approach}\label{subsec:existenceuniqueness}

Due to the lack of coercivity, we cannot use the Lax-Milgram theorem in order to prove the existence and uniqueness of the solution of the penalized problem. Instead, we are going to see that our penalized problem~\eqref{eq:penalizedproblemweak} falls within the framework of a theorem already established by J. Droniou~\cite{dronioupota2002} for non-coercive elliptic linear problems. The proof relies on a dual problem, that we will also encounter in our study of convergence of the penalization method. The dual problem (where the convection terms are in a conservative form) has the following form:
\begin{equation}\label{eq:dualpenalizedproblemweak}
    \begin{dcases}
    \text{Find  $r_\varepsilon\in H^1_0(\Omega)$ such that $\forall \varphi\in H^1_0(\Omega)$}\\
    \int_\Omega (\nabla r_\varepsilon\cdot\nabla \varphi+r_\varepsilon\varphi)+\frac{1}\varepsilon \int_{\omega}( r_\varepsilon\, n\cdot\nabla\varphi+\alpha r_\varepsilon\varphi)=\langle L,\varphi\rangle,
    \end{dcases}
\end{equation}
where $L\in H^{-1}(\Omega)$ and we recall $\alpha\ge 0$.

\begin{theorem}[Existence and uniqueness of the solution of the penalized problem and of the solution of the dual problem]\label{th:existenceuniquenessdroniou}
    Let $\mathcal U$ and $\Omega$ be bounded Lipschitz domains of $\RR^d$ (with $\Omega \supset \overline{\mathcal U}$), $\chi$ be the characteristic function of the extended domain $\omega:=\Omega\setminus \overline{\mathcal U}$, $f\in L^2(\mathcal U)$, $\alpha\ge 0$, $g\in H^1(\omega)$ and $n\in H^1(\omega)$. For all $\varepsilon>0$, there exists a unique solution $u_\varepsilon$ to the penalized problem~\eqref{eq:penalizedproblemweak} and a unique solution $r_\varepsilon$ to the dual problem~\eqref{eq:dualpenalizedproblemweak}, in dimension $d=2$ or $3$. If one can choose the lifting $n\in L^d(\omega)$ then the result holds also in any dimension $d\ge 4$.
\end{theorem}

\begin{proof}
The penalized problem and its dual form fall within the framework of a theorem already established by J. Droniou for non-coercive linear elliptic problems. Indeed, our penalized problem~\eqref{eq:penalizedproblemweak} is of the following form (Eq.~(4) of \cite{dronioupota2002}), where the convection terms are in  a non-conservative form:
\begin{equation}\label{eq:droniouproblemweak}
    \begin{dcases}
    \text{Find $v\in H^1_{\Gamma_d}(\Omega)$ such that $\forall \varphi\in H^1_{\Gamma_d}(\Omega)$}\\
    \int_\Omega A^T\nabla v\cdot\nabla \varphi+ \int_{\Omega}\varphi {\bf v}\cdot \nabla v+\int_\Omega bv\varphi+\int_{\Gamma_f}\lambda v \varphi\,d\sigma=\langle L,\varphi \rangle_{(H^1_{\Gamma_d}(\Omega))',H^1_{\Gamma_d}(\Omega)},
    \end{dcases}
\end{equation}
while the dual problem~\eqref{eq:dualpenalizedproblemweak} is of the following form (Eq.~(3) of \cite{dronioupota2002}), where the convection terms are in  a conservative form:
\begin{equation}\label{eq:droniouproblemweakdual}
    \begin{dcases}
    \text{Find $u\in H^1_{\Gamma_d}(\Omega)$ such that $\forall \varphi\in H^1_{\Gamma_d}(\Omega)$}\\
    \int_\Omega A\nabla u\cdot\nabla \varphi+ \int_{\Omega}u {\bf v}\cdot \nabla \varphi+\int_\Omega bu\varphi+\int_{\Gamma_f}\lambda u \varphi\,d\sigma=\langle L,\varphi \rangle_{(H^1_{\Gamma_d}(\Omega))',H^1_{\Gamma_d}(\Omega)},
    \end{dcases}
\end{equation}

[Theorem 2.1 (Droniou's article~\cite{dronioupota2002})] thus gives the existence and uniqueness of the solutions $v$ and $u$ under some assumptions ((8)--(14) in \cite{dronioupota2002}) that are verified by our penalized problem and its dual form as detailed in Appendix~\ref{sec:annexDroniouChecking}.
\end{proof}

%% file: sections/convergenceBL.tex
The convergence analysis of a problem of type~\eqref{eq:reactiondiffusion} in the univariate case is performed in~\cite{Bensiali2014}. It used explicit solutions that cannot be reached in general. 
The present section is devoted to a convergence analysis in multidimension using at different stages a boundary layer framework. This allows to obtain another proof for the one-dimensional case (\cite{Bensiali2014}) as well as a proof in the spherical symmetric case, using a boundary layer approach and explicit supersolutions of a dual problem. The special cases are reported in Appendix~\ref{sec:annexspecialcases}.

For sake of clarity we rewrite the initial problem~\eqref{eq:reactiondiffusion}
\begin{equation}\label{eq:Nd}
    \begin{dcases}
    -\Delta u+u=f \quad \text{in $\mathcal{U}$}\\
    \frac{\partial u}{\partial\tilde{n}}+\alpha u=\tilde{g} \quad \text{on $\partial\mathcal{U}$},
    \end{dcases}
\end{equation}
and the corresponding penalized problem~\eqref{eq:reactiondiffusionpenalized}
\begin{equation}\label{eq:penalizedNd}
    \begin{dcases}
    -\Delta u_\varepsilon+u_\varepsilon+\frac\chi\varepsilon (\nabla u_\varepsilon\cdot n+\alpha u_\varepsilon -g)=(1-\chi)f \quad \text{in $\Omega$}\\
    u_\varepsilon=0 \quad \text{on $\partial \Omega$}.
    \end{dcases}
\end{equation}

From now on, and in order to develop the boundary layer approach, we assume sufficient regularity on the data $\tilde{g}$ and  $f$ and on the domains $\mathcal{U}$ and $\omega$. We also make additional assumptions on the extension $n$. We  fix the dimension $d=2$ and postpone the general case to Section~\ref{sec:Generalisation}.

\clearpage
\vskip.5cm

{\bf Regularity hypotheses in dimension $d=2$:}

$\Omega$ and $\mathcal U$ are supposed to be  of class $C^{k+8,1}$ in dimension $2$ for some $k\ge 2$. Moreover, we suppose that $f\in H^{k+7}(\mathcal{U})$ and $\tilde{g}\in H^{k+7+\frac12}(\partial\mathcal U)$. We assume $n=\nabla\psi$ with $\psi\in C^{k+7}(\overline{\omega})$, $\psi=0$ on $\partial\mathcal{U}$, $\psi>0$ in $\omega$, $|\nabla\psi|\ge \lambda>0$ in $\omega$, and $n \cdot \nu_\Omega\ge \lambda_0>0$ on $\partial\Omega$, where $\nu_\Omega$ is the outward normal vector to $\Omega$. 
\vskip.5cm

The penalized problem~\eqref{eq:penalizedNd} is equivalent to the following system
    \begin{subnumcases}{}
    -\Delta w_\varepsilon+w_\varepsilon+\frac1\varepsilon (\nabla w_\varepsilon\cdot n+\alpha w_\varepsilon -g)=0 & \text{in $\omega$}\label{eq:wNd}\\ 
    -\Delta v_\varepsilon+v_\varepsilon=f & \text{in $\mathcal{U}$}\label{eq:vNd}\\
    w_\varepsilon=v_\varepsilon & \text{on $\partial\mathcal{U}$}\label{eq:transmission1Nd}\\
    \frac{\partial w_\varepsilon}{\partial\nu}=\frac{\partial v_\varepsilon}{\partial\nu} & \text{on $\partial\mathcal{U}$}\label{eq:transmission2Nd}\\
    w_\varepsilon=0 & \text{on $\partial\Omega$,}\label{eq:dirichletNd}
    \end{subnumcases}
where we distinguish the solution $w_\varepsilon$ inside the obstacle $\omega$ from the solution $v_\varepsilon$ in the fluid domain $\mathcal{U}$. The previous system follows from the fact that from its weak formulation, \eqref{eq:penalizedNd} is satisfied in a distributional sense and then almost everywhere since $u_\varepsilon\in H^1(\Omega)$ and given the regularity of the different data, thus the continuity of the trace $\gamma_0(u_\varepsilon)$ along $\partial\mathcal U$. We deduce that $\mathrm{div}(\nabla u_\varepsilon)=\Delta u_\varepsilon\in L^2(\Omega)$ hence $\nabla u_\varepsilon\in H_\mathrm{div}(\Omega)$, thus the continuity of $\gamma_\nu(\nabla u_\varepsilon)$ along $\partial\mathcal U$ thanks to Theorem~\ref{th:ippHdiv}.

We consider the following ansatz for the solution in terms of asymptotic expansions:
\begin{equation}\label{eq:penalizedNdansatz}
    \begin{dcases}
    v_\varepsilon(x)=V^0(x)+\varepsilon V^1(x)+\varepsilon^2 V^2(x)+\ldots,\\
    w_\varepsilon(x)=W^0\Bigl(x,\frac{\varphi(x)}\varepsilon\Bigr)+\varepsilon W^1\Bigl(x,\frac{\varphi(x)}\varepsilon\Bigr)+\varepsilon^2 W^2\Bigl(x,\frac{\varphi(x)}\varepsilon\Bigr)+\ldots
    \end{dcases}
\end{equation}
where $\varphi(x)=d(x,\partial\Omega)$ is the distance to the boundary $\partial\Omega$ (see Section~\ref{sec:distanceboundary}), and the profile terms inside the obstacle have the form
\begin{equation}\label{eq:Ndboundarylayer}
    \begin{dcases}
    W^i(x,z)=\overline{W}^i(x)+\theta(x)\widetilde{W}^i(x,z)\\
    \text{where} \ \forall i,k\ge0, \ \partial_z^k \widetilde{W}^i \xrightarrow[z \to +\infty]{} 0, \quad \forall x\in\omega\\
    \text{and $\theta\in C^\infty(\bar{\omega})$  such that $\theta\equiv 1$ in a neighborhood $\omega_0$ of $\partial\Omega$ and $\mathrm{supp}(\theta)\subset \omega_1$,}
    \end{dcases}
\end{equation}
where $\omega_1=\{x\in\omega, \varphi(x) < \delta\}$ and $\delta$ is such that the function $\varphi$ is sufficiently smooth. From Theorem~\ref{th:distanceregularity}, there exists indeed $\delta$ such that $\varphi\in C^{k+8,1}(\omega_1)$ if $\Omega\in C^{k+8,1}$. 
 This ansatz (if it is suitable) reflects the presence of a localized boundary layer near the boundary $\partial\Omega$ and its absence at the interface $\partial\mathcal{U}$ between the fluid and the obstacle.

 We denote $\Gamma_-=\{x\in\partial\Omega, \ c(x)=n\cdot (-\nabla\varphi)\le 0\}$, and we assumed that $c(x)\ge \lambda_0>0$ on $\partial\Omega$ so that $\Gamma_-=\emptyset$. This will simplify the calculations since it will hold a boundary layer everywhere on $\partial\Omega$. By this assumption and using the regularity of $n$ and $\nabla\varphi$, this implies (using Heine-Cantor theorem) that $-c(x)=n\cdot\nabla\varphi\le -\frac{\lambda_0}2<0$ in a neighborhood of $\partial\Omega$. We can choose $\delta$ such that  $\omega_1$, the support of $\theta$, is included in this neighborhood.

Then, we have the following result:

\subsection{Statement of the main result}

\begin{theorem}[Convergence of the penalization method]\label{th:convergence}

Assuming the regularity hypotheses quoted above, if $u$ resp. $u_\varepsilon$ are the solutions to the initial problem~\eqref{eq:Nd} resp. the penalized problem~\eqref{eq:penalizedNd}, then, the penalization method convergences in the sense that there exists a constant $C$ independent of $\varepsilon$ such that for sufficiently small $\varepsilon>0$,
$$\|u_\varepsilon-u\|_{H^1(\mathcal U)}\le C\varepsilon.$$
Moreover, there exists a function $\overline{W}^0$ and a neighborhood of $\partial\Omega$, $\mathcal{V}(\partial\Omega)$, such that for sufficiently small
$\varepsilon>0$,
$$\|u_\varepsilon-\overline{W}^0\|_{H^1(\omega\setminus\mathcal{V}(\partial\Omega))}\le C\varepsilon.$$
\end{theorem}

 The previous theorem follows from the following theorem describing more precisely the boundary layer.

 \begin{theorem}[Description of the boundary layer]\label{th:convergenceBL}

Assuming the regularity hypotheses quoted above,
 there exists three functions $V^0$, $V^1$ and $V^2$ defined in $\mathcal{U}$, and three functions $W^0$, $W^1$ and $W^2$ defined in $\mathcal{\omega}\times\RR^+$, which are sufficiently smooth, such that
$$
u_\varepsilon(x)=
\begin{dcases}
V^0(x)+\varepsilon V^1(x)+\varepsilon^2 V^2(x)+\varepsilon v_\varepsilon^r(x)& \text{for $x\in\mathcal U$}\\
W^0\Bigl(x,\frac{\varphi(x)}\varepsilon\Bigr)+\varepsilon W^1\Bigl(x,\frac{\varphi(x)}\varepsilon\Bigr)+\varepsilon^2 W^2\Bigl(x,\frac{\varphi(x)}\varepsilon\Bigr)+\varepsilon w_\varepsilon^r(x)& \text{for $x\in\mathcal \omega$}
\end{dcases}
$$
where $v_\varepsilon^r$ and $w_\varepsilon^r$ are bounded independently of $\varepsilon$ in $H^1$.
\end{theorem}

The proof of these two theorems is the subject of the following subsections.

\subsection{Determination of the profiles}

\input{sections/subsections/BLprofiles}

\subsection{Formal resolution of the profile equations}

\input{sections/subsections/BLformalresolution}

\subsection{Well-posedness of the profile equations}\label{sec:wellposednessNd}

\input{sections/subsections/BLwellposedness}

\subsection{Convergence of the asymptotic expansion}\label{sec:convergenceNd}

\input{sections/subsections/BLconvergence}

\subsubsection{Construction of suitable supersolutions}

\input{sections/subsections/BLsupersolutions}

\paragraph{Link to the dual problem}

\input{sections/subsections/BLlinkdual}

\paragraph{Construction procedure}

\input{sections/subsections/BLgeneral}

%% file: sections/subsections/BLprofiles.tex
For a function $(x,z) \in \omega \times \RR^+ \mapsto W(x,z)$, we denote the derivatives with respect to the space variable $x$ as $\nabla W$, $\mathrm{div}W$, $\Delta W,\ldots$ and the derivatives with respect to the boundary layer variable $z$ as $W_z$, $W_{zz},\ldots$. Thus the function $x\mapsto w(x)=W\bigl(x,\frac{\varphi(x)}\varepsilon\bigr)$ satisfies
\begin{equation}\label{eq:derivativesNd}
    \begin{dcases}
    \nabla w=\nabla W+\frac{1}\varepsilon W_z\nabla\varphi\\ 
    \Delta w=\Delta W+\frac{2}\varepsilon  \nabla W_z\cdot\nabla\varphi+\frac1\epsilon \Delta\varphi W_z+\frac{1}{\varepsilon^2} |\nabla\varphi|^2W_{zz},
    \end{dcases}
\end{equation}
where the derivatives of $w$ are evaluated at $x$ while the derivatives of $W$ are evaluated at $\bigl(x,\frac{\varphi(x)}{\varepsilon}\bigr)$. Given the regularity assumptions on $\Omega$, we note that on $\partial\Omega$, we have $\nabla\varphi = \mu = -\nu_\Omega$, where $\mu$ is  the inward unit normal of $\Omega$. We also have $|\nabla\varphi|=1$ in $\omega_1$, see~\cite{BookBoyerFabrie}.

In the following, we introduce formally the expressions~\eqref{eq:penalizedNdansatz} in the system~\eqref{eq:wNd}--\eqref{eq:dirichletNd} and identify the terms corresponding to each power of $\varepsilon$. To simplify the formal calculations, we assume in the following $\theta(x)=1$ everywhere in~\eqref{eq:Ndboundarylayer} and an exponential decrease of $\partial_z^k \widetilde{W}^i(x,z)$ with respect to the variable $z$. Also, we assume $|\nabla\varphi|=1$ everywhere since the boundary layer terms will be localised near the boundary. All this will be made rigorous afterwards. The goals here is to construct formally the asymptotic expansion of the solution.

\subsubsection{Asymptotic expansion of Eq.~\eqref{eq:wNd} inside the obstacle $\omega$}

\begin{itemize}
    \item Order $\varepsilon^{-2}$: identifying the terms corresponding to the power $\varepsilon^{-2}$ leads to 
    $$-W^0_{zz}+W^0_z (\nabla\varphi\cdot n)=0.$$
    Using the decomposition~\eqref{eq:Ndboundarylayer} one obtains
    \begin{equation}\label{eq:Wtilde0Nd}
        -\widetilde{W}_{zz}^0+\widetilde{W}_z^0(\nabla\varphi\cdot n)=0 \quad \text{in $\omega\times\RR^+$}.
    \end{equation}
    
    \item Order $\varepsilon^{-1}$:
    $$-W^1_{zz}-\Delta\varphi W^0_z-2 \nabla W^0_z\cdot\nabla\varphi+\nabla W^0\cdot n+W^1_z \nabla\varphi\cdot n+\alpha W^0 -g=0,$$
    which,  using hypothesis~\eqref{eq:Ndboundarylayer} and taking the limit when $z\to +\infty$, leads to
    \begin{equation}\label{eq:Wbar0Nd}
        \nabla \overline{W}^0\cdot n+\alpha\overline{W}^0-g=0 \quad \text{in $\omega$}
    \end{equation}
    By difference, we obtain
        \begin{equation}\label{eq:Wtilde1Nd}
        -\widetilde{W}_{zz}^1-\Delta\varphi\widetilde{W}_{z}^0 -2 \nabla \widetilde{W}^0_z\cdot\nabla\varphi+\widetilde{W}^0\cdot n+\widetilde{W}_z^1 \nabla\varphi\cdot n+\alpha \widetilde{W}^0=0 \quad \text{in $\omega\times\RR^+$}.
    \end{equation}
    \item Order $\varepsilon^{0}$:
    $$-W^2_{zz}-\Delta\varphi W^1_{z} -2 \nabla W^1_z\cdot \nabla\varphi-\Delta W^0+W^0+\nabla W^1\cdot n+W^2_z \nabla\varphi\cdot n+\alpha W^1=0,$$
    and again using the same reasoning as above,  using hypothesis~\eqref{eq:Ndboundarylayer} and $z\to +\infty$, we obtain
    \begin{equation}\label{eq:Wbar1Nd}
        -\Delta\overline{W}^0+\overline{W}^0+\nabla\overline{W}^1\cdot n+\alpha \overline{W}^1=0 \quad \text{in $\omega$},
    \end{equation}
    and by difference,
        \begin{equation}\label{eq:Wtilde2Nd}
        -\widetilde{W}_{zz}^2-\Delta\varphi \widetilde{W}_{z}^1-2 \nabla\widetilde{W}^1_z\cdot\nabla\varphi-\Delta\widetilde{W}^0+\widetilde{W}^0+\nabla\widetilde{W}^1\cdot n+\widetilde{W}^2_z \nabla\varphi\cdot n+\alpha \widetilde{W}^1=0 \quad \text{in $\omega\times\RR^+$}.
    \end{equation}
    \item Order $\varepsilon$: from
    $$-W^3_{zz}-\Delta\varphi W^2_{z}-2 \nabla W^2_z\cdot\nabla\varphi-\Delta W^1+W^1+\nabla W^2\cdot n+W^3_z \nabla\varphi\cdot n+\alpha W^2=0,$$
    we deduce once again as before,
    \begin{equation}\label{eq:Wbar2Nd}
        -\Delta\overline{W}^1+\overline{W}^1+\nabla\overline{W}^2\cdot n+\alpha \overline{W}^2=0 \quad \text{in $\omega$},
    \end{equation}
    and by difference,
        \begin{equation}\label{eq:Wtilde3Nd}
        -\widetilde{W}_{zz}^3-\Delta\varphi \widetilde{W}_{z}^2-2 \nabla\widetilde{W}^2_z\cdot\nabla\varphi-\Delta\widetilde{W}^1+\widetilde{W}^1+\nabla\widetilde{W}^2\cdot n+\widetilde{W}^3_z \nabla\varphi\cdot n+\alpha \widetilde{W}^2=0 \quad \text{in $\omega\times\RR^+$}.
    \end{equation}
\end{itemize}

\subsubsection{Asymptotic expansion of Eq.~\eqref{eq:vNd} inside the fluid domain $\mathcal{U}$}

\begin{itemize}
    \item Order $\varepsilon^{0}$:
    \begin{equation}\label{eq:V0Nd}
        -\Delta V^0+V^0=f \quad \text{in $\mathcal{U}$}.
    \end{equation}
    \item Order $\varepsilon^{1}$:
    \begin{equation}\label{eq:V1Nd}
        -\Delta V^1+V^1=0 \quad \text{in $\mathcal{U}$}.
    \end{equation}
    \item Order $\varepsilon^{2}$:
    \begin{equation}\label{eq:V2Nd}
        -\Delta V^2+V^2=0 \quad \text{in $\mathcal{U}$}.
    \end{equation}
\end{itemize}

\subsubsection{Asymptotic expansion of Eq.~\eqref{eq:transmission1Nd} and~\eqref{eq:transmission2Nd} (transmission conditions on $\partial\mathcal{U}$)}

We obtain simply, recalling the exponential decrease of the boundary layer terms
\begin{equation}\label{eq:expansiontransmission1Nd}
    V^0=\overline{W}^0, \ V^1=\overline{W}^1, \ V^2=\overline{W}^2 \quad \text{on $\partial\mathcal{U}$}.
\end{equation}
and
\begin{equation}\label{eq:expansiontransmission2Nd}
    \frac{\partial V^0}{\partial\nu}=\frac{\partial\overline{W}^0}{\partial\nu}, \ \frac{\partial V^1}{\partial\nu}=\frac{\partial\overline{W}^1}{\partial\nu}, \ \frac{\partial V^2}{\partial\nu}=\frac{\partial \overline{W}^2}{\partial\nu} \quad \text{on $\partial\mathcal{U}$}.
\end{equation}

\subsubsection{Asymptotic expansion of Eq.~\eqref{eq:dirichletNd} (Dirichlet boundary conditions on $\partial\Omega$)}

\begin{equation}\label{eq:expansiondirichletNd}
\begin{dcases}
    \widetilde{W}^0(x,0)+\overline{W}^0(x)=0 &\\
    \widetilde{W}^1(x,0)+\overline{W}^1(x)=0 &  \text{on $\partial\Omega\times\{z=0\}$}\\
    \widetilde{W}^2(x,0)+\overline{W}^2(x)=0. &
    \end{dcases}
\end{equation}

%% file: sections/subsections/BLformalresolution.tex
\subsubsection{Determination of $V^0$ and $W^0$}
From~\eqref{eq:V0Nd}, \eqref{eq:Wbar0Nd}, \eqref{eq:expansiontransmission1Nd}, \eqref{eq:expansiontransmission2Nd} and~\eqref{eq:expansiondirichletNd}, we obtain, recalling $n=\nu$ and $g=\tilde{g}$ on $\partial\mathcal{U}$ , that $V^0$ is solution to
\begin{equation}\label{eq:solV0Nd}
    \begin{dcases}
        -\Delta V^0+V^0=f & \text{in $\mathcal{U}$}\\
        \frac{\partial V^0}{\partial\nu}+\alpha V^0=\frac{\partial\overline{W}^0}{\partial\nu}+\alpha \overline{W}^0=\tilde{g} & \text{on $\partial\mathcal{U}$},
    \end{dcases}
\end{equation}
i.e. $V^0=u$ the solution of the initial problem~\eqref{eq:Nd}, and that $\overline{W}^0$ is solution to
\begin{equation}\label{eq:solWbar0Nd}
    \begin{dcases}
        \nabla\overline{W}^0\cdot n+\alpha \overline{W}^0=g & \text{in $\omega$}\\
        \overline{W}^0=V^0 & \text{on $\partial\mathcal{U}$}\\
        \overline{W}^0=0 & \text{on $\Gamma_-$}.
    \end{dcases}
\end{equation}

We recall that in our case $\Gamma_-=\emptyset$. The well-posedness of this equation will be discussed later.

From~\eqref{eq:Wtilde0Nd}, the assumption $-c(x)=n\cdot\nabla\varphi<0$ in $\mathcal{V}(\partial\Omega)$ and the fact that $\widetilde{W}^0$ tends to $0$ as $z\to+\infty$ , we obtain
\begin{equation*}
\widetilde{W}^0(x,z)=w^0(x) e^{-c(x)z} \quad \text{in $\omega\times \RR^+$},
\end{equation*}
where $w^0$ is the value of $\widetilde{W}^0$ at $z=0$. Using~\eqref{eq:expansiondirichletNd}, we have $\widetilde{W}^0(x,0)=-\overline{W}^0(x)$ on $\partial\Omega$. We extend this boundary condition to all $\omega$ and choose $\widetilde{W}^0(x,0)=-\overline{W}^0(0)$ in $\omega$. This implies $w^0(x)=-\overline{W}^0(x)$ and thus
\begin{equation}\label{eq:solWtilde0Nd}
\widetilde{W}^0(x,z)=-\overline{W}^0(x) e^{-c(x)z} \quad \text{in $\omega\times \RR^+$}.
\end{equation}

\subsubsection{Determination of $V^1$ and $W^1$}
Similarly, from~\eqref{eq:V1Nd}, \eqref{eq:Wbar1Nd}, \eqref{eq:expansiontransmission1Nd}, and~\eqref{eq:expansiontransmission2Nd}, we obtain that $V^1$ is solution to
\begin{equation}\label{eq:solV1Nd}
    \begin{dcases}
        -\Delta V^1+V^1=0 & \text{in $\mathcal{U}$}\\
        \frac{\partial V^1}{\partial\nu}+\alpha V^1=\frac{\partial\overline{W}^1}{\partial\nu}+\alpha \overline{W}^1=\Delta \overline{W}^0-\overline{W}^0 & \text{on $\partial\mathcal{U}$},
    \end{dcases}
\end{equation}
and that $\overline{W}^1$ is solution to
\begin{equation}\label{eq:solWbar1Nd}
    \begin{dcases}
        \nabla\overline{W}^1\cdot n+\alpha \overline{W}^1=\Delta \overline{W}^0-\overline{W}^0 & \text{in $\omega$}\\
        \overline{W}^1=V^1 & \text{on $\partial\mathcal{U}$}.
    \end{dcases}
\end{equation}
From~\eqref{eq:Wtilde1Nd}, we have
\begin{equation*}
    -\widetilde{W}_{zz}^1-c(x)\widetilde{W}_z^1=\Delta\varphi \widetilde{W}_{z}^0+2 \nabla\widetilde{W}^0_z\cdot\nabla\varphi-\nabla\widetilde{W}^0\cdot n-\alpha \widetilde{W}^0=F(x)e^{-c(x)z}\quad \text{in $\omega\times\RR^+$},
\end{equation*}
using the obtained expression of $\widetilde{W}^0$~\eqref{eq:solWtilde0Nd}. The solution which tends to $0$ as $z\to+\infty$ is given by
\begin{equation*}
    \widetilde{W}^1(x,z)=w^1(x) e^{-c(x)z}+\frac{F(x)}{c(x)} z e^{-c(x)z} \quad \text{in $\omega\times\RR^+$},
\end{equation*}
where $w^1$ is the value of $\widetilde{W}^1$ at $z=0$, to be determined.
Again, using~\eqref{eq:expansiondirichletNd}, we have $\widetilde{W}^1(x,0)=-\overline{W}^1(x)$ on $\partial\Omega$. We extend this boundary condition to all $\omega$ by setting $\widetilde{W}^1(x,0)=-\overline{W}^1(0)$ in $\omega$. This choice leads to $w^1(x)=-\overline{W}^1(x)$ and thus
\begin{equation}\label{eq:solWtilde1Nd}
\widetilde{W}^1(x,z)=-\overline{W}^1(x) e^{-c(x)z}+\frac{F(x)}{c(x)} z\, e^{-c(x)z} \quad \text{in $\omega\times \RR^+$}.
\end{equation}
We recall that $c(x)>0$ in $\mathcal{V}(\partial\Omega)$.

\subsubsection{Determination of $V^2$ and $W^2$}
Using the same reasoning as before, from~\eqref{eq:V2Nd}, \eqref{eq:Wbar2Nd}, \eqref{eq:expansiontransmission1Nd} and~\eqref{eq:expansiontransmission2Nd}, we obtain that $V^2$ satisfies
\begin{equation}\label{eq:solV2Nd}
    \begin{dcases}
        -\Delta V^2+V^2=0 & \text{in $\mathcal{U}$}\\
        \frac{\partial V^2}{\partial\nu}+\alpha V^2=\frac{\partial\overline{W}^2}{\partial\nu}+\alpha \overline{W}^2=\Delta\overline{W}^1-\overline{W}^1 & \text{on $\partial\mathcal{U}$},
    \end{dcases}
\end{equation}
and that $\overline{W}^2$ satisfies
\begin{equation}\label{eq:solWbar2Nd}
    \begin{dcases}
        \nabla\overline{W}^2\cdot n+\alpha \overline{W}^2=\Delta\overline{W}^1-\overline{W}^1 & \text{in $\omega$}\\
        \overline{W}^2=V^2 & \text{on $\partial\mathcal{U}$}.
    \end{dcases}
\end{equation}
From~\eqref{eq:Wtilde2Nd}, we have
\begin{align*}
    -\widetilde{W}_{zz}^2-c(x)\widetilde{W}^2_z&=\Delta\varphi\widetilde{W}^1_z +2 \nabla\widetilde{W}^1_z\cdot\nabla\varphi+\Delta\widetilde{W}^0-\widetilde{W}^0-\nabla\widetilde{W}^1\cdot n-\alpha \widetilde{W}^1 \quad \text{in $\omega\times\RR^+$}\\
    &=G(x)\, e^{-c(x)z} + H(x)\, z\, e^{-c(x)z}
\end{align*}
using the obtained expressions of $\widetilde{W}^0$ \eqref{eq:solWtilde0Nd} and $\widetilde{W}^1$ \eqref{eq:solWtilde1Nd}. The solution which tends to $0$ as $z\to+\infty$ is given by
\begin{equation*}
    \widetilde{W}^2(x,z)=w^2(x) e^{-c(x)z}+\biggl(\frac{G(x)}{c(x)}+\frac{H(x)}{c^2(x)}\biggr) z\, e^{-c(x)z} + \frac{H(x)}{2c(x)} z^2\, e^{-c(x)z} \quad \text{in $\omega\times\RR^+$},
\end{equation*}
where $w^2$ is the value of $\widetilde{W}^2$ at $z=0$, to be determined.
Once again, using~\eqref{eq:expansiondirichletNd}, we have $\widetilde{W}^2(x,0)=-\overline{W}^2(x)$ on $\partial\Omega$. We extend this boundary condition to all $\omega$ by setting $\widetilde{W}^2(x,0)=-\overline{W}^2(0)$ in $\omega$. This choice leads to $w^2(x)=-\overline{W}^2(x)$ and thus
\begin{align}\label{eq:solWtilde2Nd}
\widetilde{W}^2(x,z)=-\overline{W}^2(x) e^{-c(x)z}+\biggl(\frac{G(x)}{c(x)}+\frac{H(x)}{c^2(x)}\biggr) z\, e^{-c(x)z} + \frac{H(x)}{2c(x)} z^2\, e^{-c(x)z} \quad \text{in $\omega\times \RR^+$}.
\end{align}

%% file: sections/subsections/BLwellposedness.tex

In this section, we show that the formal equations obtained in the previous section are indeed well-posed with suitable regularity.

\begin{itemize}
\item {Regularity of $V^0$:}\\
Since $V_0=u$ solution of the initial problem~\eqref{eq:Nd}, we obtain from the elliptic regularity~\ref{th:ellipticregularity} using the asumptions on $\mathcal U$ of class $C^{k+8,1}$, $f\in H^{k+7}(\mathcal U)$, $\tilde{g}\in H^{k+7+1/2}(\partial\mathcal U)$ and Sobolev injections in dimension 2 that $V^0\in H^{k+9}(\mathcal U)\hookrightarrow C^{k+7}(\overline{\mathcal{U}})$.

\item {Regularity of $\overline{W}^0$:}\\
$\overline{W}^0$ is solution to~\eqref{eq:solWbar0Nd}, which is an advection-reaction equation. The existence, uniqueness and regularity result for this type of equations is given by Theorem~\ref{th:hyperbolicregularitybis}. Since $\widetilde{g}\in H^{k+7+1/2}(\partial \mathcal U)$, we can lift by $g\in H^{k+8}(\omega)\hookrightarrow C^{k+6}(\overline{\mathcal{\omega}})$. Together with $V^0\in C^{k+7}(\partial\mathcal U)$ and the asumptions on $\psi\in C^{k+7}(\overline{\omega})$, we obtain $\overline{W}^0\in C^{k+6}(\overline{\omega})\hookrightarrow H^{k+6}(\omega)$.

\item {Regularity of $\widetilde{W}^0$:}\\
We recall that we found~\eqref{eq:solWtilde0Nd}
\begin{equation*}
\widetilde{W}^0(x,z)=-\overline{W}^0(x) e^{-c(x)z} \quad \text{in $\omega_1\times \RR^+$},
\end{equation*}
where $c(x)=n(x)\cdot(-\nabla\varphi(x))>0$ in $\omega_1$, such that the regularity established for $\overline{W}^0$ and the regularity of $n=\nabla\psi\in C^{k+6}(\overline{\omega})$ and $\nabla\varphi\in C^{k+7}(\overline{\omega_1})$ show that $\widetilde{W}^0\in C^{k+6}(\overline{\omega_1}\times \RR^+)$ and $C^\infty(\RR^+)$ in the $z$ variable.

\item{Regularity of $V^1$:}\\
$V^1$ is given by~\eqref{eq:solV1Nd} which is a reaction-diffusion equation. Using again the elliptic regularity theorem ~\ref{th:ellipticregularity} and $\Delta \overline{W}^0-\overline{W}^0\in C^{k+4}(\overline{\omega})\hookrightarrow H^{k+4}(\overline{\omega})$, thus $\Delta \overline{W}^0-\overline{W}^0\in H^{k+3+1/2}(\partial\mathcal U)$, we obtain $V^1\in H^{k+5}(\mathcal U)\hookrightarrow C^{k+3}(\overline{\mathcal U})$.

\item {Regularity of $\overline{W}^1$:}\\
$\overline{W}^1$ is solution to~\eqref{eq:solWbar1Nd}. Since $\Delta \overline{W}^0-\overline{W}^0\in C^{k+4}(\overline{\omega})$ and $V^1\in C^{k+3}(\partial\mathcal U)$, we obtain similarly $\overline{W}^1\in C^{k+3}(\overline{\omega})\hookrightarrow H^{k+3}(\omega)$.

\item{Regularity of $\widetilde{W}^1$:}\\
$\widetilde{W}^1$ is given by~\eqref{eq:solWtilde1Nd} in $\omega_1\times\RR^+$, so that the regularity established for $\overline{W}^1\in C^{k+3}(\overline{\omega})$ and $\widetilde{W}^0\in C^{k+6}(\overline{\omega_1}\times\RR^+)$, and the regularity of $n=\nabla\psi\in C^{k+6}(\overline{\omega})$ and $\varphi\in C^{k+8}(\overline{\omega_1})$ show that $\widetilde{W}^1\in C^{k+3}(\overline{\omega_1}\times \RR^+)$ and $C^\infty(\RR^+)$ in the $z$ variable.

\item{Regularity of $V^2$:}\\
$V^2$ is given by~\eqref{eq:solV2Nd}. Again, using the elliptic regularity~\ref{th:ellipticregularity} and $\Delta \overline{W}^1-\overline{W}^1\in C^{k+1}(\overline{\omega})\hookrightarrow H^{k+2}(\overline{\omega})$, thus $\Delta \overline{W}^1-\overline{W}^1\in H^{k+1/2}(\partial\mathcal U)$, we obtain $V^2\in H^{k+2}(\mathcal U)\hookrightarrow C^{k}(\overline{\mathcal U})$.

\item {Regularity of $\overline{W}^2$:}\\
$\overline{W}^2$ is solution to~\eqref{eq:solWbar2Nd}. Since $\Delta \overline{W}^1-\overline{W}^1\in C^{k+1}(\overline{\omega})$ and $V^1\in C^{k}(\partial\mathcal U)$, we obtain similarly $\overline{W}^2\in C^{k}(\overline{\omega})\hookrightarrow H^{k}(\omega)$.

\item{Regularity of $\widetilde{W}^2$:}\\
$\widetilde{W}^2$ is given by~\eqref{eq:solWtilde2Nd} in $\omega_1\times\RR^+$, so that the regularity established for $\overline{W}^2\in C^{k}(\overline{\omega})$, $\widetilde{W}^1\in C^{k+3}(\overline{\omega_1}\times\RR^+)$ and $\widetilde{W}^0\in C^{k+6}(\overline{\omega_1}\times\RR^+)$, and the regularity of $n=\nabla\psi\in C^{k+6}(\overline{\omega})$ and $\varphi\in C^{k+8}(\overline{\omega_1})$ show that $\widetilde{W}^2\in C^{k}(\overline{\omega_1}\times \RR^+)$ and $C^\infty(\RR^+)$ in the $z$ variable.

\end{itemize}

%% file: sections/subsections/BLconvergence.tex

We search for a solution of the penalized problem~\eqref{eq:penalizedNd} in the following form
\begin{equation}\label{eq:expansionNd}
    \begin{dcases}
    w_\varepsilon(x)=\theta(x)\widetilde{W}^0\Bigl(x,\frac{\varphi(x)}\varepsilon\Bigr)+\varepsilon \theta(x)\widetilde{W}^1\Bigl(x,\frac{\varphi(x)}\varepsilon\Bigr)+\varepsilon^2\theta(x)\widetilde{W}^2\Bigl(x,\frac{\varphi(x)}\varepsilon\Bigr)&\\
    \qquad \qquad + \overline{W}^0(x)+\varepsilon \overline{W}^1(x)+\varepsilon^2 \overline{W}^2(x)+ \varepsilon w_\varepsilon^r(x) & \text{in $\omega$}\\
    v_\varepsilon(x)={V}^0(x)+\varepsilon {V}^1(x)+\varepsilon^2 {V}^2(x)+ \varepsilon v_\varepsilon^r(x) & \text{in $\mathcal{U}$},
    \end{dcases}
\end{equation}
were $\widetilde{W}^i$, $\overline{W}^i$ and ${V}^i$ are the profile terms constructed previously~\eqref{eq:solV0Nd}--\eqref{eq:solWtilde2Nd}, and $w_\varepsilon^r$ and $v_\varepsilon^r$ are the remainder terms that we will estimate in the following.

We use the following notations:
\begin{equation}\label{eq:approximatedsolutionNd}
\begin{dcases}
    W_\mathrm{app}=\theta\widetilde{w}^0+\varepsilon \theta\widetilde{w}^1+\varepsilon^2\theta\widetilde{w}^2+\overline{W}^0+\varepsilon \overline{W}^1+\varepsilon^2 \overline{W}^2\\
    V_\mathrm{app}=V^0+\varepsilon V^1+\varepsilon^2 V^2,
\end{dcases}
\end{equation}
where $\widetilde{w}^i(x)=\widetilde{W}^i\bigl(x,\frac{\varphi(x)}\varepsilon\bigr)$, so that
\begin{equation*}
\begin{dcases}
w_\varepsilon= W_\mathrm{app}+\varepsilon w_\varepsilon^r & \text{in $\omega$}\\
v_\varepsilon=V_\mathrm{app}+\varepsilon v_\varepsilon^r& \text{in $\mathcal{U}$}.
\end{dcases}
\end{equation*}

The aim of the following subsections is to show that the remainders $w_\varepsilon^r$ and $v_\varepsilon^r$ are bounded in $H^1$ independently of $\varepsilon$, from which we will conclude that
\begin{equation}\label{eq:convergencefluidNd}
\|v_\varepsilon-V^0\|_{H^1(\mathcal{U})}=O(\varepsilon),
\end{equation}
that is the convergence of the solution of the penalized problem~\eqref{eq:penalizedNd} towards the solution of the initial problem~\eqref{eq:Nd} inside the fluid domain.
On the other hand, we will obtain that
\begin{equation}\label{eq:boundarylayerobstacleNd}
\|w_\varepsilon-W_\mathrm{app}\|_{H^1(\omega)}=O(\varepsilon),
\end{equation}
that is the presence of a boundary layer at the boundary $\partial\Omega$. In particular far from the boundary $\partial\Omega$ we have
\begin{equation}\label{eq:convergeneobstacleNd}
\|w_\varepsilon-\overline{W}^0\|_{H^1(\omega\setminus\mathcal{V}(\partial\Omega))}=O(\varepsilon),
\end{equation}
and $\overline{W}^0$ is the limit solution in the complementary domain.

\subsubsection{Equations of the remainders}
Using the equations satisfied by $w_\varepsilon$ and $v_\varepsilon$, the remainders satisfy the following system
    \begin{subnumcases}{}
    -\Delta w^r_\varepsilon+w^r_\varepsilon+\frac1\varepsilon (\nabla w^r_\varepsilon\cdot n+\alpha w^r_\varepsilon)= R^\varepsilon_\mathrm{obst}  & \text{in $\omega$}\label{eq:wremainderNd}\\ 
    -\Delta v^r_\varepsilon+v^r_\varepsilon= R^\varepsilon_\mathrm{flu} & \text{in $\mathcal{U}$}\label{eq:vremainderNd}\\
    w^r_\varepsilon=v^r_\varepsilon & \text{on $\partial\mathcal{U}$}\label{eq:transmission1remainderNd}\\
    \frac{\partial w^r_\varepsilon}{\partial \nu}=\frac{\partial v^r_\varepsilon}{\partial\nu} & \text{on $\partial\mathcal{U}$}\label{eq:transmission2remainderNd}\\
    w^r_\varepsilon=R^\varepsilon_\mathrm{boundary}& \text{ on $\partial\Omega$}\label{eq:dirichletremainderNd}
    \end{subnumcases}
where
\begin{align*}
    R^\varepsilon_\mathrm{obst}&=\frac1\varepsilon \Delta W_\mathrm{app}-\frac1\varepsilon W_\mathrm{app}-\frac1{\varepsilon^2} \nabla W_\mathrm{app}\cdot n-\frac\alpha{\varepsilon^2} W_\mathrm{app}+\frac{g}{\varepsilon^2}\\
    &=\frac1\varepsilon \Delta W_\mathrm{app}-\frac1{\varepsilon^2} \nabla W_\mathrm{app}\cdot n-\Bigl(\frac1\varepsilon+\frac\alpha{\varepsilon^2}\Bigr) W_\mathrm{app}+\frac{g}{\varepsilon^2}\\
    R^\varepsilon_\mathrm{flu}&=\frac1\varepsilon (f+\Delta V_\mathrm{app}-V_\mathrm{app})\\
    &=\frac1\varepsilon (f+\Delta {V^0}+\varepsilon \Delta {V^1}+\varepsilon^2 \Delta {V^2}-V^0-\varepsilon V^1- \varepsilon^2 V^2)=0\\
    R^\varepsilon_\mathrm{boundary}&=-\frac1\varepsilon W_\mathrm{app}\\
    &=-\frac1\varepsilon (\widetilde{W}^0(x,0)+\varepsilon\widetilde{W}^1(x,0)+\varepsilon^2 \widetilde{W}^2(x,0) + \overline{W}^0(x)+\varepsilon \overline{W}^1(x)+\varepsilon^2 \overline{W}^2(x) )=0.
\end{align*}

To estimate the remainders, we need to estimate the right-hand sides of the equations, namely $R^\varepsilon_\mathrm{obst}$.

Using~\eqref{eq:approximatedsolutionNd}, one has 
\begin{align*}
    \nabla W_\mathrm{app}&=\nabla\theta\,\widetilde{w}^0+\theta\,\nabla{\widetilde{w}^0}+\varepsilon \nabla\theta\,{\widetilde{w}}^1+\varepsilon \theta\,\nabla{\widetilde{w}^1}+\varepsilon^2\nabla\theta\,\widetilde{w}^2+\varepsilon^2\theta\,\nabla{\widetilde{w}^2}+\nabla{\overline{W}^0}+\varepsilon \nabla{\overline{W}^1}+\varepsilon^2 \nabla{\overline{W}^2}\\
        \Delta W_\mathrm{app}&=\Delta\theta\,\widetilde{w}^0+2\nabla\theta\cdot\nabla{\widetilde{w}^0}+\theta\,\Delta{\widetilde{w}^0}+\varepsilon \Delta\theta\,\widetilde{w}^1+2\varepsilon \nabla\theta\cdot\nabla{\widetilde{w}^1}+\varepsilon \theta\,\Delta{\widetilde{w}^1}+\varepsilon^2\Delta\theta\,\widetilde{w}^2+2\varepsilon^2\nabla\theta\cdot\nabla{\widetilde{w}^2}\\
        &\quad +\varepsilon^2\theta\,\Delta{\widetilde{w}^2}+\Delta{\overline{W}^0}+\varepsilon \Delta{\overline{W}^1}+\varepsilon^2 \Delta{\overline{W}^2}.
\end{align*}
Using~\eqref{eq:derivativesNd} and the fact that $|\nabla\varphi|=1$ inside the support of $\theta$, we obtain
\begin{align*}
    R^\varepsilon_\mathrm{obst}&=\frac1\varepsilon \Bigl(\Delta\theta\,\widetilde{W}^0+2\nabla\theta\cdot\nabla \widetilde{W}^0+\frac{2}\varepsilon\nabla\theta\cdot\nabla\varphi\,\widetilde{W}^0_z\\
    &\quad+\theta\,\Delta\widetilde{W}^0 + \frac{2\theta}{\varepsilon}\nabla\widetilde{W}^0_{z}\cdot\nabla\varphi+\frac{\theta}{\varepsilon}\Delta\varphi\,\widetilde{W}^0_{z}+\frac{\theta}{\varepsilon^2}\widetilde{W}^0_{zz}\\
    &\quad+\varepsilon \Delta\theta\,\widetilde{W}^1+2\varepsilon \nabla\theta\cdot\nabla \widetilde{W}^1+2\varepsilon \frac{1}\varepsilon \nabla\theta\cdot\nabla\varphi\,\widetilde{W}^1_z\\
    &\quad+\varepsilon \theta\,\Delta\widetilde{W}^1 + \varepsilon \frac{2\theta}{\varepsilon}\nabla\widetilde{W}^1_{z}\cdot\nabla\varphi+\varepsilon\frac{\theta}{\varepsilon}\Delta\varphi\,\widetilde{W}^1_{z}+\varepsilon\frac{\theta}{\varepsilon^2}\widetilde{W}^1_{zz}\\
    &\quad+\varepsilon^2 \Delta\theta\,\widetilde{W}^2+2\varepsilon^2\nabla\theta\cdot\nabla \widetilde{W}^2+2\varepsilon^2\frac{1}\varepsilon\nabla\theta\cdot\nabla\varphi\,\widetilde{W}^2_z\\
    &\quad+\varepsilon^2 \theta\,\Delta\widetilde{W}^2 + \varepsilon^2 \frac{2\theta}{\varepsilon}\nabla\widetilde{W}^2_{z}\cdot\nabla\varphi+\varepsilon^2\frac{\theta}{\varepsilon}\Delta\varphi\,\widetilde{W}^2_{z}+\varepsilon^2\frac{\theta}{\varepsilon^2}\widetilde{W}^2_{zz}\\
    &\quad +\Delta{\overline{W}^0}+\varepsilon \Delta{\overline{W}^1}+\varepsilon^2 \Delta{\overline{W}^2}\Bigr)\\
    &\quad-\frac1{\varepsilon^2} \Bigl(\nabla\theta\cdot n\,\widetilde{W}^0+\theta\, \nabla\widetilde{W}^0\cdot n+\frac{\theta}\varepsilon\nabla\varphi\cdot n\,\widetilde{W}^0_z\\
    &\quad+\varepsilon\nabla\theta\cdot n\,\widetilde{W}^1+\varepsilon\theta\, \nabla\widetilde{W}^1\cdot n+\varepsilon\frac{\theta}\varepsilon\nabla\varphi\cdot n\,\widetilde{W}^1_z\\
    &\quad +\varepsilon^2\nabla\theta\cdot n\,\widetilde{W}^2+\varepsilon^2\theta\, \nabla\widetilde{W}^2\cdot n+\varepsilon^2\frac{\theta}\varepsilon\nabla\varphi\cdot n\,\widetilde{W}^2_z\\
    &\quad+\nabla\overline{W}^0\cdot n+\varepsilon \nabla\overline{W}^1\cdot n+\varepsilon^2 \nabla\overline{W}^2\cdot n\Bigr)\\
    &\quad-\Bigl(\frac1\varepsilon+\frac\alpha{\varepsilon^2}\Bigr) (\theta\,\widetilde{W}^0+\varepsilon \theta\,\widetilde{W}^1+\varepsilon^2\theta\,\widetilde{W}^2+\overline{W}^0+\varepsilon \overline{W}^1+\varepsilon^2 \overline{W}^2)+\frac{g}{\varepsilon^2}\\
    &=\frac1\varepsilon \Delta\theta\,\widetilde{W}^0+\frac2\varepsilon\nabla\theta\cdot\nabla \widetilde{W}^0+\frac{2}{\varepsilon^2}\nabla\theta\cdot\nabla\varphi\,\widetilde{W}^0_z\\
    &\quad+\frac\theta\varepsilon\,\Delta\widetilde{W}^0 + \frac{2\theta}{\varepsilon^2}\nabla\widetilde{W}^0_{z}\cdot\nabla\varphi+\frac{\theta}{\varepsilon^2}\Delta\varphi\,\widetilde{W}^0_{z}+\frac{\theta}{\varepsilon^3}\widetilde{W}^0_{zz}\\
    &\quad+\Delta\theta\,\widetilde{W}^1+2\nabla\theta\cdot\nabla \widetilde{W}^1+\frac{2}\varepsilon \nabla\theta\cdot\nabla\varphi\,\widetilde{W}^1_z\\
    &\quad+ \theta\,\Delta\widetilde{W}^1 + \frac{2\theta}{\varepsilon}\nabla\widetilde{W}^1_{z}\cdot\nabla\varphi+\frac{\theta}{\varepsilon}\Delta\varphi\,\widetilde{W}^1_{z}+\frac{\theta}{\varepsilon^2}\widetilde{W}^1_{zz}\\
    &\quad+\varepsilon\Delta\theta\,\widetilde{W}^2+2\varepsilon\nabla\theta\cdot\nabla \widetilde{W}^2+2\nabla\theta\cdot\nabla\varphi\,\widetilde{W}^2_z\\
    &\quad+\varepsilon \theta\,\Delta\widetilde{W}^2 +  2\theta\,\nabla\widetilde{W}^2_{z}\cdot\nabla\varphi+\theta\,\Delta\varphi\,\widetilde{W}^2_{z}+\frac{\theta}{\varepsilon}\widetilde{W}^2_{zz}\\
    &\quad +\frac1\varepsilon\Delta{\overline{W}^0}+ \Delta{\overline{W}^1}+\varepsilon \Delta{\overline{W}^2}\\
    &\quad-\frac1{\varepsilon^2} \nabla\theta\cdot n\,\widetilde{W}^0-\frac\theta{\varepsilon^2}\, \nabla\widetilde{W}^0\cdot n-\frac{\theta}{\varepsilon^3}\nabla\varphi\cdot n\,\widetilde{W}^0_z\\
    &\quad-\frac1\varepsilon\nabla\theta\cdot n\,\widetilde{W}^1-\frac\theta\varepsilon\, \nabla\widetilde{W}^1\cdot n-\frac{\theta}{\varepsilon^2}\nabla\varphi\cdot n\,\widetilde{W}^1_z\\
    &\quad -\nabla\theta\cdot n\,\widetilde{W}^2-\theta\, \nabla\widetilde{W}^2\cdot n-\frac{\theta}\varepsilon\nabla\varphi\cdot n\,\widetilde{W}^2_z\\
    &\quad-\frac1{\varepsilon^2}\nabla\overline{W}^0\cdot n-\frac1\varepsilon \nabla\overline{W}^1\cdot n-\nabla\overline{W}^2\cdot n\\
    &\quad -\frac\theta\varepsilon \widetilde{W}^0- \theta\widetilde{W}^1-\varepsilon\theta\widetilde{W}^2-\frac1\varepsilon\overline{W}^0- \overline{W}^1-\varepsilon \overline{W}^2\\ 
    &\quad-\frac{\alpha\theta}{\varepsilon^2} \widetilde{W}^0
    - \frac{\alpha\theta}\varepsilon\widetilde{W}^1-\alpha\theta\widetilde{W}^2-\frac\alpha{\varepsilon^2}\overline{W}^0-\frac\alpha{\varepsilon} \overline{W}^1-\alpha \overline{W}^2+\frac{g}{\varepsilon^2},
\end{align*}
where as usual $\widetilde{W}^i$, $\nabla\widetilde{W}^i$ and $\Delta\widetilde{W}^i$ are taken at $(x,\frac{\varphi(x)}{\varepsilon})$.

Using the equations satisfied by the profile terms $\widetilde{W}^i$ and $\overline{W}^i$, different terms simplify in the previous expression and it remains
\begin{align*}
    R^\varepsilon_\mathrm{obst}&=\frac1\varepsilon \Delta\theta\,\widetilde{W}^0+\frac2\varepsilon\nabla\theta\cdot\nabla \widetilde{W}^0+\frac{2}{\varepsilon^2}\nabla\theta\cdot\nabla\varphi\,\widetilde{W}^0_z\\
    &\quad+\Delta\theta\,\widetilde{W}^1+2\nabla\theta\cdot\nabla \widetilde{W}^1+\frac{2}\varepsilon \nabla\theta\cdot\nabla\varphi\,\widetilde{W}^1_z\\
    &\quad+ \theta\,\Delta\widetilde{W}^1+\varepsilon\Delta\theta\,\widetilde{W}^2+2\varepsilon\nabla\theta\cdot\nabla \widetilde{W}^2+2\nabla\theta\cdot\nabla\varphi\,\widetilde{W}^2_z\\
    &\quad+\varepsilon \theta\,\Delta\widetilde{W}^2 +  2\theta\,\nabla\widetilde{W}^2_{z}\cdot\nabla\varphi+\theta\,\Delta\varphi\,\widetilde{W}^2_{z}+\varepsilon \Delta{\overline{W}^2}\\
    &\quad-\frac1{\varepsilon^2} \nabla\theta\cdot n\,\widetilde{W}^0-\frac1\varepsilon\nabla\theta\cdot n\,\widetilde{W}^1-\nabla\theta\cdot n\,\widetilde{W}^2-\theta\, \nabla\widetilde{W}^2\cdot n\\    
    &\quad - \theta\,\widetilde{W}^1-\varepsilon\theta\,\widetilde{W}^2-\varepsilon \overline{W}^2-\alpha\theta\,\widetilde{W}^2.
\end{align*}

From the well-posedness section~\ref{sec:wellposednessNd}, since we assume $k\ge 2$ in Theorem~\ref{th:convergence}, all the terms of $R^\varepsilon_\mathrm{obst}$ are at least $C^0(\overline{\omega})$.

Since $\widetilde{W}^i(x,z)$ are of the form $\sum_{j=0}^2 A_j^i(x) z^j e^{-c(x)z}$ where $A_j^i(x)$ and $c(x)$ are sufficiently regular, and $z^j e^{-c(x)z}$ are bounded by $\bigl(\frac{j}{c(x)}\bigr)^je^{-j}$ with respect to $z$ in $\RR^+$ for $j\in\NN$ and $c(x)\ge \frac{\lambda_0}2>0$ in $\omega_1$, $\widetilde{W}^i\bigl(x,\frac{\varphi(x)}\varepsilon\bigr)$ are bounded in $\omega_1$ (which is a neighborhood of $\partial\Omega$) independently of $\varepsilon$, thus we have $\forall x\in\omega$

\begin{align*}
    \biggl|- \theta(x)\widetilde{W}^1\biggl(x,\frac{\varphi(x)}\varepsilon\biggr)-\alpha\theta(x)\widetilde{W}^2\biggl(x,\frac{\varphi(x)}\varepsilon\biggr)\biggr|
     \le C.
\end{align*}

Given the form of $\widetilde{W}^i(x,z)$, we observe that $\nabla\widetilde{W}^i(x,z)$, $\widetilde{W}^i_z(x,z)$, $\nabla\widetilde{W}^i(x,z)$ and $\Delta\widetilde{W}^i_z(x,z)$ are also of the form $\sum_{j=0}^2 B_j^i(x) z^j e^{-c(x)z}$ with suitable regularity, thus $\forall x\in\omega$
\begin{align*}
    &\biggl|\theta(x)\Delta\widetilde{W}^1\biggl(x,\frac{\varphi(x)}\varepsilon\biggr)+2\theta(x)\,\nabla\widetilde{W}^2_{z}\biggl(x,\frac{\varphi(x)}\varepsilon\biggr)\cdot\nabla\varphi(x)\\
    &\quad +\theta(x)\,\Delta\varphi(x)\,\widetilde{W}^2_{z}\biggl(x,\frac{\varphi(x)}\varepsilon\biggr)-\theta(x)\, \nabla\widetilde{W}^2\biggl(x,\frac{\varphi(x)}\varepsilon\biggr)\cdot n(x)\biggr|
     \le C,
\end{align*}
which estimates cannot be refined to $C\varepsilon$, and
\begin{align*}
    \biggl|\varepsilon\theta(x)\Delta\widetilde{W}^2\biggl(x,\frac{\varphi(x)}\varepsilon\biggr)-\varepsilon\theta(x)\widetilde{W}^2\biggl(x,\frac{\varphi(x)}\varepsilon\biggr)\biggr|
     \le C\varepsilon.
\end{align*}
Also,
\begin{align*}
    \Bigl|\varepsilon {\Delta\overline{W}^2}(x)-\varepsilon \overline{W}^2(x)\Bigr|
     \le C\varepsilon.
\end{align*}
The remaining terms are of the form $\frac{1}{\varepsilon^k} \Delta\theta(x)\,\widetilde{W}^i\bigl(x,\frac{\varphi(x)}\varepsilon\bigr)$ or $\frac{1}{\varepsilon^k} \nabla\theta(x)\cdot\nabla\widetilde{W}^i\bigl(x,\frac{\varphi(x)}\varepsilon\bigr)$ or $\frac{1}{\varepsilon^k} \nabla\theta(x)\cdot\nabla\varphi(x) \,\widetilde{W}^i_z\bigl(x,\frac{\varphi(x)}\varepsilon\bigr)$ or $\frac{1}{\varepsilon^k} \nabla\theta(x)\cdot n(x) \,\widetilde{W}^i\bigl(x,\frac{\varphi(x)}\varepsilon\bigr)$ for $k\in\ZZ$.  The terms under study are thus of the from 
$$\frac{1}{\varepsilon^k}\Delta\theta(x)\sum_{j=0}^2 A_j^i(x) \Bigl(\frac{\varphi(x)}\varepsilon\Bigr)^j e^{-c(x)\frac{\varphi(x)}\varepsilon}$$ or 
$$\frac{1}{\varepsilon^k} \nabla\theta(x)\cdot \sum_{j=0}^2 B_j^i(x) \Bigl(\frac{\varphi(x)}\varepsilon\Bigr)^j e^{-c(x)\frac{\varphi(x)}\varepsilon}$$ 
or
$$\frac{1}{\varepsilon^k} \nabla\theta(x)\cdot\nabla\varphi(x) \sum_{j=0}^2 B_j^i(x) \Bigl(\frac{\varphi(x)}\varepsilon\Bigr)^j e^{-c(x)\frac{\varphi(x)}\varepsilon}$$ 
or
$$\frac{1}{\varepsilon^k}\nabla\theta(x)\cdot n(x)\sum_{j=0}^2 A_j^i(x) \Bigl(\frac{\varphi(x)}\varepsilon\Bigr)^j e^{-c(x)\frac{\varphi(x)}\varepsilon}$$ 

Since $\theta\equiv 1$ in a neighborhood of $\partial\Omega$, these terms are zero if $\varphi(x)\le \delta_0$ for a well-chosen $\delta_0>0$. They are also zero outside of the $\omega_1$, the support of $\theta$. The $j^{th}$ terms are thus bounded by
\begin{equation*}
    C \frac{1}{\varepsilon^{j+k}} e^{-\frac{\lambda_0}2\frac{\delta_0}\varepsilon}\le C\varepsilon,
\end{equation*}
since for all $\ell\in\ZZ$, $\varepsilon^\ell e^{-\frac{\lambda_0}2\frac{\delta_0}\varepsilon}\to 0$ as $\varepsilon\to 0^+$.

In conclusion, we have the following estimate
\begin{equation}\label{eq:rhsestimateNd}
    \|R^\varepsilon_\mathrm{obst}\|_{L^2(\omega)}\le C.
\end{equation}

In the following, we will need a finer estimate: since $R^\varepsilon_\mathrm{obst}$ reduces to $\varepsilon \Delta{\overline{W}^2}-\varepsilon \overline{W}^2$ outside of the support of $\theta$, we have
\begin{align}
    \|R^\varepsilon_\mathrm{obst}\|_{L^2(\omega_1)}&\le C\label{eq:rhsestimatefiner1Nd}\\
    \|R^\varepsilon_\mathrm{obst}\|_{L^2(\omega\setminus\omega_1)}&\le C\varepsilon\label{eq:rhsestimatefiner2Nd},
\end{align}
where we recall $\mathrm{supp}(\theta)\subset \omega_1$.

\subsubsection{Estimate of the remainders}
The last step is to estimate the remainders using some energy estimates. 
Unfortunately, multiplying by the remainders and integrating by parts yields an interface term on $\partial\mathcal U$ of the form $\frac1\varepsilon (w^r_\varepsilon)^2$ (with the wrong sign) that is not easy to control. Inspired by the approaches used in boundary layer methods in the case of advection-diffusion problems~\cite{lions2006perturbations, gie2018singular, dalibard2018mathematical} where we multiply by test functions of the form $w^r_\varepsilon(x) e^{\pm x}$, we will rather multiply by weighted remainders as test functions where the weights have to be determined in order to get rid of the interface terms (between the fluid and the obstacle). More precisely, we multiply Eq.~\eqref{eq:vremainderNd} by $v^r_\varepsilon p_\varepsilon$ and we integrate over $\mathcal{U}$, this yields
\begin{align}
    -\int_\mathcal{U} \Delta{v^r_\varepsilon} {v^r_\varepsilon} p_\varepsilon+ \int_\mathcal{U} {v^r_\varepsilon}^2 p_\varepsilon=0\notag\\
    \int_\mathcal{U} \nabla{v^r_\varepsilon}\cdot \nabla({v^r_\varepsilon} p_\varepsilon)-\int_{\partial\mathcal{U}} \frac{\partial{v^r_\varepsilon}}{\partial\nu} {v^r_\varepsilon} p_\varepsilon+ \int_\mathcal{U} {v^r_\varepsilon}^2 p_\varepsilon=0\notag\\
    \int_\mathcal{U} \nabla{v^r_\varepsilon}\cdot\nabla {v^r_\varepsilon} p_\varepsilon+\int_\mathcal{U} \nabla v^r_\varepsilon\cdot \nabla p_\varepsilon {v^r_\varepsilon} -\int_{\partial\mathcal U} \frac{\partial v^r_\varepsilon}{\partial\nu} {v^r_\varepsilon} p_\varepsilon+ \int_{\mathcal U} {v^r_\varepsilon}^2 p_\varepsilon=0\notag\\
    \int_{\mathcal U} {(\nabla{v^r_\varepsilon})}^2 p_\varepsilon+\int_{\mathcal U} \nabla\Bigl(\frac{{v^r_\varepsilon}^2}{2}\Bigr)\cdot \nabla p_\varepsilon-\int_{\partial\mathcal U} \frac{\partial v^r_\varepsilon}{\partial\nu} {v^r_\varepsilon} p_\varepsilon+ \int_{\mathcal U} {v^r_\varepsilon}^2 p_\varepsilon=0\notag\\
     \int_{\mathcal U} {(\nabla{v^r_\varepsilon})}^2 p_\varepsilon-\int_{\mathcal U} \frac{{v^r_\varepsilon}^2}{2} \Delta p_\varepsilon+\int_{\partial\mathcal U}\frac{\partial p_\varepsilon}{\partial \nu}\frac{{v^r_\varepsilon}^2}{2} -\int_{\partial\mathcal U}\frac{\partial v^r_\varepsilon}{\partial \nu}{v^r_\varepsilon p_\varepsilon}+ \int_{\mathcal U} {v^r_\varepsilon}^2 p_\varepsilon=0\notag\\
    \int_{\mathcal U} {(\nabla{v^r_\varepsilon})}^2 p_\varepsilon + \int_{\mathcal U} \Bigl(p_\varepsilon-\frac{\Delta p_\varepsilon}{2} \Bigr){v^r_\varepsilon}^2+\int_{\partial\mathcal U}\frac{\partial p_\varepsilon}{\partial \nu}\frac{{v^r_\varepsilon}^2}{2} -\int_{\partial\mathcal U}\frac{\partial v^r_\varepsilon}{\partial \nu}{v^r_\varepsilon p_\varepsilon}=0.\label{eq:E1Nd}
\end{align}

\begin{remark}\label{rem:ipp1}
The previous integrations by parts are legitimate if $p_\varepsilon$ is regular enough, say $p_\varepsilon\in C^2(\overline{\mathcal U})$. Indeed,  $v_\varepsilon^r\in H^1(\mathcal U)$ thus $v_\varepsilon^r p_\varepsilon\in H^1(\mathcal{U})$, in addition $\Delta v_\varepsilon^r\in L^2(\mathcal U)$, hence $\nabla v_\varepsilon^r\in H_\mathrm{div}(\mathcal U)$ thus the second line is guaranteed by the Stokes formula (see Theorem~\ref{th:ippHdiv}). On the other hand, the following equality holds
$$
\int_\mathcal{U} \nabla v^r_\varepsilon\cdot \nabla p_\varepsilon {v^r_\varepsilon}=-\int_{\mathcal U} \frac{{v^r_\varepsilon}^2}{2} \Delta p_\varepsilon+\int_{\partial\mathcal U}\frac{\partial p_\varepsilon}{\partial \nu}\frac{{v^r_\varepsilon}^2}{2}.
$$
This can be shown by establishing the result first for $v_\varepsilon^r\in C^1(\mathcal U) \cap C(\overline{\mathcal U})$ as we did in the previous formal integrations by parts, and then extending it using the density of $C^1(\mathcal U) \cap C(\overline{\mathcal U})$ in $H^1(\mathcal U)$~\cite{adams2003sobolev}.
Indeed, we can use the continuity of the following bilinear form to pass to the limit.

\begin{equation*}
\begin{aligned}
H^1(\mathcal U)\times H^1(\mathcal U) & \longrightarrow  \mathbb{R} \\
(u,v) & \longmapsto \int_\mathcal{U} \nabla u\cdot \nabla p_\varepsilon\, v +\int_{\mathcal U} \frac{uv}{2} \Delta p_\varepsilon-\int_{\partial\mathcal U}\frac{\partial p_\varepsilon}{\partial \nu}\frac{u v}{2}.
\end{aligned}
\end{equation*}

\end{remark}

Similarly, we multiply Eq.~\eqref{eq:wremainderNd} by $w^r_\varepsilon q_\varepsilon$ and we integrate over $\omega$, we obtain
\begin{align}
    -\int_\omega \Delta{w^r_\varepsilon} {w^r_\varepsilon} q_\varepsilon+ \int_\omega {w^r_\varepsilon}^2 q_\varepsilon+\frac1\varepsilon \int_\omega \nabla{w^r_\varepsilon}\cdot n\, {w^r_\varepsilon} q_\varepsilon +\frac\alpha\varepsilon \int_\omega {w^r_\varepsilon}^2 q_\varepsilon  &=\int_\omega R^\varepsilon_\mathrm{obst} {w^r_\varepsilon} q_\varepsilon \notag\\
    \int_\omega \nabla{w^r_\varepsilon}\cdot\nabla ({w^r_\varepsilon} q_\varepsilon)-\int_{\partial\omega} \frac{\partial{w^r_\varepsilon}}{\partial\nu_\omega}\cdot {w^r_\varepsilon} q_\varepsilon+ \int_\omega {w^r_\varepsilon}^2 q_\varepsilon+\frac1\varepsilon \int_\omega \nabla \Bigl(\frac{{w^r_\varepsilon}^2 }2\Bigr)\cdot (q_\varepsilon n) +\frac\alpha\varepsilon \int_\omega {w^r_\varepsilon}^2 q_\varepsilon  &=\int_\omega R^\varepsilon_\mathrm{obst} {w^r_\varepsilon} q_\varepsilon\notag\\
    \int_\omega ({\nabla{w^r_\varepsilon}})^2 q_\varepsilon+\int_\omega \nabla{w^r_\varepsilon}\cdot \nabla q_\varepsilon {w^r_\varepsilon} -\int_{\partial\mathcal U} \frac{{\partial w^r_\varepsilon}}{\partial\nu_\omega}\cdot {w^r_\varepsilon} q_\varepsilon+ \int_\omega {w^r_\varepsilon}^2 q_\varepsilon \qquad \qquad \qquad \qquad \qquad&\notag\\
    -\frac1\varepsilon \int_\omega \frac{{w^r_\varepsilon}^2}2 \mathrm{div}(q_\varepsilon n) +\frac1\varepsilon \int_{\partial\mathcal U} \frac{{w^r_\varepsilon}^2}2 q_\varepsilon n\cdot \nu_\omega +\frac\alpha\varepsilon \int_\omega {w^r_\varepsilon}^2 q_\varepsilon&=\int_\omega R^\varepsilon_\mathrm{obst} {w^r_\varepsilon} q_\varepsilon\notag\\
    \int_\omega ({\nabla{w^r_\varepsilon}})^2 q_\varepsilon+\int_\omega \nabla\Bigl(\frac{{w^r_\varepsilon}^2}2\Bigr)\cdot \nabla q_\varepsilon -\int_{\partial\mathcal U} \frac{{\partial w^r_\varepsilon}}{\partial\nu_\omega} {w^r_\varepsilon} q_\varepsilon+ \int_\omega {w^r_\varepsilon}^2 q_\varepsilon \qquad \qquad \qquad \qquad \qquad&\notag\\
    -\frac1\varepsilon \int_\omega \frac{{w^r_\varepsilon}^2}2 \mathrm{div}(q_\varepsilon n) -\frac1\varepsilon \int_{\partial\mathcal U} \frac{{w^r_\varepsilon}^2}2 q_\varepsilon +\frac\alpha\varepsilon \int_\omega {w^r_\varepsilon}^2 q_\varepsilon&=\int_\omega R^\varepsilon_\mathrm{obst} {w^r_\varepsilon} q_\varepsilon\notag\\
    \int_\omega ({\nabla{w^r_\varepsilon}})^2 q_\varepsilon-\int_\omega  \Delta q_\varepsilon \frac{{w^r_\varepsilon}^2}2 +\int_{\partial\mathcal U}  \frac{{w^r_\varepsilon}^2}2 \frac{\partial q_\varepsilon}{\partial\nu_\omega}  -\int_{\partial\mathcal U} \frac{{\partial w^r_\varepsilon}}{\partial\nu_\omega} {w^r_\varepsilon} q_\varepsilon+ \int_\omega {w^r_\varepsilon}^2 q_\varepsilon \qquad \qquad \qquad \qquad \qquad&\notag\\
    -\frac1\varepsilon \int_\omega \frac{{w^r_\varepsilon}^2}2 \mathrm{div}(q_\varepsilon n) -\frac1\varepsilon \int_{\partial\mathcal U} \frac{{w^r_\varepsilon}^2}2 q_\varepsilon +\frac\alpha\varepsilon \int_\omega {w^r_\varepsilon}^2 q_\varepsilon&=\int_\omega R^\varepsilon_\mathrm{obst} {w^r_\varepsilon} q_\varepsilon\notag\\
    \int_\omega ({\nabla{w^r_\varepsilon}})^2 q_\varepsilon+\int_\omega  {w^r_\varepsilon}^2 \biggl(\Bigl(1+\frac\alpha\varepsilon\Bigr) q_\varepsilon-\frac{\Delta q_\varepsilon}2 - \frac{\mathrm{div}(q_\varepsilon n)}{2\varepsilon}\biggr)\qquad \qquad \qquad \qquad \qquad&\label{eq:E2Nd}\\
    +\int_{\partial\mathcal U}\frac{{w^r_\varepsilon}^2}2 \Bigl(\frac{\partial q_\varepsilon}{\partial\nu_\omega}-\frac{q_\varepsilon}\varepsilon\Bigr)  -\int_{\partial\mathcal U} \frac{{\partial w^r_\varepsilon}}{\partial\nu_\omega}{w^r_\varepsilon} q_\varepsilon  &=\int_\omega R^\varepsilon_\mathrm{obst} {w^r_\varepsilon} q_\varepsilon. \notag
    \end{align}

    \begin{remark}\label{rem:ipp2}
        As in Remark~\ref{rem:ipp1}, the previous integrations by parts can be made legitimate if $q_\varepsilon\in C^2(\overline{\omega})$, meaning the last equation holds. Indeed,  $w_\varepsilon^r\in H^1(\omega)$ thus $w_\varepsilon^r q_\varepsilon\in H^1(\omega)$, in addition $\Delta w_\varepsilon^r\in L^2(\omega)$, hence $\nabla w_\varepsilon^r\in H_\mathrm{div}(\omega)$ thus the first integration  by parts is again guaranteed by the Stokes formula (see Theorem~\ref{th:ippHdiv}). On the other hand, as in Remark~\ref{rem:ipp1}, we can show the following formulas
        $$\int_\omega \nabla{w^r_\varepsilon}\cdot n\, {w^r_\varepsilon} q_\varepsilon=- \int_\omega \frac{{w^r_\varepsilon}^2}2 \mathrm{div}(q_\varepsilon n) + \int_{\partial\mathcal U} \frac{{w^r_\varepsilon}^2}2 q_\varepsilon n\cdot \nu_\omega$$
        and
        $$\int_\omega \nabla{w^r_\varepsilon}\cdot \nabla q_\varepsilon {w^r_\varepsilon}=-\int_\omega  \Delta q_\varepsilon \frac{{w^r_\varepsilon}^2}2 +\int_{\partial\mathcal U}  \frac{{w^r_\varepsilon}^2}2 \frac{\partial q_\varepsilon}{\partial\nu_\omega}$$
        using the density of $C^1(\omega)\cap C(\overline{\omega})$ in $H^1(\omega)$ and the continuity of the bilinear forms:
        \begin{equation*}
\begin{aligned}
H^1(\mathcal U)\times H^1(\mathcal U) & \longrightarrow  \mathbb{R} \\
(u,v) & \longmapsto \int_\omega \nabla{u}\cdot n\, {v} q_\varepsilon+ \int_\omega \frac{u v}2 \mathrm{div}(q_\varepsilon n) - \int_{\partial\mathcal U} \frac{u v}2 q_\varepsilon n\cdot \nu_\omega,
\end{aligned}
\end{equation*}
and
\begin{equation*}
\begin{aligned}
H^1(\mathcal U)\times H^1(\mathcal U) & \longrightarrow  \mathbb{R} \\
(u,v) & \longmapsto \int_\mathcal{U} \nabla u\cdot \nabla q_\varepsilon\, v +\int_{\mathcal U} \frac{uv}{2} \Delta q_\varepsilon-\int_{\partial\mathcal U}\frac{\partial q_\varepsilon}{\partial \nu}\frac{u v}{2}.
\end{aligned}
\end{equation*}

    \end{remark}

By adding~\eqref{eq:E1Nd} and~\eqref{eq:E2Nd}, and by choosing $p_\varepsilon$ and $q_\varepsilon$ such that the interface terms vanish, i.e.
\begin{equation}\label{eq:transmissionpqNd}
    \begin{dcases}
    p_\varepsilon=q_\varepsilon & \text{on $\partial\mathcal U$}\\
    \frac{\partial p_\varepsilon}{\partial\nu}=\frac{\partial q_\varepsilon}{\partial\nu}+\frac{q_\varepsilon}\varepsilon & \text{on $\partial\mathcal U$},
    \end{dcases}
\end{equation}
where $\nu$ is the outward unit normal vector to $\mathcal U$ (notice that $\nu_\omega=-\nu$ on $\partial\mathcal{U}$), we obtain
\begin{equation}\label{eq:weightedestimateNd}
    \int_{\mathcal U} ({\nabla{v^r_\varepsilon}})^2 p_\varepsilon + \int_{\mathcal U} \Bigl(p_\varepsilon-\frac{\Delta p_\varepsilon}{2} \Bigr){v^r_\varepsilon}^2 +\int_\omega ({\nabla{w^r_\varepsilon}})^2 q_\varepsilon+\int_\omega {w^r_\varepsilon}^2 \Bigl(\Bigl(1+\frac\alpha\varepsilon\Bigr)q_\varepsilon-\frac{\Delta q_\varepsilon}2-\frac{\mathrm{div}(q_\varepsilon n)}{2\varepsilon}\Bigr)=\int_\omega R^\varepsilon_\mathrm{obst} {w^r_\varepsilon} q_\varepsilon.
\end{equation}
{\bf Assume there exists $p_\varepsilon$ and $q_\varepsilon$} such that for sufficiently small $\varepsilon>0$
\begin{equation}\label{eq:condpqNd}
    \begin{dcases}
    \text{\eqref{eq:transmissionpqNd} is satisfied}\\
    p_\varepsilon \ge \beta>0 & \text{in $\mathcal{U}$}\\
    q_\varepsilon\ge \beta>0 & \text{in $\omega$}\\
    p_\varepsilon-\frac{\Delta p_\varepsilon}{2}\ge 0 & \text{in $\mathcal{U}$}\\
    \Bigl(1+\frac\alpha\varepsilon\Bigr)q_\varepsilon-\frac{\Delta q_\varepsilon}2-\frac{\mathrm{div}(q_\varepsilon n)}{2\varepsilon}\ge 0 & \text{in $\omega$}\\
    \|q_\varepsilon\|_{L^\infty(\omega_1)}\le C \text{ and }\|q_\varepsilon\|_{L^\infty(\omega\setminus\omega_1)}\le \frac{C}\varepsilon,
    \end{dcases}
\end{equation}
then from~\eqref{eq:weightedestimateNd}, we deduce, using also~\eqref{eq:rhsestimatefiner1Nd} and~\eqref{eq:rhsestimatefiner2Nd}, that
\begin{align*}
    \beta \Bigl( \int_{\mathcal U} {(\nabla{v^r_\varepsilon})}^2 + \int_\omega {({\nabla w^r_\varepsilon})}^2 \Bigr) &\le \|R^\varepsilon_\mathrm{obst} q_\varepsilon\|_{L^2(\omega)} \|w^r_\varepsilon\|_{L^2(\omega)}\\
     &\le \Bigl( \|R^\varepsilon_\mathrm{obst} q_\varepsilon\|_{L^2(\omega_1)}+\|R^\varepsilon_\mathrm{obst} q_\varepsilon\|_{L^2(\omega\setminus\omega_1)}\Bigr) \|w^r_\varepsilon\|_{L^2(\omega)}\\
         &\le \Bigl( \|R^\varepsilon_\mathrm{obst}\|_{L^2(\omega_1)} \|q_\varepsilon\|_{L^\infty(\omega_1)}+\|R^\varepsilon_\mathrm{obst}\|_{L^2(\omega\setminus\omega_1)} \|q_\varepsilon\|_{L^\infty(\omega\setminus\omega_1)}\Bigr) \|w^r_\varepsilon\|_{L^2(\omega)}\\
         &\le C \|{\nabla w^r_\varepsilon}\|_{L^2(\omega)} \qquad\text{(by Poincaré inequality ($w^r_\varepsilon=0$ on $\partial\Omega$))}\\
    \|{\nabla v^r_\varepsilon}\|^2_{L^2(\mathcal U)} + \|{\nabla w^r_\varepsilon}\|^2_{L^2(\omega)} &\le C \|{\nabla w^r_\varepsilon}\|_{L^2(\omega)}\\
    &\le \frac{C^2}2+\frac12 \|{\nabla w^r_\varepsilon}\|^2_{L^2(\omega)}  \qquad\text{(by Young inequality)}\\
    \|{\nabla v^r_\varepsilon}\|^2_{L^2(\mathcal U)} + \frac12 \|{\nabla w^r_\varepsilon}\|^2_{L^2(\omega)} &\le C.
\end{align*}
    In conclusion, using again Poincaré inequality, we have established that $\|v^r_\varepsilon\|_{H^1(\mathcal U)}\le C$ and $\|w^r_\varepsilon\|_{H^1(\omega)}\le C$. Indeed, using Poincaré inequality for $w^r_\varepsilon\in H^1(\omega)$ with  $w^r_\varepsilon=0$ on $\partial\Omega$, we get
    \begin{align*}
    	\|w^r_\varepsilon\|_{L^2(\omega)}\le C \|\nabla w^r_\varepsilon\|_{L^2(\omega)} \le C,
    \end{align*}
    and, using Poincaré inequality for $u^r_\varepsilon\in H^1_0(\Omega)$ (defined by $v^r_\varepsilon$ in $\mathcal U$ and $w^r_\varepsilon$ in $\omega$), we get
        \begin{align*}
        	\|v^r_\varepsilon\|^2_{L^2(\mathcal{U})}&\le \|u^r_\varepsilon\|^2_{L^2(\Omega)}\\
        										  &\le C \|\nabla u^r_\varepsilon\|^2_{L^2(\Omega)}\\
        										  &\le C (\|\nabla v^r_\varepsilon\|^2_{L^2(\mathcal U)}+\|\nabla w^r_\varepsilon\|^2_{L^2(\omega)})\\
        										  & \le C.										 
        \end{align*}   
    The convergence of the asymptotic expansion follows as explained in Subsection~\ref{sec:convergenceNd}.

We thus have the following lemma.
\begin{lemma}
    If there exists $p_\varepsilon$, $q_\varepsilon$ satisfying~\eqref{eq:condpqNd}, then the penalized method converges, in the sense that equations~\eqref{eq:convergencefluidNd},~\eqref{eq:boundarylayerobstacleNd} and~\eqref{eq:convergeneobstacleNd} hold.
\end{lemma}

The construction of such suitable weight functions is the subject of the following subsection.

%% file: sections/subsections/BLsupersolutions.tex
In the previous subsection, we showed that the convergence of the penalization method reduces to finding suitable weight functions $p_\varepsilon$ and $q_\varepsilon$ satisfying~\eqref{eq:condpqNd}. In order to get rid of the coefficients $\frac12$, we will seek for weight functions $p_\varepsilon$, $q_\varepsilon$ that satisfy
\begin{equation}\label{eq:pqNd}
    \begin{dcases}
    -\Delta p_\varepsilon+p_\varepsilon=a_\varepsilon\ge 0 & \text{in $\mathcal{U}$}\\
    -\Delta q_\varepsilon-\frac{1}{\varepsilon}\mathrm{div}(q_\varepsilon n)+q_\varepsilon=b_\varepsilon\ge 0 & \text{in $\omega$}\\
    p_\varepsilon=q_\varepsilon & \text{on $\partial\mathcal U$}\\
    \frac{\partial p_\varepsilon}{\partial\nu}=\frac{\partial q_\varepsilon}{\partial\nu}+\frac{q_\varepsilon}\varepsilon & \text{on $\partial\mathcal U$}\\
    p_\varepsilon \ge \beta>0 & \text{in $\mathcal{U}$}\\
    q_\varepsilon\ge \beta>0 & \text{in $\omega$}\\
    \|q_\varepsilon\|_{L^\infty(\omega_1)}\le C \text{ and }\|q_\varepsilon\|_{L^\infty(\omega\setminus\omega_1)}\le \frac{C}\varepsilon,
    \end{dcases}
\end{equation}
with $a_\varepsilon$ and $b_\varepsilon$ sufficiently regular.

It is easy to show that if $p_\varepsilon$,  $q_\varepsilon$ satisfy~\eqref{eq:pqNd}, then they satisfy~\eqref{eq:condpqNd}. Indeed,~\eqref{eq:transmissionpqNd} is satisfied, and we have
\begin{align*}
    p_\varepsilon-\frac{\Delta p_\varepsilon}{2}=\frac{p_\varepsilon}{2}+\frac{-\Delta p_\varepsilon+p_\varepsilon}{2}\ge \frac{p_\varepsilon}2\ge 0\\
    \Bigl(1+\frac\alpha\varepsilon\Bigr)q_\varepsilon-\frac{\Delta q_\varepsilon}2-\frac{\mathrm{div}(q_\varepsilon n)}{2\varepsilon}=\frac\alpha\varepsilon q_\varepsilon + \frac{q_\varepsilon}{2}+\frac12\Bigl(-\Delta q_\varepsilon-\frac{1}{\varepsilon}\mathrm{div}(q_\varepsilon n)+q_\varepsilon\Bigr)\ge\Bigl(\frac\alpha\varepsilon+\frac12\Bigr)q_\varepsilon \ge 0.
\end{align*}

%% file: sections/subsections/BLlinkdual.tex
%

~\eqref{eq:pqNd} is related to a \textbf{dual problem} of the penalized problem. Indeed, let $s\in H^1_0(\Omega)$, where $\Omega=\mathcal U\cup \omega\cup\partial\mathcal U$, then multiplying the equations~\eqref{eq:pqNd} by $p_\varepsilon$ and $q_\varepsilon$ respectively and integrating by parts, we obtain:
\begin{equation}
    \int_{\mathcal U} \nabla{p_\varepsilon}\cdot\nabla s-\int_{\partial\mathcal U}\frac{\partial p_\varepsilon}{\partial\nu} s+\int_{\mathcal U} p_\varepsilon s=\int_{\mathcal U} a_\varepsilon s
\end{equation}
and
\begin{align}
    \int_\omega \nabla{q_\varepsilon}\cdot\nabla s-\int_{\partial \mathcal U}\frac{\partial q_\varepsilon}{\partial \nu_2} s+\frac1\varepsilon \int_\omega {q_\varepsilon}n\cdot\nabla s-\frac1\varepsilon \int_{\partial\mathcal U} q_\varepsilon s\, n\cdot \nu_2+\int_\omega q_\varepsilon s=\int_{\omega} b_\varepsilon s\notag\\
    \int_\omega \nabla{q_\varepsilon}\cdot\nabla s+\int_{\partial \mathcal U}\frac{\partial q_\varepsilon}{\partial \nu} s+\frac1\varepsilon \int_\omega {q_\varepsilon}n\cdot\nabla s+\frac1\varepsilon \int_{\partial\mathcal U} q_\varepsilon s+\int_\omega q_\varepsilon s=\int_{\omega} b_\varepsilon s,
\end{align}
where $\nu$ and $\nu_2$ are respectively the outward and inward normal vectors to $\mathcal U$.

Summing the two previous equations and using the transmission conditions, we obtain
\begin{equation}
    \int_{\mathcal U} \nabla{p_\varepsilon}\cdot\nabla s+\int_\omega \nabla{q_\varepsilon}\cdot\nabla s+\int_{\mathcal U} p_\varepsilon s+\int_\omega q_\varepsilon s+\frac1\varepsilon \int_\omega {q_\varepsilon}n\cdot\nabla s=\int_{\mathcal U} a_\varepsilon s+\int_{\omega} b_\varepsilon s.
\end{equation}
If we denote by $r_\varepsilon$ the function whose restriction to $\mathcal U$ is $p_\varepsilon$ and whose restriction to $\omega$ is $q_\varepsilon$, this means that $r_\varepsilon$ is solution of the following variational form (assuming $p_\varepsilon\in H^1(\mathcal U)$ and $q_\varepsilon\in H^1(\omega)$ yields $r_\varepsilon\in H^1(\Omega)$ using the transmission condition $p_\varepsilon=q_\varepsilon$ on $\partial\mathcal U$)
\begin{equation}\label{eq:dualproblemvarNd}
    \begin{dcases}
    \text{Find $r_\varepsilon\in V=\{v\in H^1(\Omega), r_\varepsilon={q^0_\varepsilon} \text{ on }\partial\Omega\}$ such that $\forall s\in H^1_0(\Omega)$}\\
        \int_{\Omega} \nabla{r_\varepsilon}\cdot\nabla s+\int_{\Omega} r_\varepsilon s+\frac1\varepsilon \int_\Omega \chi_\omega {r_\varepsilon}n\cdot\nabla s=\int_{\Omega} c_\varepsilon s,
    \end{dcases}
\end{equation}
which is of the form of a dual problem to the penalized problem~\eqref{eq:dualpenalizedproblemweak} (with $\alpha=0$), where the advection term is in a conservative form. The existence and uniqueness of the solution to the dual problem was also established using the approach of Droniou (it was used in the proof of the existence and uniqueness of the penalized problem), see Thoeorem~\ref{th:existenceuniquenessdroniou}.

Now, we construct supersolutions $p_\varepsilon$ and $q_\varepsilon$ that satisfy~\eqref{eq:pqNd}. In particular cases (1D case, spherical symmetric case in 2D), the construction can be explicit. We report theses cases in Appendix.~\ref{sec:annexspecialcases}.

%% file: sections/subsections/BLgeneral.tex
We recall that we assumed, for simplicity, $\psi$ to be sufficiently regular and $n=\nabla\psi$ in $\omega$, we also assumed $\psi=0$ on $\partial\mathcal{U}$, $\psi>0$ in $\omega$ and $\nabla\psi\ne 0$ in $\omega$. $\psi$ can then be interpreted as a ``distance'' to $\partial\mathcal U$.

Based on a formal boundary layer approach close to $\partial\mathcal U$ performed on the dual problem~\eqref{eq:pqNd} (see Appendix~\ref{sec:annexBLdual}), we search for supersolutions of the form
\begin{align}\label{eq:supersoltionsapp}
    p_\varepsilon^{app}(x)&=\frac1\varepsilon P^{-1}(x)+P^0(x)& \text{in $\mathcal U$}\\
    q_\varepsilon^{app}(x)&=\frac1\varepsilon \widetilde{Q}^{-1}\Bigl(x,\frac{\psi(x)}\varepsilon\Bigr)+\overline{Q}^0(x)+\widetilde{Q}^0\Bigl(x,\frac{\psi(x)}\varepsilon\Bigr)& \text{in $\omega$}
\end{align}
where
\begin{align}
\widetilde{Q}^{-1}(x,z)&=\underline{P}^{-1}(x) e^{-z}\\
\widetilde{Q}^{0}(x,z)&=(\underline{P}^{0}(x)-\overline{Q}^0(x)) e^{-z}+F(x) z e^{-z}
\end{align}
where as in Appendix~\ref{sec:annexBLdual}, $\underline{P}^i$ realizes a simultaneous lifting in $\omega$ of the trace of $P^{i}$  and of the trace of $\frac{\partial P^{i}}{\partial\nu}$ on $\partial\mathcal U$ having the same regularity as $P^{i}$.
$\overline{Q}^0$ satisfies~\eqref{eq:solQ0Nd}
\begin{equation*}
    -\nabla\overline{Q}^0\cdot n -\Delta\psi \overline{Q^0}=b^{-1} \ge 0 \quad \text{in $\omega$},
\end{equation*}
$P^{-1}$ is solution to~\eqref{eq:solPm1Ndbis}
\begin{equation*}
    \begin{dcases}
        -\Delta P^{-1}+P^{-1}=a^{-1} \ge 0 & \text{in $\mathcal{U}$}\\
        \frac{\partial P^{-1}}{\partial\nu}=\overline{Q}^{0} & \text{on $\partial\mathcal{U}$},
    \end{dcases}
\end{equation*}
and $P^{0}$ is solution to~\eqref{eq:solP0Ndbis} (we fixed $a^0=0$)
\begin{equation*}
    \begin{dcases}
        -\Delta P^{0}+P^{0}=0 & \text{in $\mathcal{U}$}\\
        \frac{\partial P^{0}}{\partial\nu}= \Gamma& \text{on $\partial\mathcal{U}$}.
    \end{dcases}
\end{equation*}
In addition,
$$F(x)=-\frac{\nabla\underline{P}^{-1}\cdot \nabla\psi}{|\nabla\psi|^2} \quad\text{in $\omega$},$$
and
$$F(x)=-\overline{Q}^0(x)\quad \text{on $\partial\mathcal U$}.$$

There are various degrees of freedom in the previous equations, namely $\overline{Q}^0_{|\partial\mathcal U}$, $\Gamma$, $a^{-1}$ and $b^{-1}$. We will show in the following that we can tune these degrees freedom so that they yield suitable supersolutions, i.e. so that $p_\varepsilon^{app}$ and $q_\varepsilon^{app}$ satisfy~\eqref{eq:pqNd} (and hence also~\eqref{eq:condpqNd}) for sufficiently small $\varepsilon$.


\paragraph{Checking transmission conditions in~\eqref{eq:pqNd}}

We have
\begin{align*}
\widetilde{Q}^{-1}(x,0)&=\underline{P}^{-1}(x)=P^{-1}(x) & \text{on $\partial\mathcal U$}\\
\overline{Q}^0(x)+\widetilde{Q}^0(x,0)&=\overline{Q}^0(x)+\underline{P}^0(x)-\overline{Q}^0(x)=P^0(x)& \text{on $\partial\mathcal U$},
\end{align*}
therefore the first transmission condition is satisfied:
$$p_\varepsilon^{app}(x)=q_\varepsilon^{app}(x)\quad \text{on $\partial\mathcal{U}$}.$$

Regarding the transmission condition on the normal derivatives, we have
$$\nabla\biggl(\widetilde{Q}^{-1}\Bigl(x,\frac{\psi(x)}\varepsilon\Bigr)\biggr)=\nabla \widetilde{Q}^{-1} \Bigl(x,\frac{\psi(x)}\varepsilon\Bigr)+\frac1\varepsilon \widetilde{Q}^{-1}_z\Bigl(x,\frac{\psi(x)}\varepsilon\Bigr) \nabla\psi(x).$$
Thus
\begin{align*}
\nabla\biggl(\widetilde{Q}^{-1}\Bigl(x,\frac{\psi(x)}\varepsilon\Bigr)\biggr)\cdot\nu&=\nabla \widetilde{Q}^{-1}(x,0)\cdot\nu+\frac1\varepsilon \widetilde{Q}^{-1}_z(x,0) \nabla\psi(x)\cdot\nu & \text{on $\partial\mathcal U$}\\
&=\nabla \widetilde{Q}^{-1}(x,0)\cdot\nu+\frac1\varepsilon ( -\widetilde{Q}^{-1}(x,0))  &\text{on $\partial\mathcal U$}\\
&=\nabla \underline{P}^{-1}(x)\cdot\nu-\frac1\varepsilon \underline{P}^{-1}(x)& \text{on $\partial\mathcal U$}\\
&=\frac{\partial P^{-1}}{\partial\nu}(x)-\frac1\varepsilon {P}^{-1}(x)& \text{on $\partial\mathcal U$},
\end{align*}
and
\begin{align*}
\nabla\biggl(\widetilde{Q}^{0}\Bigl(x,\frac{\psi(x)}\varepsilon\Bigr)\biggr)\cdot\nu&=\nabla \widetilde{Q}^{0}(x,0)\cdot\nu+\frac1\varepsilon \widetilde{Q}^{0}_z(x,0) \nabla\psi(x)\cdot\nu & \text{on $\partial\mathcal U$}\\
&=\nabla (\underline{P}^0(x)-\overline{Q}^0(x))\cdot\nu+\frac1\varepsilon ( -(\underline{P}^0(x)-\overline{Q}^0(x))+F(x))  &\text{on $\partial\mathcal U$}\\
&=\nabla \underline{P}^0(x)\cdot\nu-\nabla \overline{Q}^0(x)\cdot\nu-\frac1\varepsilon \underline{P}^{0}(x)& \text{on $\partial\mathcal U$}\\
&=\frac{\partial P^{0}}{\partial\nu}(x)-\frac{\partial \overline{Q}^{0}}{\partial\nu}(x)-\frac1\varepsilon {P}^{0}(x)& \text{on $\partial\mathcal U$}.
\end{align*}
We then obtain
\begin{align*}
\frac{\partial q_\varepsilon^{app}}{\partial \nu}+\frac{q_\varepsilon^{app}}{\varepsilon}&=\frac1\varepsilon \Bigl(\frac{\partial P^{-1}}{\partial\nu}-\frac1\varepsilon {P}^{-1}\Bigr)+\frac{\partial \overline{Q}^0}{\partial \nu}+\frac{\partial P^{0}}{\partial\nu}-\frac{\partial \overline{Q}^{0}}{\partial\nu}-\frac1\varepsilon {P}^{0}+\frac1{\varepsilon^2} P^{-1}+\frac1\varepsilon P^0& \text{on $\partial\mathcal U$}\\
&=\frac1\varepsilon\frac{\partial P^{-1}}{\partial\nu}+\frac{\partial P^{0}}{\partial\nu}& \text{on $\partial\mathcal U$}\\
&=\frac{\partial p_\varepsilon^{app}}{\partial \nu}& \text{on $\partial\mathcal U$},
\end{align*}
and the second transmission condition on the normal derivatives is also satisfied.

\paragraph{Checking supersolutions in~\eqref{eq:pqNd}}

We are going to look for dominant terms in both $-\Delta p^{app}_\varepsilon+p^{app}_\varepsilon$ and $-\Delta q^{app}_\varepsilon-\frac{1}{\varepsilon}\mathrm{div}(q^{app}_\varepsilon n)+q^{app}_\varepsilon$ in order to show that $p_\varepsilon^{app}$ and $q_\varepsilon^{app}$ are indeed supersolutions for sufficiently small $\varepsilon$.

Since $P^{-1}$ is solution to~\eqref{eq:solPm1Ndbis}, the dominant term in $-\Delta p^{app}_\varepsilon+p^{app}_\varepsilon$ is 
$$\frac1\varepsilon a^{-1}(x) \ge 0.$$

Regarding $q_\varepsilon^{app}$, one has, using~\eqref{eq:dualderivativesNd},
\begin{align*}
-\Delta q^{app}_\varepsilon-\frac{1}{\varepsilon}\mathrm{div}(q^{app}_\varepsilon n)+q^{app}_\varepsilon=&-\Delta q^{app}_\varepsilon-\frac{1}{\varepsilon}\nabla q^{app}_\varepsilon \cdot\nabla\psi-\frac{1}\varepsilon \Delta\psi q^{app}_\varepsilon+q^{app}_\varepsilon\\
=&-\frac1\varepsilon \Delta\widetilde{Q}^{-1}-\frac2{\varepsilon^2}\nabla \widetilde{Q}^{-1}_z\cdot\nabla\psi-\frac1{\varepsilon^2} \Delta\psi \widetilde{Q}^{-1}_z-\frac1{\varepsilon^3}|\nabla\psi|^2 \widetilde{Q}^{-1}_{zz}\\
&\quad -\Delta \overline{Q}^0\\
&\quad - \Delta\widetilde{Q}^{0}-\frac2{\varepsilon}\nabla \widetilde{Q}^{0}_z\cdot\nabla\psi-\frac1{\varepsilon} \Delta\psi \widetilde{Q}^{0}_z-\frac1{\varepsilon^2}|\nabla\psi|^2 \widetilde{Q}^{0}_{zz}\\
&-\frac1{\varepsilon^2}\nabla \widetilde{Q}^{-1}_z\cdot\nabla\psi-\frac1{\varepsilon^3} \widetilde{Q}^{-1}_z |\nabla\psi|^2\\
&\quad -\frac1{\varepsilon}\nabla\overline{Q}^0\cdot\nabla\psi\\
&\quad -\frac1{\varepsilon}\nabla \widetilde{Q}^{0}_z\cdot\nabla\psi-\frac1{\varepsilon^2} \widetilde{Q}^{0}_z |\nabla\psi|^2\\
&-\frac1{\varepsilon^2}\Delta\psi \widetilde{Q}^{-1}-\frac1\varepsilon\Delta\psi \overline{Q}^0-\frac1\varepsilon \Delta\psi \widetilde{Q}^0\\
&+\frac1{\varepsilon} \widetilde{Q}^{-1}+\overline{Q}^0+\widetilde{Q}^0.
\end{align*}

Using the equations satisfied by $\overline{Q}^0$, $\widetilde{Q}^{-1}$ and $\widetilde{Q}^0$ (see Appendix~\ref{sec:annexBLdual}) that we recall in the following,
\begin{align*}
-\overline{Q}^0\Delta\psi-\nabla \overline{Q}^0\cdot\nabla\psi&=b^ {-1}\\
-\widetilde{Q}^{-1}_{zz}-\widetilde{Q}^{-1}_z&=0\\
\widetilde{Q}^{-1}_z&=-\widetilde{Q}^{-1}\\
-\widetilde{Q}^0_{zz}-\widetilde{Q}^0_z&=F(x) e^{-z}\\
\widetilde{Q}^0_z+\widetilde{Q}^0&=F(x)e^{-z},
\end{align*}
the equation on $q_\varepsilon^{app}$ reduces to
\begin{align*}
-\Delta q^{app}_\varepsilon-\frac{1}{\varepsilon}\mathrm{div}(q^{app}_\varepsilon n)+q^{app}_\varepsilon=&-\frac1\varepsilon \Delta \widetilde{Q}^{-1}+\frac2{\varepsilon^2}\nabla\widetilde{Q}^{-1}\cdot\nabla\psi\\
&\quad -\Delta \overline{Q}^0\\
&\quad -\Delta \widetilde{Q}^0- \frac2{\varepsilon}\nabla(F(x) e^{-z}-\widetilde{Q}^{0})\cdot\nabla\psi-\frac1\varepsilon \Delta\psi F(x) e^{-z}+\frac1{\varepsilon^2}|\nabla\psi|^2 F(x) e^{-z}\\
&-\frac1{\varepsilon^2}\nabla\widetilde{Q}^{-1}\cdot\nabla\psi\\
&\quad +\frac1\varepsilon b^{-1}\\
&\quad -\frac1\varepsilon \nabla \widetilde{Q}^0\cdot\nabla\psi\\
&+\frac1\varepsilon \widetilde{Q}^{-1}+\overline{Q}^0+\widetilde{Q}^0,
\end{align*}
evaluated at $z=\frac{\psi(x)}{\varepsilon}$.

The terms involving
\begin{itemize}
\item $\overline{Q}^0:$
\begin{align*}
\frac1\varepsilon b^{-1}-\Delta \overline{Q}^0+\overline{Q}^0.
\end{align*}

 \item $\widetilde{Q}^{-1}:$
\begin{align*}
&-\frac1\varepsilon \Delta \widetilde{Q}^{-1}+\frac2{\varepsilon^2}\nabla\widetilde{Q}^{-1}\cdot\nabla\psi-\frac1{\varepsilon^2}\nabla\widetilde{Q}^{-1}\cdot\nabla\psi+\frac1\varepsilon \widetilde{Q}^{-1}\\
&\qquad = -\frac1\varepsilon \Delta \widetilde{Q}^{-1}+\frac1{\varepsilon^2}\nabla\widetilde{Q}^{-1}\cdot\nabla\psi+\frac1\varepsilon \widetilde{Q}^{-1}\\
&\qquad = \Bigl(\frac1{\varepsilon^2}\nabla\underline{P}^{-1}\cdot\nabla\psi-\frac1\varepsilon \Delta \underline{P}^{-1}+\frac1\varepsilon \underline{P}^{-1}\Bigr) e^{-z}.
\end{align*}

\item $\widetilde{Q}^0:$
\begin{align*}
& -\Delta \widetilde{Q}^0+ \frac2{\varepsilon}\nabla \widetilde{Q}^{0}\cdot\nabla\psi-\frac1\varepsilon \nabla \widetilde{Q}^0\cdot\nabla\psi+\widetilde{Q}^0\\
&\qquad =-\Delta \widetilde{Q}^0+ \frac1{\varepsilon}\nabla \widetilde{Q}^{0}\cdot\nabla\psi+\widetilde{Q}^0\\
&\qquad = -(\Delta \underline{P}^0-\Delta\overline{Q}^0 ) e^{-z}-\Delta F(x) z e^{-z}+\frac1\varepsilon\Bigl((\nabla \underline{P}^0-\nabla\overline{Q}^0)\cdot\nabla\psi e^{-z}+\nabla F(x)\cdot\nabla\psi z e^{-z} \Bigr)\\
&\qquad\quad +(\underline{P}^0-\overline{Q}^0)e^{-z}+F(x) z e^{-z}\\
&\qquad = \Bigl(-\Delta \underline{P}^0+\Delta\overline{Q}^0+\frac1\varepsilon\nabla \underline{P}^0\cdot\nabla\psi-\frac1\varepsilon\nabla\overline{Q}^0\cdot\nabla\psi+\underline{P}^0-\overline{Q}^0\Bigr) e^{-z}\\
&\qquad\quad +\Bigl(-\Delta F(x)+\frac1\varepsilon\nabla F(x)\cdot\nabla\psi+F(x)\Bigr)z e^{-z}.
\end{align*}

\item remaining terms:
\begin{align*}
 &-\frac2{\varepsilon}\nabla F(x)\cdot\nabla\psi e^{-z}-\frac1\varepsilon \Delta\psi F(x) e^{-z}+\frac1{\varepsilon^2}|\nabla\psi|^2 F(x) e^{-z}\\
 &\quad =\Bigl(-\frac2{\varepsilon}\nabla F(x)\cdot\nabla\psi-\frac1\varepsilon \Delta\psi F(x) +\frac1{\varepsilon^2}|\nabla\psi|^2 F(x)\Bigr) e^{-z}.
\end{align*}

\end{itemize}

In conclusion, recalling $z=\frac{\psi(x)}{\varepsilon}$, the equation on $q_\varepsilon^{app}$ reads
\begin{align}\label{eq:Eqqapp}
&\frac1\varepsilon b^{-1}-\Delta \overline{Q}^0+\overline{Q}^0+\Bigl(\frac1{\varepsilon^2}\nabla F(x)\cdot\nabla\psi+\frac1{\varepsilon}\bigl(-\Delta F(x)+F(x)\bigr)\Bigr) \psi(x) e^{-\frac{\psi(x)}{\varepsilon}}\notag\\
&\quad +\frac1{\varepsilon^2}\bigl(\nabla\underline{P}^{-1}\cdot\nabla\psi+|\nabla\psi|^2 F(x)\bigr)e^{-\frac{\psi(x)}{\varepsilon}}\notag\\
&\quad+\frac1\varepsilon \Bigl(-\Delta \underline{P}^{-1}+ \underline{P}^{-1}+\nabla\underline{P}^0\cdot\nabla\psi-\nabla\overline{Q}^0\cdot\nabla\psi-2\nabla F(x)\cdot\nabla\psi-\Delta\psi F(x)\Bigr)e^{-\frac{\psi(x)}{\varepsilon}}\notag\\
&\quad +\big(-\Delta\underline{P}^0+\underline{P}^0+\Delta\overline{Q}^0-\overline{Q}^0\bigr)e^{-\frac{\psi(x)}{\varepsilon}}\notag\\
&=\frac1\varepsilon b^{-1}-\Delta \overline{Q}^0+\overline{Q}^0+\Bigl(\frac1{\varepsilon^2}\nabla F(x)\cdot\nabla\psi+\frac1{\varepsilon}\bigl(-\Delta F(x)+F(x)\bigr)\Bigr) \psi(x) e^{-\frac{\psi(x)}{\varepsilon}}\notag\\
&\quad+\frac1\varepsilon \Bigl(-\Delta \underline{P}^{-1}+ \underline{P}^{-1}+\nabla\underline{P}^0\cdot\nabla\psi-\nabla\overline{Q}^0\cdot\nabla\psi-2\nabla F(x)\cdot\nabla\psi-\Delta\psi F(x)\Bigr)e^{-\frac{\psi(x)}{\varepsilon}}\notag\notag\\
&\quad +\big(-\Delta\underline{P}^0+\underline{P}^0+\Delta\overline{Q}^0-\overline{Q}^0\bigr)e^{-\frac{\psi(x)}{\varepsilon}}.
\end{align}
Assuming enough regularity on the functions depending on $x$, we have, for $x$ such that $\psi(x)\ge \delta_0>0$ (not including the interface $\partial\mathcal U$), and $k\in\{0,1,2\}$,
\begin{equation}\label{eq:limit}
\Bigl\|\frac{G_k(x)}{\varepsilon^k} e^{-\frac{\psi(x)}{\varepsilon}}\Bigr\|_{L^\infty(\psi\ge\delta_0)}\le \frac{C}{\varepsilon^2} e^{-\frac{\delta_0}{\varepsilon}} \to 0 \text{ when $\varepsilon\to 0$},
\end{equation}
thus for $\delta>0$, we have for sufficiently small $\varepsilon$,
$$\frac{G_k(x)}{\varepsilon^k} e^{-\frac{\psi(x)}{\varepsilon}}\ge -\Bigl\|\frac{G_k(x)}{\varepsilon^k} e^{-\frac{\psi(x)}{\varepsilon}}\Bigr\|_{L^\infty(\psi\ge\delta_0)}\ge -\delta.$$ 
Therefore, for $q_\varepsilon^{app}$ to be a supersolution in the region $\psi\ge\delta_0$, it is sufficient that $b^{-1}\ge \eta >0 $ in this region, or that $b^{-1}=0$ and $-\Delta \overline{Q}^0+\overline{Q}^0\ge \eta >\delta >0$ in $\omega$, where $\eta$ is a constant.

On the interface $\partial\mathcal{U}$ ($\psi(x)=0$), the dominant term in the equation $q_\varepsilon^{app}$~\eqref{eq:Eqqapp} is 
\begin{align*}
\frac1\varepsilon \bigl(b^{-1}-\Delta \underline{P}^{-1}+ \underline{P}^{-1}+\frac{\partial P^0}{\partial\nu}-\frac{\partial\overline{Q}^0}{\partial \nu}-2\nabla F(x)\cdot\nu+\Delta\psi \overline{Q}^0\bigr)
\end{align*}
so that it is sufficient to impose the following equation on $P^0$
\begin{equation}
\begin{dcases}
-\Delta P^0+P^0=0 & \text{in $\mathcal U$}\\ 
\frac{\partial P^0}{\partial\nu}= \Gamma > -b^{-1}+\Delta \underline{P}^{-1}- \underline{P}^{-1}+\frac{\partial\overline{Q}^0}{\partial \nu}+2\nabla F(x)\cdot\nu-\Delta\psi \overline{Q}^0& \text{on $\partial\mathcal U$,}
\end{dcases}
\end{equation}
in order that $q_\varepsilon^{app}$ is a strict supersolution on the boundary $\partial\mathcal U$.

By continuity (assumption) hence uniform continuity, $q_\varepsilon^{app}$ is a supersolution in the region $\psi<\delta_0$ (near the boundary) since we can choose $\delta_0$ small enough so that the equation~\eqref{eq:Eqqapp} remains positive near the boundary.

\paragraph{Imposing positivity of supersolutions in~\eqref{eq:pqNd}}

We want $p_\varepsilon^{app}$ and $q_\varepsilon^{app}$ to be positive for $\varepsilon$ small enough: $p_\varepsilon^{app}\ge \beta>0$ and $q_\varepsilon^{app}\ge \beta>0$.

Given the expression of $p_\varepsilon^{app}$, it is sufficient that $P^{-1}>0$ in $\overline{\mathcal{U}}$ where $P^{-1}$ is solution to
\begin{equation}
\begin{dcases}
-\Delta P^{-1}+P^{-1}=a^{-1}\ge 0 & \text{in $\mathcal U$}\\
\frac{\partial P^{-1}}{\partial\nu}= \overline{Q}^0 & \text{on $\partial\mathcal U$}.
\end{dcases}
\end{equation}
This can be achieved with a suitable choice of $a^{-1}$ and $\overline{Q}^0$ on $\partial\mathcal{U}$  using the maximum principle. Indeed, using Corollary~\ref{cor:maximumprinciplebis}, if $a^{-1}\ge \beta>0$ in $\mathcal U$ and $\overline{Q}^0\ge 0$ on $\partial\mathcal U$, then $P^{-1}\ge \beta>0$.

Regarding $q_\varepsilon^{app}$, one has
\begin{align}\label{eq:Expqapp}
q_\varepsilon^{app}(x)&=\frac1\varepsilon \underline{P}^{-1}(x) e^{-\frac{\psi(x)}{\varepsilon}}+\overline{Q}^0(x)+(\underline{P}^0(x)-\overline{Q}^0(x))e^{-\frac{\psi(x)}{\varepsilon}}+F(x)\frac{\psi(x)}{\varepsilon}e^{-\frac{\psi(x)}{\varepsilon}}\notag\\
&=\overline{Q}^0(x)+(\underline{P}^0(x)-\overline{Q}^0(x))e^{-\frac{\psi(x)}{\varepsilon}}+\frac1\varepsilon(\underline{P}^{-1}(x)+F(x)\psi(x))e^{-\frac{\psi(x)}{\varepsilon}}.
\end{align}
Using the same reasoning as before (by distinguishing what happens far from the interface from what happens near the interface and assuming enough regularity), we obtain that in order to impose the positivity of $q_\varepsilon^{app}$, it is sufficient that $\overline{Q}^0>0$ in $\overline\omega$, where we recall $\overline{Q}^0$ satisfies
\begin{equation}
\begin{dcases}
-\mathrm{div}(\overline{Q}^0 n)=-\nabla\overline{Q}^0\cdot\nabla\psi- \overline{Q}^0 \Delta\psi=b^{-1}\ge 0 & \text{in $\omega$}\\
\text{with suitable BC.} &
\end{dcases}
\end{equation}

\paragraph{Checking estimates in~\eqref{eq:pqNd}}

The wanted estimates on $q_\varepsilon^{app}$ are satisfied. Indeed, at the interface $\partial\mathcal{U}$, $q_\varepsilon^{app}=p_\varepsilon^{app}$ therfeore
$$|q_\varepsilon^{app}(x_0)|\le \frac{C}{\varepsilon},\quad \text{on $\partial\mathcal{U}$}$$
and by regularity, there existe $\delta_0$ such that for $x$ such that $\psi(x)\le\delta_0$ (near the interface),
$$|q_\varepsilon^{app}(x)|\le |q_\varepsilon^{app}(x)-q_\varepsilon^{app}(x_0)|+|q_\varepsilon^{app}(x_0)|\le 1+\frac{C}{\varepsilon} \le \frac{C'}{\varepsilon}.$$

And using the expression of $q_\varepsilon^{app}$~\eqref{eq:Expqapp} and the limits~\eqref{eq:limit}, we obtain far from the interface, for sufficiently small $\varepsilon$,
$$\|q_\varepsilon^{app}\|_{L^\infty(\psi\ge\delta_0)}\le \|\overline{Q}^0\|_{L^\infty(\omega)}+1\le C,$$
as expected.

\paragraph{Summary}

To sum up, we reduced the proof of the existence of $p_\varepsilon$, $q_\varepsilon$ solutions of~\eqref{eq:pqNd} to the existence of a regular enough function $\overline{Q}^0$ in $\omega$ that satisfies
\begin{equation}\label{eq:constraintsQbar0}
\begin{dcases}
\overline{Q}^0>0 & \text{in $\overline{\omega}$}\\
-\mathrm{div}(\overline{Q}^0 n)=-\nabla\overline{Q}^0\cdot\nabla\psi- \overline{Q}^0 \Delta\psi=b^{-1}\ge \eta> 0 & \text{in $\omega$},
\end{dcases}
\end{equation}
or
\begin{equation}
\begin{dcases}
\overline{Q}^0>0 & \text{in $\overline{\omega}$}\\
-\mathrm{div}(\overline{Q}^0 n)=-\nabla\overline{Q}^0\cdot\nabla\psi- \overline{Q}^0 \Delta\psi=0 & \text{in $\omega$}\\
-\Delta \overline{Q}^0+\overline{Q}^0\ge \eta>0 & \text{in $\omega$},
\end{dcases}
\end{equation}
where $\eta$ is a constant. If $\Delta\psi=\mathrm{div}(n)= 0$, then it is sufficient to take $\overline{Q}^0=\eta>0$. This solves the one-dimensional case for example since $\psi(x)=1-x$ in $\omega={]0},1[$, $n=-1$ and $\Delta\psi=0$. In general, however, $\Delta\psi\ne  0$ (spherical symmetric case for instance, see Appendix~\ref{sec:annexspecialcases}) and we will use in the sequel the method of characteristics in order to show the existence of $\overline{Q}^0$ satisfying the constraints~\eqref{eq:constraintsQbar0}.

Using the method of characteristics (see Theorem~\ref{th:hyperbolicregularitybis}), we obtain by imposing $\overline{Q}_0$ on $\partial\mathcal U$ a unique solution $\overline{Q}^0$ to~\eqref{eq:constraintsQbar0} given by
\begin{equation}\label{eq:solQxit}
\overline{Q}^0(X(\xi,t))=\Bigl(\overline{Q}^0(\gamma(\xi))+\int_0^t -b^{-1}(X(\xi,s))e^{\int_0^s \Delta\psi(X(\xi,u))\, du}\, ds\Bigr)e^{-\int_0^t \Delta\psi(X(\xi,s))\, ds}.
\end{equation}

We want $\overline{Q}^0\ge \beta>0$ in $\omega$, and since the characteristics $t\to X(\xi,t)$ fiber all $\omega$ for $\xi\in[0,\Xi]$ and $t\in[0,T_\xi]$ with $T_\xi\le T_f<\infty$ (as explained in Appendix~\ref{sec-append-Charac}), it is sufficient that
\begin{equation}
\forall \xi\in[0,\Xi],\ \forall t\in[0,T_\xi],\quad \overline{Q}^0(X(\xi,t))\ge \beta >0.
\end{equation}
i.e.
\begin{equation}
\forall \xi\in[0,\Xi],\ \forall t\in[0,T_\xi],\quad \overline{Q}^0(\gamma(\xi))\ge \int_0^t b^{-1}(X(\xi,s))e^{\int_0^s \Delta\psi(X(\xi,u))\, du}\, ds +\beta e^{\int_0^t \Delta\psi(X(\xi,s))\, ds}.
\end{equation}

By choosing $b^{-1}=1>0$,
one has, $\forall \xi\in[0,\Xi], \ \forall t\in[0,T_\xi]$,
\begin{align}
\int_0^t e^{\int_0^s \Delta\psi(X(\xi,u))\, du}\, ds+\beta e^{\int_0^t \Delta\psi(X(\xi,s))\, ds}&\le \int_0^{T_\xi} e^{\int_0^s \Delta\psi(X(\xi,u))\,du}\,ds+\beta e^{\int_0^{T_\xi} \Delta\psi(X(\xi,s))\, ds}\\
&\le \int_0^{T_\xi} e^{s \|\Delta\psi\|_{L^\infty}}\,ds+\beta e^{T_\xi \|\Delta\psi\|_{L^\infty}} \\
& \le \Bigl(\frac1{\|\Delta\psi\|_{L^\infty}}+\beta\Bigr)e^{T_\xi\|\Delta\psi\|_{L^\infty} }\\
&\le \Bigl(\frac1{\|\Delta\psi\|_{L^\infty}}+\beta\Bigr)e^{T_f\|\Delta\psi\|_{L^\infty} },
\end{align}
where $\|\Delta\psi\|_{L^\infty}=\max_{(x,y)\in \omega} |\Delta\psi(x,y)|$.

Thus it is sufficient to impose $\overline{Q}^0=cte\ge\Bigl(\frac1{\|\Delta\psi\|_{L^\infty}}+\beta\Bigr)e^{T_f\|\Delta\psi\|_{L^\infty} }$ on $\partial\mathcal U$, in order to achieve $\overline{Q}^0\ge \beta>0$ in $\omega$.

As a summary, the following steps summarize the construction of suitable supersolutions~\eqref{eq:supersoltionsapp} that satisfy~\eqref{eq:pqNd}.
\begin{enumerate}
    \item We fix $b^{-1}=1$ and $T_f$ as in Appendix~\ref{sec-append-Charac}, and solve
    \begin{equation}\label{eq:solQ0Ndsummary}
    \begin{dcases}
        \nabla\overline{Q}^0\cdot n+\Delta\psi \overline{Q}^0=-b^{-1}=-1 & \text{in $\omega$}\\
        \overline{Q}^0(x)=cte\ge\Bigl(\frac1{\|\Delta\psi\|_{L^\infty}}+\beta\Bigr)e^{T_f\|\Delta\psi\|_{L^\infty} } & \text{on $\partial\mathcal{U}$}.
    \end{dcases}
    \end{equation}
    Using Theorem~\ref{th:hyperbolicregularitybis}, the solution $\overline{Q}^0\in C^{k+5}(\overline{\omega})$, and by construction $\overline{Q}^0\ge \beta>0$. Thus~\eqref{eq:constraintsQbar0} is satisfied.
    \item We fix $a^{-1}=1$, and solve
    \begin{equation}\label{eq:solPm1Ndsummary}
    \begin{dcases}
        -\Delta P^{-1}+P^{-1}=a^{-1}=1 & \text{in $\mathcal{U}$}\\
        \frac{\partial P^{-1}}{\partial\nu}=\overline{Q}^{0} & \text{on $\partial\mathcal{U}$}.
    \end{dcases}
\end{equation}
Using the elliptic regularity (Theorem~\ref{th:ellipticregularity}), the solution $P^{-1}\in H^{k+9}(\mathcal{U})\hookrightarrow C^{k+7}(\overline{\mathcal U})$, and by the maximum principle (Corollary~\ref{cor:maximumprinciplebis}), $P^{-1}\ge 1$ in $\overline{\mathcal U}$.
\item $-b^{-1}+\Delta \underline{P}^{-1}- \underline{P}^{-1}+\frac{\partial\overline{Q}^0}{\partial \nu}+2\nabla F(x)\cdot\nu-\Delta\psi \overline{Q}^0$ is continuous on $\partial\mathcal{U}$. We solve
\begin{equation}\label{eq:solP0Ndsummary}
\begin{dcases}
-\Delta P^0+P^0=0 & \text{in $\mathcal U$}\\
\frac{\partial P^0}{\partial\nu}= \Gamma \ge -b^{-1}+\Delta \underline{P}^{-1}- \underline{P}^{-1}+\frac{\partial\overline{Q}^0}{\partial \nu}+2\nabla F(x)\cdot\nu-\Delta\psi \overline{Q}^0& \text{on $\partial\mathcal U$.}
\end{dcases}
\end{equation}
Using the elliptic regularity (Theorem~\ref{th:ellipticregularity}), the solution $P^{0}\in H^{k+9}(\mathcal{U})\hookrightarrow C^{k+7}(\overline{\mathcal U})$.
\item Recalling all the obtained constraints, this allows to define suitable supersolutions~\eqref{eq:supersoltionsapp} $p_\varepsilon^{app}\in C^{k+7}(\overline{\mathcal U})\subset C^2(\overline{\mathcal U})$ and $q_\varepsilon^{app}\in C^{k+5}(\overline{\omega})\subset C^{2}(\overline{\omega})$, that satisfy~\eqref{eq:pqNd}.   
\end{enumerate}

%% file: sections/numerics.tex
In this section, we investigate numerically the convergence of the solution $u_\varepsilon$ of the penalized problem~\eqref{eq:reactiondiffusionpenalized} towards the solution $u$ of the initial problem~\eqref{eq:reactiondiffusion} in the initial domain $\mathcal U$ when $\varepsilon \to 0$. We consider $\Omega={]}0,1[^2$, because in practice we introduce the penalization in order to use a Cartesian mesh. We will consider two situations for the initial domain $\mathcal U$: $\mathcal U$ will be either a disk or a square included into $\Omega$, and we will approximate numerically the penalized solution $u_\varepsilon$ using bivariate finite differences. Note that $\Omega$ does not satisfy the regularity hypotheses of Theorem~\ref{th:convergence} since $\partial\Omega$ is only piecewise regular.

\subsection{Disk domain embedded in a square}

Here, we consider $\mathcal U$ to be a disk of radius $R< \frac12$ and of center $(x_0,y_0)=\bigl(\frac12,\frac12\bigr)$.

\subsubsection{Extensions of the data on the boundary and of the normal vector}

        We choose the following extensions of $\tilde{n}$ and $\tilde{g}$ in $\omega:=\Omega\setminus\overline{\mathcal U}$:
    \begin{itemize}
        \item $n_x(x,y)=\frac{x-x_0}{\sqrt{(x-x_0)^2+(y-y_0)^2}}$, $n_y(x,y)=\frac{y-y_0}{\sqrt{(x-x_0)^2+(y-y_0)^2}}$, 
        \item $g(x,y)=\tilde{g}(x_0+R \cos\theta,y_0+R\sin\theta)$, where $\theta=\arg(x-x_0+i(y-y_0))$.
    \end{itemize}

\subsubsection{Discretization and numerical approximation}\label{sec:discretization}

        The domain $\Omega$ is discretized using a uniform mesh of step $h=1/N$ ($N\in\mathbb{N}$), and whose nodes are the points $(ih,jh)$, $0\le i,j\le N$, see Figure~\ref{fig:scheme2d}. We denote by $U_{ij}$ the approximation of $u_\varepsilon(ih,jh)$. Following the same approach used in one dimension~\cite{Bensiali2014}, a centered scheme is used to approximate the Laplacian of $u_\varepsilon$, whereas the gradient $\nabla u_\varepsilon$ is discretized using a first-order upwind scheme depending on the sign of the velocity $n$.

            \begin{figure}[ht]
            	\centering
            	\includegraphics[scale=0.6]{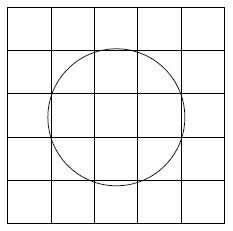}
            	\caption{Cartesian mesh of the fictitious domain $\Omega={]}0,1[^2$.}
            	\label{fig:scheme2d}
            \end{figure}

            After introducing a vector $X_h\in\mathbb{R}^{(N+1)^2}$ such that $U_{ij}=X_{j+(N+1)i}=X_n$~\cite{ciarlet1998introduction}, with $i_n=\bigl\lfloor\frac{n}{N+1}\bigr\rfloor$, $j_n=n-(N+1)i_n$, we obtain the following linear system \begin{equation}\label{eq:linearsystem}
        A_hX_h=B_h,
    \end{equation}
    where 
 \begin{itemize}
 \item if $i_n=0$ or $N$, or $j_n=0$ or $N$ then\\
 $A_{n,n}=1 \text{ and } A_{n,m}=0 \text{ for $m\ne n$}, \text{ and } B_n=0 \text{ (Dirichlet boundary conditions)}$
 \item else, the nonzero entries of the matrix $A_h$ are summarized in Table~\ref{tab:matrix1},
 \item if $i_n\notin \{0,N\}$ and $j_n\notin \{0,N\}$ then\\
 $B_n=(1-\chi_n)f(x_n, y_n)+\frac{\chi_n}{\varepsilon} g(x_n, y_n))$.
 \end{itemize}
  We used the following notations:
 \begin{itemize}
\item $x_n=i h$ and $y_n=j h$ where $n=(N+1)i+j$
\item $F_n=F(x_n,y_n)$ for a function $F\colon \RR^2\to\RR$.
 \end{itemize}

      \begin{table}[!h]
\centering
\begin{tabular}{c|cccc}  
\toprule
$i_n$  & $\le N/2$ & $> N/2$ & $\le N/2$ & $>N/2$ \\
$j_n$   & $\le N/2$ & $\le N/2$ & $> N/2$ & $>N/2$ \\
\midrule
$ \implies (n_x)_n$  & $\le 0$ & $\ge 0$ & $\le 0$ & $\ge 0$ \\
$ \implies (n_y)_n$   & $\le 0$ & $\le 0$ & $\ge 0$ & $\ge 0$ \\
\midrule
$A_{n,n-N-1}$     &    $-\dfrac{1}{h^2}$      &  $-\dfrac{1}{h^2}-\dfrac{\chi_n (n_x)_n}{\varepsilon h}$   & $-\dfrac{1}{h^2}$ & $-\dfrac{1}{h^2}-\dfrac{\chi_n (n_x)_n}{\varepsilon h}$   \\
$A_{n,n-1}$       &  $-\dfrac{1}{h^2}$    & $-\dfrac{1}{h^2}$  & $-\dfrac{1}{h^2}-\dfrac{\chi_n (n_y)_n}{\varepsilon h}$& $-\dfrac{1}{h^2}-\dfrac{\chi_n (n_y)_n}{\varepsilon h}$  \\
$A_{n,n}$       &    \multicolumn{4}{c}{$\dfrac4{h^2}+1+\dfrac{\chi_n}{\varepsilon h}(|(n_x)_n|+|(n_y)_n|)+\dfrac{\alpha\chi_n}{\varepsilon}$}  \\
$A_{n,n+1}$       &  $-\dfrac{1}{h^2}+\dfrac{\chi_n (n_y)_n}{\varepsilon h}$    & $-\dfrac{1}{h^2}+\dfrac{\chi_n (n_y)_n}{\varepsilon h}$  & $-\dfrac{1}{h^2}$& $-\dfrac{1}{h^2}$  \\
$A_{n,n+N+1}$       &  $-\dfrac{1}{h^2}+\dfrac{\chi_n (n_x)_n}{\varepsilon h}$    & $-\dfrac{1}{h^2}$  & $-\dfrac{1}{h^2}+\dfrac{\chi_n (n_x)_n}{\varepsilon h}$& $-\dfrac{1}{h^2}$  \\
\bottomrule
\end{tabular}
\caption{Nonzero entries of the matrix $A_h$ of the finite difference scheme~\eqref{eq:linearsystem}.}
\label{tab:matrix1}
\end{table}

 The study carried out for the numerical approximation in one dimension~\cite{Bensiali2014} can be adapted and we obtain the following
 
 \begin{lemma}[Existence and uniqueness]
 For all $\varepsilon,h>0$, the system~\eqref{eq:linearsystem} has a unique solution.
 \end{lemma}
 
 \begin{proof}
 One can check that $A_h$ is a strictly diagonally dominant matrix, hence invertible by the Hadamard lemma.
 \end{proof}
 
 \begin{lemma}[Stability]\label{lem:stability}
 For all $\varepsilon,h>0$, the finite difference scheme~\eqref{eq:linearsystem} is stable in the norm $\|\cdot\|_\infty$ with $\|A_h^{-1}\|_\infty\le 1$.
 \end{lemma}
 
 \begin{proof}
 The result follows from the fact that $A_h$ is strictly diagonally dominant with ${|A_{nn}|-\sum_{m\ne n}|A_{nm}|\ge 1>0}$.
 \end{proof}

\begin{remark}
Given $\varepsilon>0$, if $u_\varepsilon$ is regular enough (which is not guaranteed), then we can recover classically the results of consistency and convergence for the finite difference scheme. We found a first-order convergence with respect to $h$ in the presented case.

\end{remark}

\subsubsection{Condition number}

 The penalization introduces terms of order $\varepsilon^{-1}$ in the matrix $A_h$, which leads to an ill-conditioned matrix problem for small values of the penalization parameter $\varepsilon$. Indeed, the condition number in the infinity norm is
     \begin{align}
         \kappa(A_h)=\|A_h\|_\infty\cdot\|A_h^{-1}\|_\infty\le \|A_h\|_\infty,
     \end{align}
     using the stability already established (Lemma~\ref{lem:stability}). On the other hand
     \begin{align}
          \sum_m |A_{nm}|&\le \frac{8}{h^2}+1+\frac{\chi_n}{\varepsilon}\alpha+2\frac{\chi_n}{\varepsilon h}(|(n_x)_n|+|(n_y)_n|)
          \le \frac8{h^2}+1+\frac\alpha\varepsilon+\frac{4}{R\varepsilon h},
     \end{align}
     thus
     \begin{align}
         \|A_h\|_\infty&=\max_{n} \Bigl(\sum_m |A_{nm}|\Bigr)
         \le \frac{1}{\varepsilon}\Bigl(\alpha+\frac{4}{Rh}\Bigr)+\Bigl(1+\frac8{h^2}\Bigr).
     \end{align}
     
     For a fixed space step $h$, one thus has
     \begin{equation}
         \kappa(A_h)\underset{\varepsilon\to0}{=}O(\varepsilon^{-1}),
     \end{equation}
     which holds in every matrix norm by the equivalence of norms in finite dimension (fixed space step $h$).

    Applying a diagonal preconditioner, $C=\mathrm{diag}(A_h)$, allows to obtain a condition number almost independent of $\varepsilon$ as shown in Figure~\ref{fig:conditionnumber}. 

        \begin{figure}[h!]
        	\centering
        	\includegraphics[scale=0.4]{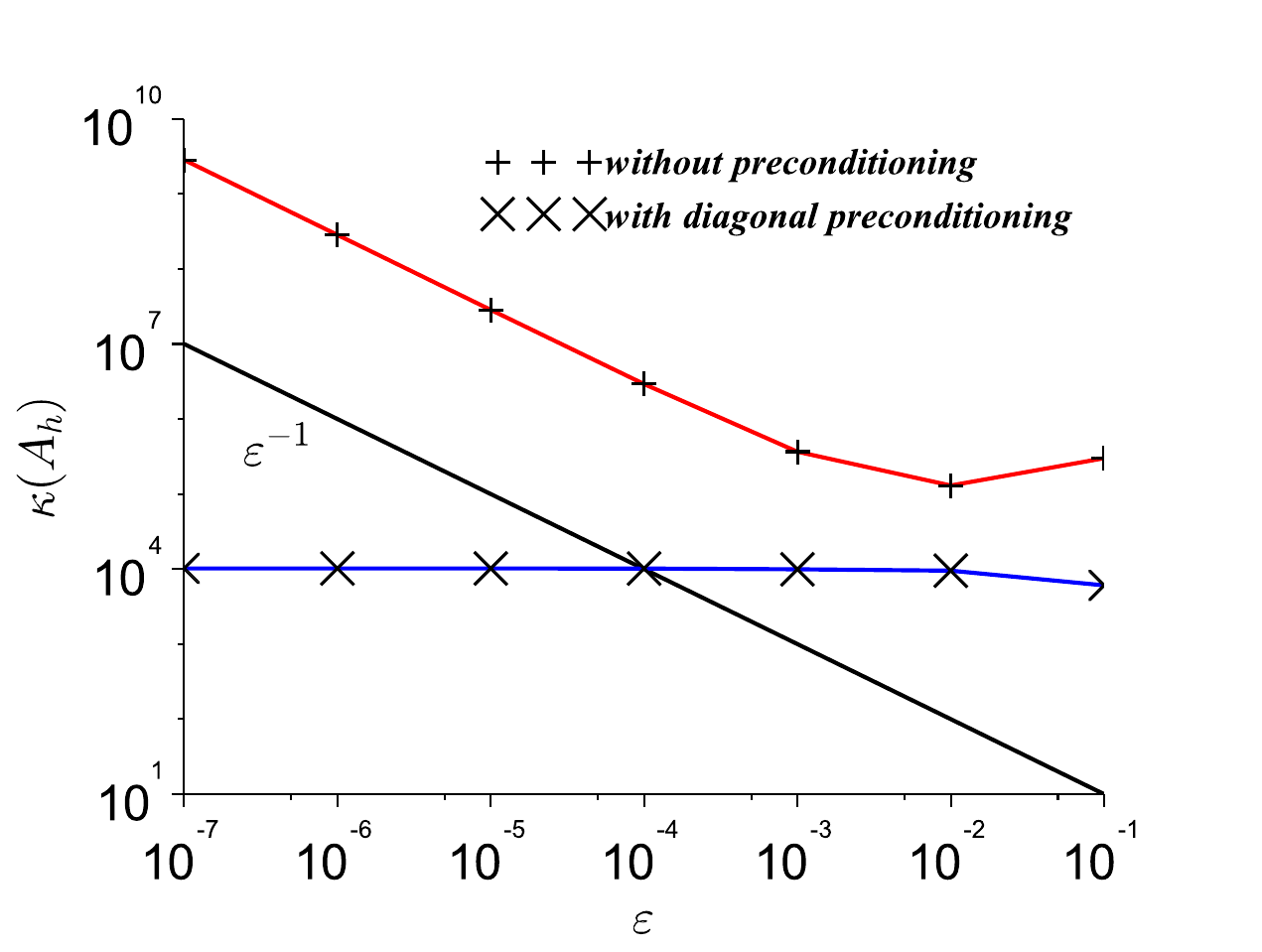}
        	\caption{Condition number $\kappa(A_h)$ in the 2-norm versus $\varepsilon$ for $\alpha = 2$, $R=0.3$ and $N = 100$.}
        	\label{fig:conditionnumber}
        \end{figure}

\subsubsection{Numerical tests}

  We implemented the presented numerical scheme using Scilab software.
  We consider the following test example:
    \begin{itemize}
        \item $u(x,y)=\sin(c(x+y-1))$ the exact solution of the initial problem~\eqref{eq:reactiondiffusion} with
        \item $f(x,y)=(2c^2+1)\sin(c(x+y-1))$ in $\mathcal U$,
        \item $\tilde{g}=c\cos(c(x+y-1))\frac{x+y-x_0-y_0}{\sqrt{(x-x_0)^2+(y-y_0)^2}}+\alpha \sin(c(x+y-1))$ on $\partial\mathcal U$.
    \end{itemize}
    The numerical solution of the penalized problem~\eqref{eq:reactiondiffusionpenalized} using the finite difference scheme~\eqref{eq:linearsystem} is compared to the exact solution $u$ of the initial problem~\eqref{eq:reactiondiffusion} inside $\mathcal U$ for $c=5$, $\alpha=2$ (Robin boundary conditions) and $R=0.3$. We consider the following errors:
    \begin{itemize}
        \item $L^2$-norm error: 
        $$\mathsf{Er_{L^2}}=\sum_{i=0}^N \sum_{j=0}^N h^2 (u(ih,jh)-U_{ij})^2 \times (1-\chi(ih,jh)),$$
        \item $L^\infty$-norm error: 
        $$\mathsf{Er_{L^\infty}}=\max\limits_{0\le i,j\le N} |u(ih,jh)-U_{ij}|\times (1-\chi(ih,jh)),$$    
        \item $H^1$-norm error: 
        $$\mathsf{Er_{H^1}}=\sum_{i=0}^N \sum_{j=0}^N h^2 \Bigl[(u(ih,jh)-U_{ij})^2+\Bigl(\frac{\partial u}{\partial x}(ih,jh)-U^x_{ij}\Bigr)^2+\Bigl(\frac{\partial u}{\partial y}(ih,jh)-U^y_{ij}\Bigr)^2\Bigr] \times (1-\chi(ih,jh)),$$    
        \end{itemize}
where $U^x_{ij}$ and $U^y_{ij}$ are second-order finite difference approximations of $u_\varepsilon (ih,jh)\approx U_{ij}$.

        Figures~\ref{fig:results2d} and~\ref{fig:results2detah}, and Table~\ref{tab:erroretah} show the results. We can make the following remarks:
        \begin{itemize}
            \item The numerical results suggest the convergence of $u_\varepsilon$ towards $u$ inside $\mathcal U$ when $\varepsilon$ goes to $0$. The order of convergence is close to $1$ as expected but one needs to increase the number of points $N$ or use a higher-order scheme in order to make it more clear (see Section~\ref{sec:squareinsquare}). A boundary layer is visible in the neighbourhood of $\partial\Omega$ as expected.
            \item An order of convergence of $1$ is observed with respect to $h=\frac1{N}$ for $\varepsilon$ small enough, which is consistent with the chosen discretization. 
        \end{itemize}

            \begin{figure}[h!]
            	\centering
            	\centering
            	\begin{subfigure}[b]{0.48\textwidth}
            		\centering
            		\includegraphics[width=\textwidth]{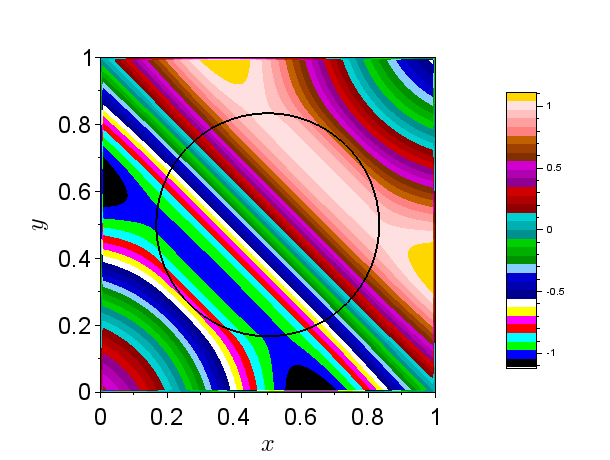}
            		\caption{Numerical solution}
            		\label{fig:numsoldisk}
            	\end{subfigure}
            	\hfill
            	\begin{subfigure}[b]{0.48\textwidth}
            		\centering
            		\includegraphics[width=\textwidth]{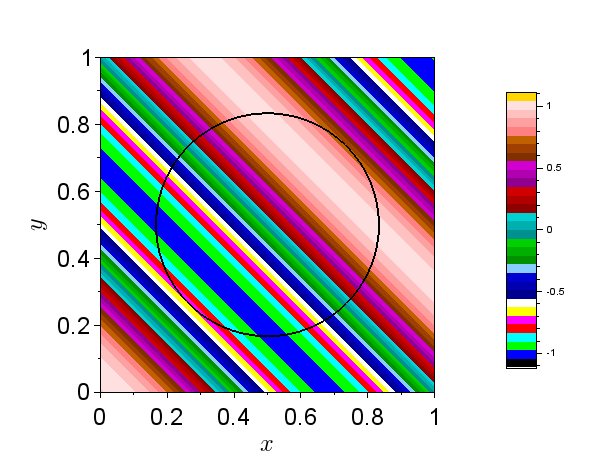}
            		\caption{Exact solution}
            		\label{fig:exactsoldisk}
            	\end{subfigure}
            	\caption{Comparison between the numerical solution of the penalized problem and the exact solution of the initial problem for $\alpha = 2$, $\varepsilon=10^{-10}$ and $N = 150$.}
            	\label{fig:results2d}
            \end{figure}
            
            \begin{figure}[!h]
            	\centering
            	\begin{subfigure}[b]{0.45\textwidth}
            		\centering
            		\includegraphics[width=\textwidth]{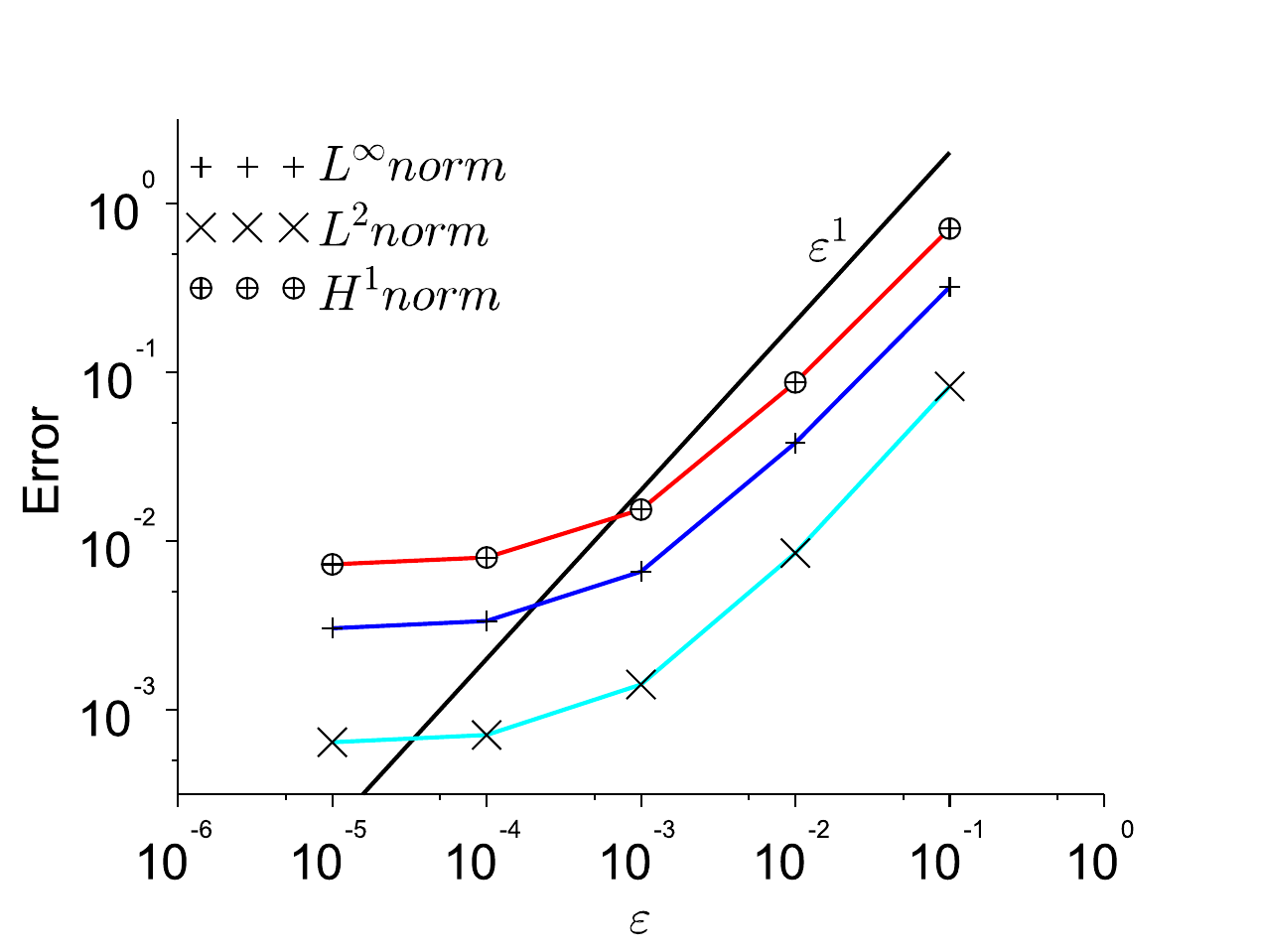}
            		\caption{}
            		\label{fig:erroretacercle}
            	\end{subfigure}
            	\hfill
            	\begin{subfigure}[b]{0.45\textwidth}
            		\centering
            		\includegraphics[width=\textwidth]{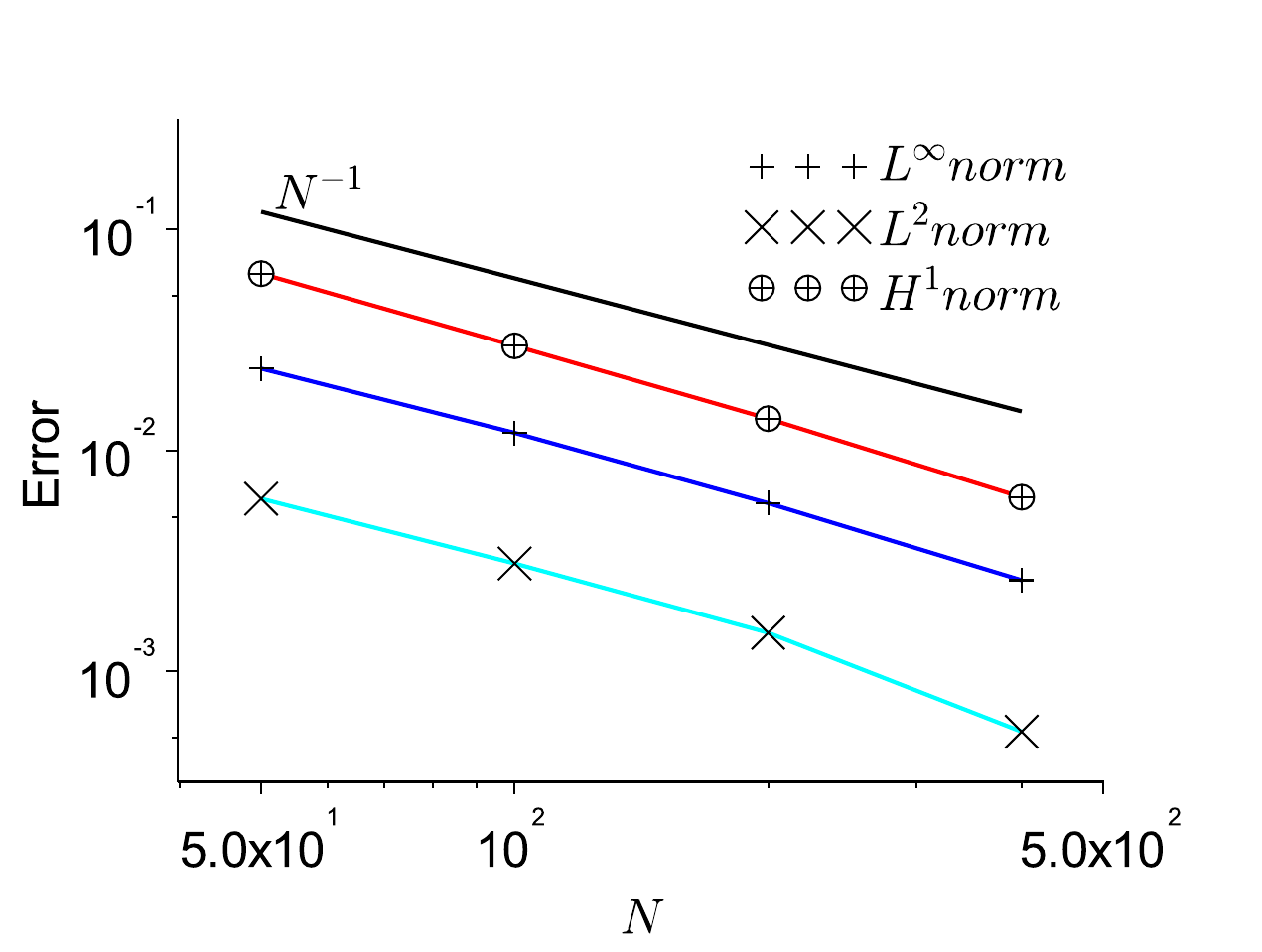}
            		\caption{}
            		\label{fig:errorNcercle}
            	\end{subfigure}
            	\caption{Disk embedded in a square. Error (inside the fluid domain) between the numerical solution of the penalized problem and the exact solution of the initial problem with respect to: (a)  $\varepsilon$ for $N=350$, (b) $N$ for $\varepsilon=10^{-10}$.}
            	\label{fig:results2detah}
            \end{figure}

\begin{table}[!h]
    \begin{subtable}[h]{0.48\textwidth}
        \centering
        \begin{tabular}{cccc}  
\toprule
$\varepsilon$   & $L^\infty$ & $L^2$ & $H^1$ \\
\midrule
$10^{-1}$      & --    & -- &  --    \\
$10^{-2}$    &    0.923      &  0.987   & 0.911   \\
$10^{-3}$       &  0.762    & 0.778 & 0.753   \\
\bottomrule
\end{tabular}
       \caption{}
       \label{tab:erroretacercle}
    \end{subtable}
    \hfill
    \begin{subtable}[h]{0.48\textwidth}
        \centering
        \begin{tabular}{cccc}  
\toprule
$N$   & $L^\infty$ & $L^2$ & $H^1$ \\
\midrule
50      & --    & -- &  --    \\
100     &    -0.966      &  -0.972   & -1.084   \\
200       &  -1.057    & -1.042 & -1.095   \\
400       &  -1.158    & -1.485  & -1.181 \\
\bottomrule
\end{tabular}
        \caption{}
        \label{tab:errorNcercle}
     \end{subtable}
     \caption{Disk embedded in a square. Convergence order with respect to: (a) the penalization parameter $\varepsilon$ for $N=350$, (b) the mesh size $h=\frac1N$ for $\varepsilon=10^{-10}$.}
     \label{tab:erroretah}
\end{table}

\subsection{Square domain embedded in a square}\label{sec:squareinsquare}

In order to investigate more precisely the order of convergence with respect to the penalization parameter, we consider here an initial domain which is a square $\mathcal{U}=\bigl{]}\frac12-R,\frac12+R\bigr[^2$ where $0<R<\frac12$, that we embed in a larger square $\Omega={]}0,1[^2$. In this way, we avoid the error related to the geometric representation of the boundary of the initial domain $ \mathcal U$ when using a Cartesian mesh with finite differences for the penalized problem.

\subsubsection{Extensions of the data on the boundary and of the normal vector}

In the case of a square $\mathcal U$ embedded in a larger square $\Omega$, natural extensions of the data on the boundary $\tilde{g}\in H^{1/2}(\partial\mathcal{U})$ and $\tilde{n}\in H^{1/2}(\partial\mathcal{U})$ to the obstacle domain $\omega=\Omega\setminus\overline{\mathcal{U}}$ are represented on Fig.~\ref{fig:extension}. We will consider in the tests continuous $\tilde{g}$ on each side of the square, so that these choices guarantee that the extensions $n\in H^1(\omega)$ and $g\in H^1(\omega)$: the only eventual points of discontinuity are the corners of the smaller square.


\begin{figure}[h]
	\centering
	\begin{subfigure}[b]{0.48\textwidth}
		\centering
		\includegraphics[width=\textwidth]{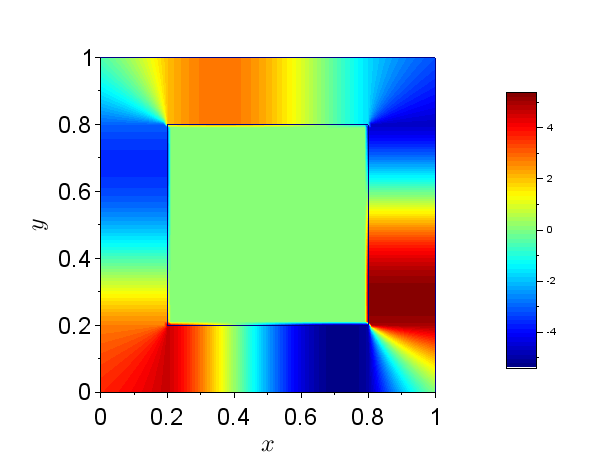}
		\caption{Extension of $\tilde{g}$}
		\label{fig:extensiong}
	\end{subfigure}
	\hfill
	\begin{subfigure}[b]{0.48\textwidth}
		\centering
		\includegraphics[width=\textwidth]{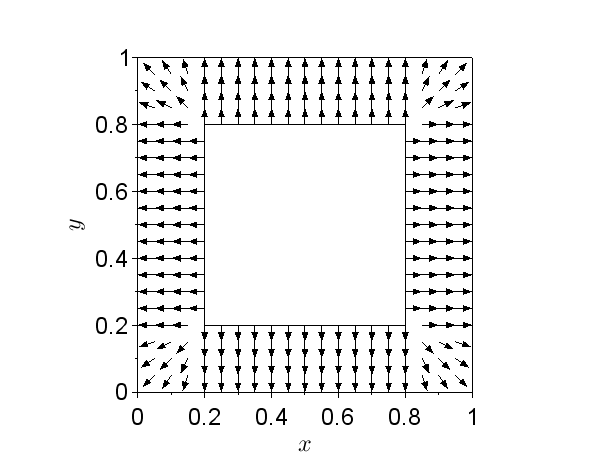}
		\caption{Extension of $\tilde{n}$}
		\label{fig:extensionn}
	\end{subfigure}
	\caption{Natural extensions of the data on the boundary and the normal vector to the obstacle domain.}
	\label{fig:extension}
\end{figure}

\subsubsection{Second-order upwind finite difference scheme}

As for the disk case, we use upwind finite differences adapted to advection-dominated problems on a uniform grid (see Section~\ref{sec:discretization}). We use here second-order centered finite differences for the Laplacian term and second-order upwind finite differences for the advection terms, namely:
\begin{equation}\label{eq:laplacian}
    \Delta u_\varepsilon(ih, jh)\approx \frac{U_{i+1,j}-2U_{i,j}+U_{i-1,j}}{h^2}+\frac{U_{i,j+1}-2U_{i,j}+U_{i,j-1}}{h^2},
\end{equation}
\begin{equation}\label{eq:dudx}
\frac{\partial u_\varepsilon}{\partial x}(ih, jh)\approx
\begin{dcases}
\frac{3U_{i,j}-4U_{i-1,j}+U_{i-2,j}}{2h} & \text{if $n_x(ih,jh) \ge 0$}\\
\frac{-U_{i+2,j}+4U_{i+1,j}-3U_{i,j}}{2h} &\text{if $n_x(ih,jh) < 0$},
\end{dcases}
\end{equation}
and
\begin{equation}\label{eq:dudy}
\frac{\partial u_\varepsilon}{\partial y}(ih, jh)\approx
\begin{dcases}
\frac{3U_{i,j}-4U_{i,j-1}+U_{i,j-2}}{2h} & \text{if $n_y(ih,jh) \ge 0$}\\
\frac{-U_{i,j+2}+4U_{i,j+1}-3U_{i,j}}{2h} &\text{if $n_y(ih,jh) < 0$}.
\end{dcases}
\end{equation}
            After introducing, as for the disk case, a vector $X_h\in\mathbb{R}^{(N+1)^2}$ such that $U_{ij}=X_{j+(N+1)i}=X_n$, we obtain the following linear system
    \begin{equation}\label{eq:linearsystemsquare}
        A_hX_h=B_h,
    \end{equation}
    where

 \begin{itemize}
 \item if $x_n=0$ or $1$, or $y_n=0$ or $1$ then\\
 $A_{n,n}=1 \text{ and } A_{n,m}=0 \text{ for $m\ne n$}, \text{ and } B_n=0 \text{ (Dirichlet boundary conditions)}$
 \item else the nonzero entries of the matrix $A_h$ are summarized in Table~\ref{tab:matrix2},
 \item if $x_n\notin \{0,1\}$ and $x_n\notin \{0,1\}$ then\\
 $B_n=(1-\chi_n)f(x_n, y_n)+\frac{\chi_n}{\varepsilon} g(x_n, y_n)$,
 \end{itemize}
  using the same notations as for the disk case:
 \begin{itemize}
\item $x_n=i h$ and $y_n=j h$ where $n=(N+1)i+j$
\item $F_n=F(x_n,y_n)$ for a function $F\colon \RR^2\to\RR$.
 \end{itemize}
    
\begin{table}[!h]
\centering
\begin{tabular}{c|cccc}  
\toprule
$x_n$  & $\ge0.5+R$ & $\le 0.5-R$ & $> 0.5-R$ & else \\
$y_n$   & $>0.5-R$ & $< 0.5+R$ & $\le 0.5-R$ & else \\
\midrule
$ \implies (n_x)_n$  & $\ge 0$ & $\le 0$ & $\ge 0$ & $\le 0$ \\
$ \implies (n_y)_n$   & $\ge 0$ & $\le 0$ & $\le 0$ & $\ge 0$ \\
\midrule
$A_{n,n-2N-2}$      & $\dfrac{\chi_n (n_x)_n}{2\varepsilon h}$    & $0$ &  $\dfrac{\chi_n (n_x)_n}{2\varepsilon h}$ & $0$  \\
$A_{n,n-N-1}$     &    $-\dfrac{1}{h^2}-\dfrac{2\chi_n (n_x)_n}{\varepsilon h}$      &  $-\dfrac{1}{h^2}$   & $-\dfrac{1}{h^2}-\dfrac{2\chi_n (n_x)_n}{\varepsilon h}$ & $-\dfrac{1}{h^2}$   \\
$A_{n,n-2}$       &  $\dfrac{\chi_n (n_y)_n}{2\varepsilon h}$    & $0$ & $0$ & $\dfrac{\chi_n (n_y)_n}{2\varepsilon h}$   \\
$A_{n,n-1}$       &  $-\dfrac{1}{h^2}-\dfrac{2\chi_n (n_y)_n}{\varepsilon h}$    & $-\dfrac{1}{h^2}$  & $-\dfrac{1}{h^2}$& $-\frac{1}{h^2}-\dfrac{2\chi_n (n_y)_n}{\varepsilon h}$  \\
$A_{n,n}$       &    \multicolumn{4}{c}{$\dfrac4{h^2}+1+\dfrac{3\chi_n}{2\varepsilon h}(|(n_x)_n|+|(n_y)_n|)+\dfrac{\alpha\chi_n}{\varepsilon}$}  \\
$A_{n,n+1}$       &  $-\dfrac{1}{h^2}$    & $-\dfrac{1}{h^2}+\dfrac{2\chi_n (n_y)_n}{\varepsilon h}$  & $-\dfrac{1}{h^2}+\dfrac{2\chi_n (n_y)_n}{\varepsilon h}$& $-\frac{1}{h^2}$  \\
$A_{n,n+2}$       &  $0$    & $-\dfrac{\chi_n (n_y)_n}{2\varepsilon h}$  & $-\dfrac{\chi_n (n_y)_n}{2\varepsilon h}$& $0$  \\
$A_{n,n+N+1}$       &  $-\dfrac{1}{h^2}$    & $-\dfrac{1}{h^2}+\dfrac{2\chi_n (n_x)_n}{\varepsilon h}$  & $-\dfrac{1}{h^2}$& $-\dfrac{1}{h^2}+\frac{2\chi_n (n_x)_n}{\varepsilon h}$  \\
$A_{n,n+2N+2}$       &  $0$    & $-\dfrac{\chi_n (n_x)_n}{2\varepsilon h}$  & $0$& $-\dfrac{\chi_n (n_x)_n}{2\varepsilon h}$  \\
\bottomrule
\end{tabular}
\caption{Nonzero entries of the matrix $A_h$ of the finite difference scheme~\eqref{eq:linearsystemsquare}.}
\label{tab:matrix2}
\end{table}

 \paragraph*{Special treatment of discontinuities:}
 We recall that
 \begin{equation}
 \chi(x,y)=\chi_{\mathcal{\omega}}(x,y)
 \end{equation}
is the characteristic function of $\omega=\Omega\in\overline{\mathcal U}$ and hence is a discontinuous function, and that $n$ and $g$ are extensions to $\omega$ of the data on the boundary of the initial domain, and have also discontinuities. Since we are using finite differences, we have to handle these discontinuities by making suitable choices for the values at discontinuity points. Following the numerical tests carried out in the one-dimensional case~\cite{Bensiali2014} where second order convergence was recovered when we impose $\chi$ to be $1$ on the boundary of the initial domain, we do the same in the two-dimensional case by forcing $\chi$ to be $1$ on the boundary of $\mathcal{U}$. Also, special treatment has to be done for the corners of the initial square since the extensions $n$ and $g$ have discontinuities at these points.  

Consider for example the upper left corner $(x_c,y_c)=(0.5-R,0.5+R)$. At this point, the continuous equation reads
\begin{align*}
    -\Delta u_\varepsilon(x_c,y_c) + u_\varepsilon (x_c,y_c)+ \frac{\chi(x_c,y_c)}\varepsilon \Bigl(\frac{\partial u_\varepsilon}{\partial x}(x_c,y_c) n_x (x_c,y_c)+\frac{\partial u_\varepsilon}{\partial y}(x_c,y_c) n_y (x_c,y_c)+\alpha u_\varepsilon (x_c,y_c)\Bigr)\\
    \quad  =(1-\chi(x_c,y_c)) f(x_c,y_c)+ \frac{\chi(x_c,y_c)}\varepsilon g(x_c,y_c)
\end{align*}
thus, if we impose $\chi(x_c,y_c)=1$,
\begin{align*}
    -\Delta u_\varepsilon(x_c,y_c) + u_\varepsilon (x_c,y_c)+ \frac{1}\varepsilon \Bigl(\frac{\partial u_\varepsilon}{\partial x}(x_c,y_c) n_x (x_c,y_c)+\frac{\partial u_\varepsilon}{\partial y}(x_c,y_c) n_y (x_c,y_c)+\alpha u_\varepsilon (x_c,y_c)\Bigr)\\
    \quad  =\frac{1}\varepsilon g(x_c,y_c).
\end{align*}
Since the aim of penalization is to recover the boundary conditions of the initial problem for $\varepsilon$ small enough, we propose to choose $n_x$, $n_y$ and $g$ at the corner so that the equality
\begin{equation}\label{eq:condcorners}
    \frac{\partial u}{\partial x}(x_c,y_c) n_x (x_c,y_c)+\frac{\partial u}{\partial y}(x_c,y_c) n_y (x_c,y_c)+\alpha u (x_c,y_c)=g(x_c,y_c)
\end{equation}
is satisfied for the exact solution $u$. For sufficiently smooth $u\in C^2([0.5-R,0.5+R]^2)$ we have the following equations satisfied pointwise on the sides of the square
\begin{equation}
    \begin{dcases}
    \frac{\partial u}{\partial x}(x,y_c) n_x (x,y_c)+\frac{\partial u}{\partial y}(x,y_c) n_y (x,y_c)+\alpha u (x,y_c)=g(x,y_c) & x\in {]}x_c,x_c+2R[\\
    \frac{\partial u}{\partial x}(x_c,y) n_x (x_c,y)+\frac{\partial u}{\partial y}(x_c,y) n_y (x_c,y)+\alpha u (x_c,y)=g(x_c,y) & y\in {]}y_c-2R,y_c[.
    \end{dcases}
\end{equation}
Thus, it is sufficient for example to take
\begin{equation}
\begin{dcases}
n_x(x_c,y_c)=\lim\limits_{x\to x_c^+} n_x(x,y_c)\\
n_y(x_c,y_c)=\lim\limits_{x\to x_c^+} n_y(x,y_c)\\
g(x_c,y_c)=\lim\limits_{x\to x_c^+} g(x,y_c)
\end{dcases}
\end{equation}
in order to satisfy~\eqref{eq:condcorners}. Other choices are also possible by taking instead $\lim\limits_{y\to y_c^-}$, or even the sum of the two equations and then the mean:
\begin{equation}
\begin{dcases}
n_x(x_c,y_c)=\frac12 \Bigl(\lim\limits_{x\to x_c^+} n_x(x,y_c)+\lim\limits_{y\to y_c^-} n_x(x_c,y)\Bigr)\\
n_y(x_c,y_c)=\frac12 \Bigl(\lim\limits_{x\to x_c^+} n_y(x,y_c)+\lim\limits_{y\to y_c^-} n_y(x_c,y)\Bigr)\\
g(x_c,y_c)=\frac12 \Bigl(\lim\limits_{x\to x_c^+} g(x,y_c)+\lim\limits_{y\to y_c^-} g(x_c,y)\Bigr).
\end{dcases}
\end{equation}

\subsubsection{Numerical tests}

We consider the following test example, from a manufactured solution:
\begin{itemize}
    \item $u(x,y)=\sin(5(x+y-1))$ the exact solution of the initial problem
    \item $f$ and $\tilde{g}$ accordingly: $f=-\Delta u +u$ in $\mathcal U$, $\tilde{g}=\nabla u\cdot \tilde{n}+\alpha u$ on $\partial\mathcal U$.
\end{itemize}
In the following tests, we will consider $\alpha=2$ and $R=0.3$.

\paragraph{Error inside the fluid domain}

We compare the numerical solution of the penalized problem using finite differences to the exact solution of the initial problem inside the fluid domain. The results are reported on Figures~\ref{fig:solution} and \ref{fig:error}, and Tables~\ref{tab:errorN} and \ref{tab:erroreta}. 

We consider here the following errors:
    \begin{itemize}
        \item $L^2$-norm error: 
        $$\mathsf{Er_{L^2}}=\sum_{i=0}^N \sum_{j=0}^N h^2 (u(ih,jh)-U_{ij})^2 \times \mathds{1}_C(ih,jh),$$
        \item $L^\infty$-norm error: 
        $$\mathsf{Er_{L^\infty}}=\max\limits_{0\le i,j\le N} |u(ih,jh)-U_{ij}|\times \mathds{1}_C(ih,jh),$$    
                \item $H^1$-norm error: 
        $$\mathsf{Er_{H^1}}=\sum_{i=0}^N \sum_{j=0}^N h^2 \Bigl[(u(ih,jh)-U_{ij})^2+\Bigl(\frac{\partial u}{\partial x}(ih,jh)-U^x_{ij}\Bigr)^2+\Bigl(\frac{\partial u}{\partial y}(ih,jh)-U^y_{ij}\Bigr)^2\Bigr] \times \mathds{1}_C(ih,jh),$$    
        \end{itemize}
where $U^x_{ij}$ and $U^y_{ij}$ are second-order finite difference approximations of $u_\varepsilon (ih,jh)\approx U_{ij}$ (decentered at the boundary of $\mathcal{U}$). In the tables, the errors correspond to $C=[0.5-R,0.5+R]^2$, errors (no boundary) correspond to $C=]0.5-R,0.5+R[^2$, while the interior errors correspond to $C=[0.5-R+S,0.5+R-S]^2$ where $0<S<R$. Fox example, for $R=0.3$, we took $S=0.1$.

\begin{figure}[h]
	\centering
	\begin{subfigure}[b]{0.48\textwidth}
		\centering
		\includegraphics[width=\textwidth]{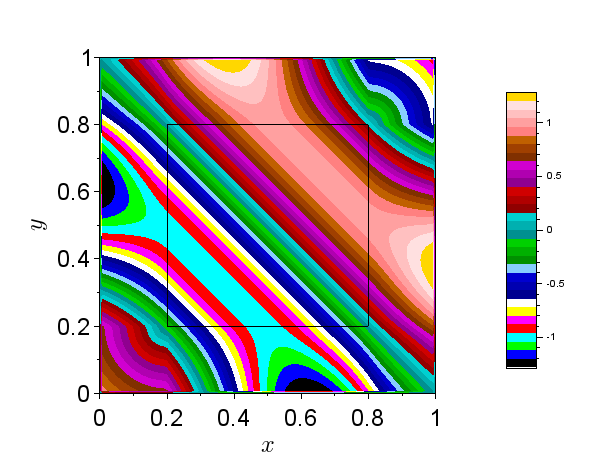}
		\caption{Numerical solution}
		\label{fig:numsol}
	\end{subfigure}
	\hfill
	\begin{subfigure}[b]{0.48\textwidth}
		\centering
		\includegraphics[width=\textwidth]{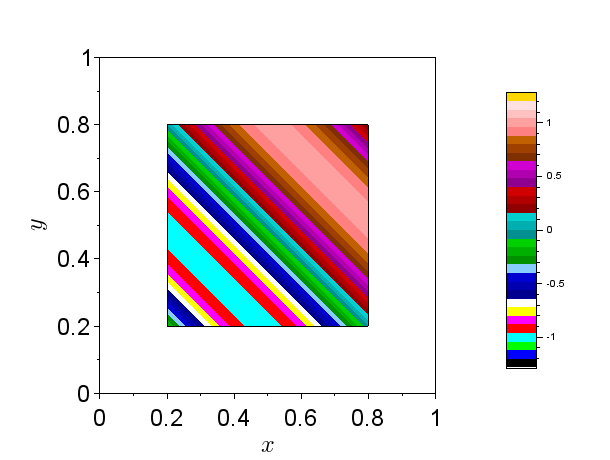}
		\caption{Exact solution}
		\label{fig:exactsol}
	\end{subfigure}
	\caption{Comparison between the numerical solution of the penalized problem and thee exact solution of the initial problem for $\varepsilon=10^{-10}$ and $N=150$.}
	\label{fig:solution}
\end{figure}

\begin{figure}[!h]
	\centering
	\begin{subfigure}[b]{0.45\textwidth}
		\centering
		\includegraphics[width=\textwidth]{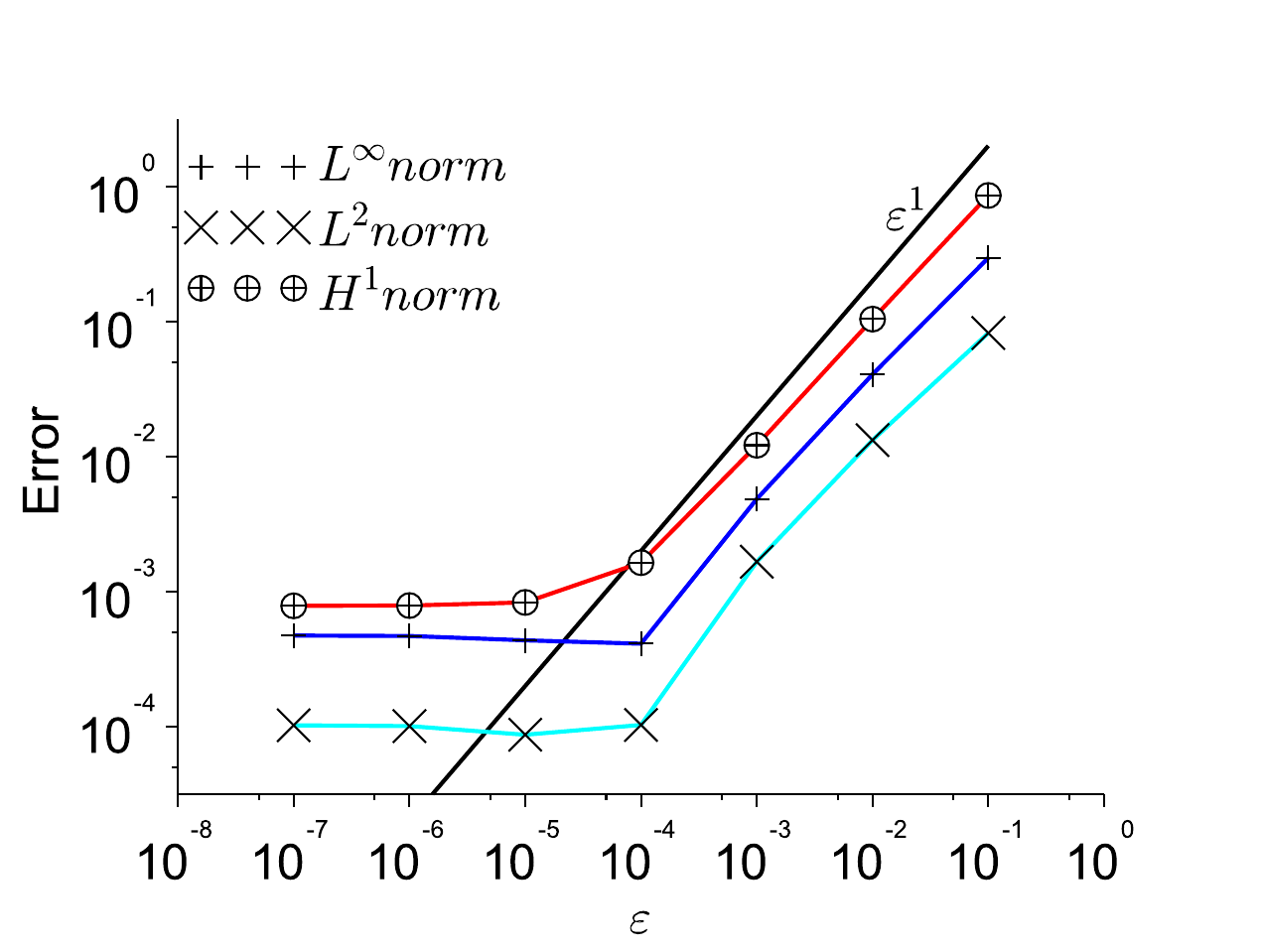}
		\caption{}
		\label{fig:erroreta}
	\end{subfigure}
	\hfill
	\begin{subfigure}[b]{0.45\textwidth}
		\centering
		\includegraphics[width=\textwidth]{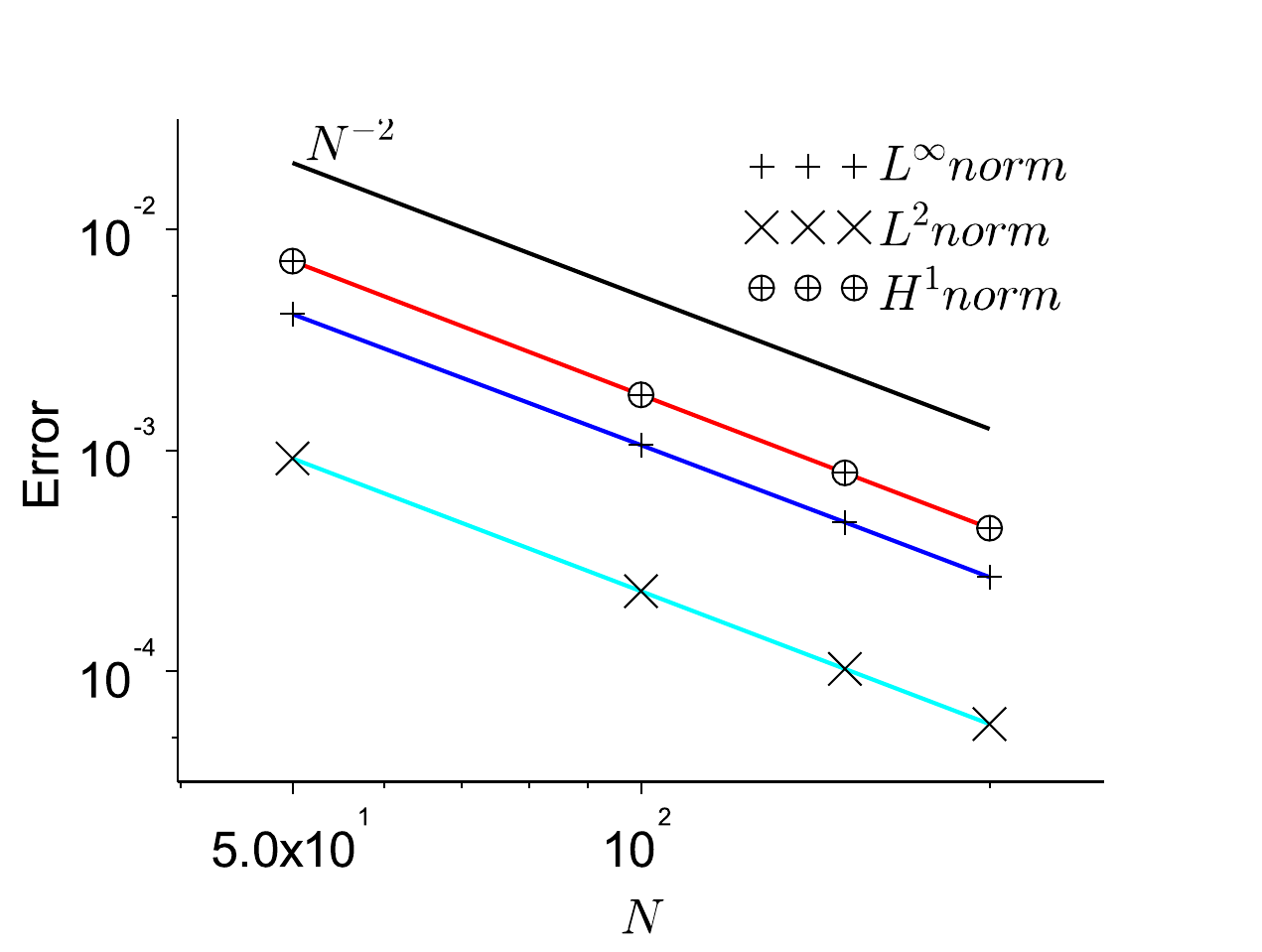}
		\caption{}
		\label{fig:errorN}
	\end{subfigure}
	\caption{Square embedded in a square. Error (inside the fluid domain) between the numerical solution of the penalized problem and the exact solution of the initial problem with respect to: (a)  $\varepsilon$ for $N=150$, (b) $N$ for $\varepsilon=10^{-10}$.}
	\label{fig:error}
\end{figure}

\begin{table}[!h]
\hspace{-1cm}
\centering
\begin{tabular}{cccccccccc}  
\toprule
$N$   & $L^\infty$ & $L^2$ & $H^1$ & $L^\infty$ (no & $L^2$ (no  & $H^1$ (no  & $L^\infty $ & $L^2$  & $H^1$\\
 & & & & boundary) & boundary) & boundary) & (interior) & (interior) & (interior) \\
\midrule
50      & --    & -- &  -- & & & & -- & -- & --   \\
100     &    -1.969      &  -1.995   & -2.010 & -1.908 & -1.899 & -1.939  & -1.872 & -1.873 & -1.919  \\
150       &  -1.982    & -1.997 & -2.005 & -1.949 & -1.943 & -1.965 & -1.843 & -1.849 & -1.833  \\
200       &  -1.987    & -1.998  & -2.003& -1.965 & -1.960 & -1.975 & -2.072 & -2.063 & -2.132 \\
\bottomrule
\end{tabular}
\caption{Square embedded in a square. Convergence order with respect to the mesh size $h=\frac1N$ for $\varepsilon=10^{-10}$.}
\label{tab:errorN}
\end{table}

\begin{table}[h]
\centering
\begin{tabular}{cccccccccc}   
\toprule
$\varepsilon$   & $L^\infty$ & $L^2$ & $H^1$ & $L^\infty$ (no & $L^2$ (no  & $H^1$ (no  & $L^\infty $ & $L^2$  & $H^1$\\
 & & & & boundary) & boundary) & boundary) & (interior) & (interior) & (interior) \\
\midrule
$10^{-1}$      & --    & -- &  --  & -- & -- & -- & -- & -- & --  \\
$10^{-2}$     &    0.862      &  0.793   & 0.915 & 0.857 & 0.789 & 0.912 & 0.749 & 0.766 & 0.825  \\
$10^{-3}$        &  0.922    & 0.900 & 0.934 & 0.921 & 0.900 & 0.927 & 0.863 & 0.905 & 0.895  \\
$10^{-4}$       &  1.071    & 1.208  & 0.869 & 1.082 & 1.216 & 0.868 & 1.325 & 1.304 & 0.938 \\
\bottomrule
\end{tabular}
\caption{Square embedded in a square. Convergence order with respect to the penalization parameter $\varepsilon$ for $N=150$.}
\label{tab:erroreta}
\end{table}

The numerical results suggest as expected the convergence of $u_\varepsilon$  towards $u$ inside the fluid domain, with an order of convergence of $1$ with respect to the penalization parameter as in the one-dimensional case~\cite{Bensiali2014}, and as obtained in the theoretical study assuming more regularity (Section~\ref{sec:convergenceBL}). We recover a second-order convergence with respect to the mesh size, if special care is taken in the choice of the values at discontinuity points as explained earlier (in addition to $N$ chosen so that the Cartesian grid contains the points on the boundary of the initial square).

\paragraph{Error inside the obstacle domain and investigation of the boundary layer}

In this section, we propose to study also the error inside the obstacle domain as well as the presence of a boundary layer, that is expected theoretically and observed numerically at the boundary $\partial\Omega$. 

Indeed, a cross-section of the numerical solution of the penalized problem for $y=0.5$ shows that $u_\varepsilon(\cdot,0.5)$ converges to $u_{\lim}(\cdot,0.5)$ far from the boundary, with a boundary layer near the boundary as the solution $u_\varepsilon$ has to satisfy homogeneous Dirichlet boundary conditions on the boundary, see Figure~\ref{fig:boundarylayer}.

Theoretically (Theorem~\ref{th:convergence}), the expression of $u_{\lim}$ is the following in $[0,1]\times]0.2,0.8[$
\begin{equation}
    u_{\lim}(x,y)=
    \begin{dcases}
        V^0(x,y)=u(x,y) & \text{if $x\in [0.2,0.8]$}\\
        \overline{W}^0(x,y) &\text{if $x\in [0,0.2]\cup [0.8,1]$},
    \end{dcases}
\end{equation}
and we obtain here explicitly (by solving~\eqref{eq:solWbar0Nd}), again for $y\in{]}0.2,0.8[$,
\begin{equation}
    \overline{W}^0(x,y)=
    \begin{dcases}
        \Bigl(u(0.2,y)-\frac{\tilde{g}(0.2,y)}\alpha\Bigr) e^{-\alpha(0.2-x)}+\frac{\tilde{g}(0.2,y)}\alpha & \text{if $x\in [0,0.2]$}\\
        \Bigl(u(0.8,y)-\frac{\tilde{g}(0.8,y)}\alpha\Bigr) e^{-\alpha(x-0.8)}+\frac{\tilde{g}(0.8,y)}\alpha &\text{if $x\in [0.8,1]$}.
    \end{dcases}
\end{equation}

\begin{figure}[!h]
	\centering
	\begin{subfigure}[b]{0.49\textwidth}
		\centering
		\includegraphics[width=\textwidth]{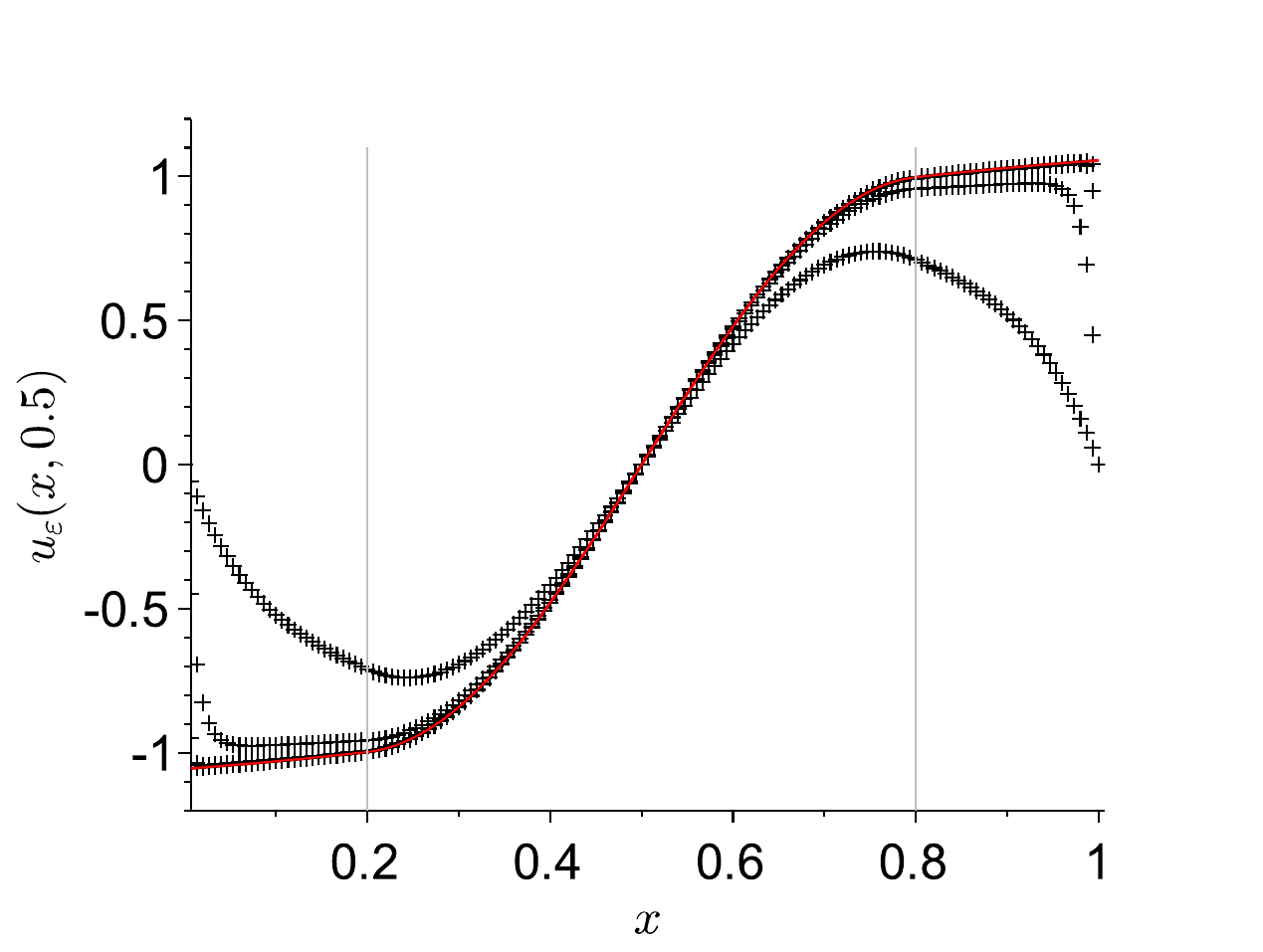}
		\caption{Plots of $u_\varepsilon$ and $u_{\lim}$}
		\label{fig:boundarylayer1}
	\end{subfigure}
	\begin{subfigure}[b]{0.49\textwidth}
		\centering
		\includegraphics[width=\textwidth]{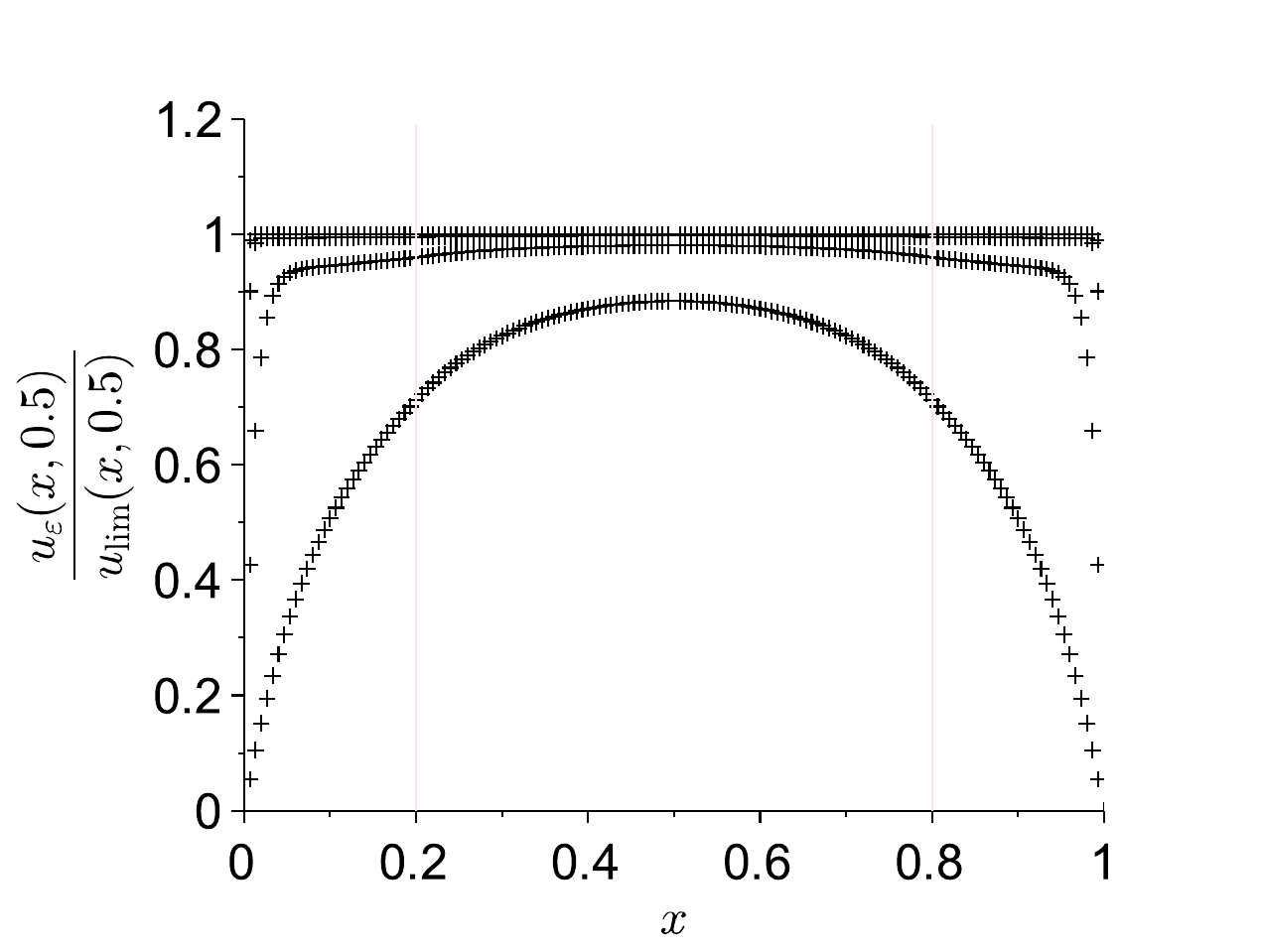}
		\caption{Plots of the ratio $\frac{u_\varepsilon}{u_{\lim}}$}
		\label{fig:boundarylayer2}
	\end{subfigure}
	\caption{Comparison between the numerical solution of the penalized problem $u_\varepsilon$ for $\varepsilon\in\{10^{-1},10^{-2},10^{-3},10^{-4}\}$ and (red curve) the limit solution $u_{\lim}$ ($V^0=u$ in the fluid domain, $\overline{W}^0$ in the obstacle domain) for $y=0.5$. Here, $N=150$.}
	\label{fig:boundarylayer}
\end{figure}

In addition, Table~\ref{tab:erroretaobstacle} shows the convergence order of the error between the numerical solution of the penalized problem and the limit solution inside the obstacle domain far from the boundary: more precisely, we consider the error inside the strip $[0.1,0.2]\times [0.5,0.5+h]$. An order of convergence of 1 is recovered in the $H^1$-norm (as expected theoretically (cf. Theorem~\ref{th:convergence})) if the mesh size is chosen small enough (we took $N=1000$, hence $h=10^{-3}$) in order to witness this order for small values of $\varepsilon$.

\begin{table}[!h]
\centering
\begin{tabular}{cccc}   
\toprule
$\varepsilon$   & $L^\infty$ & $L^2$ & $H^1$ \\
\midrule
$10^{-1}$      & --    & -- &  --    \\
$10^{-2}$     &    0.991     &  0.949   & 1.142   \\
$10^{-3}$        &  0.929    & 0.915 & 1.012 \\
$10^{-4}$       &  0.974    & 0.968  & 1.011  \\
\bottomrule
\end{tabular}
\caption{Convergence order of $\|u_\varepsilon-\overline{W}^0\|$ in the strip $[0.1,0.2]\times[0.5,0.5+h]$ with respect to the penalization parameter $\varepsilon$ for $N=1000$.} 
\label{tab:erroretaobstacle}
\end{table}

In order to check the thickness of the boundary layer, which theoretically should be of order $\varepsilon$, we recall that we expect, using~\eqref{eq:expansionNd} and~\eqref{eq:boundarylayerobstacleNd},
\begin{equation}
u(x,0.5)=\overline{W}^0(x,0.5)+\theta(x) \widetilde{W}^0\Bigl(x,0.5,\frac{\varphi(x,0.5)}\varepsilon\Bigr)+O(\varepsilon),
\end{equation}
and we obtain here explicitly (by solving~\eqref{eq:solWtilde0Nd})
\begin{equation}
    \widetilde{W}^0\Bigl(x,0.5,\frac{\varphi(x,0.5)}\varepsilon\Bigr)=
    \begin{dcases}
        -\overline{W}^0(x,0.5)e^{-\frac{x}\varepsilon} & \text{if $x\in [0,0.2]$}\\
        -\overline{W}^0(x,0.5)e^{-\frac{1-x}\varepsilon} &\text{if $x\in [0.8,1]$}.
    \end{dcases}
\end{equation}
where $\theta$ is $C^\infty$, localized near the boundary $x=0$ or $x=1$ and $\theta\equiv 1$ near this same boundary. Thus, near the boundary $x=0$ for instance, we should have, for $\varepsilon$ small enough,
\begin{equation}
    u_\varepsilon(x,0.5)\approx u_{\lim}(x,0.5)(1-e^{-\frac{x}\varepsilon}).
\end{equation}

The thickness of the boundary layer near the boundary $x=0$ can be defined (conventionally) by $E_\varepsilon$ such that
\begin{equation}
    \frac{u_\varepsilon(E_\varepsilon,0.5)}{u_{\lim}(E_\varepsilon,0.5)}=0.95 \quad \implies \quad 1-e^{-\frac{E_\varepsilon}\varepsilon}\approx 0.95 \quad \implies \quad -E_\varepsilon \approx \log(0.05)\, \varepsilon.
\end{equation}

As estimating $E_\varepsilon$ numerically using its initial definition is not easy given that the approximated values of $u_\varepsilon(x,0.5)$ are discrete, and since we are more interested in the order of convergence of the thickness of the boundary layer rather than the value of the thickness itself, we suggest to estimate the thickness $BL_\varepsilon=\varepsilon$ instead (this corresponds to the definition $E_\varepsilon$ such that $\frac{u_\varepsilon(E_\varepsilon,0.5)}{u_{\lim}(E_\varepsilon,0.5)}\approx 1-e^{-1}\approx0.63$). This can be obtained, if we denote $RU_\varepsilon(x)\approx\frac{u_\varepsilon(x,0.5)}{u_{\lim}(x,0.5)}$, by
\begin{equation}\label{eq:lengthBL1}
    BL^1_\varepsilon\approx\frac{1}{\frac{\partial RU_\varepsilon}{\partial x}(0)}\approx \frac{h}{RU_\varepsilon(h)}
\end{equation}
or
\begin{equation}\label{eq:lengthBL2}
    BL^2_\varepsilon\approx -\frac{h}{\log(1-RU_\varepsilon(h))}
\end{equation}
where $h=\frac1{N}$ is the meshsize. The results are reported on Table~\ref{tab:convlengthBL}. We recover the first-order convergence of the length of the boundary layer with respect to $\epsilon$, at least for values of $\varepsilon\ge 10^{-3}$. For smaller values of $\varepsilon$ ($10^{-4}$ for instance), we are limited by the discretization step $h=10^{-3}$ which does not allow to capture properly the length of the boundary layer, and one needs to refine the mesh further for this purpose. In addition, it seems that the estimate $BL^2_\varepsilon$ is more accurate, which makes sense since the estimate $BL_\varepsilon^1$ contains an additional error with respect to $h$ related to the approximation of the derivative.

\begin{table}[!h]
\centering
\begin{tabular}{ccccc}   
\toprule
$\varepsilon$   & Value of $BL^1_\varepsilon$ & Order of convergence & Value of $BL^2_\varepsilon$ & Order of convergence \\
\midrule
$10^{-1}$      & $0.116$    & -- &  $0.116$  & --  \\
$10^{-2}$     &    $1.10\times 10^{-2}$   &  1.024   & $1.05\times 10^{-2}$   &  1.043\\
$10^{-3}$     &  $1.79\times 10^{-3}$     & 0.789 & $1.22\times 10^{-3}$ & 0.933\\
$10^{-4}$     &  $1.06\times 10^{-3}$    & 0.224  & $3.64\times 10^{-4}$ & 0.525  \\
\bottomrule
\end{tabular}
\caption{Estimate and convergence order of the thickness of the boundary layer of $u_\varepsilon(x,0.5)$ near the boundary $x=0$ with respect to the penalization parameter $\varepsilon$ for $N=1000$.} 
\label{tab:convlengthBL}
\end{table}

%% file: sections/generalisation.tex
Using the same approach as above, the main result of this article stating the convergence of the penalization method holds in higher dimensions $d\ge 2$, assuming enough regularity on the domains and the data. More precisely, Theorems~\ref{th:convergence} and~\ref{th:convergenceBL} hold assuming the following regularity hypotheses:

\clearpage
\vskip.5cm

{\bf Regularity hypotheses in dimension $d$:}

\begin{itemize}
    \item $\Omega$ and $\mathcal U$ are supposed to be  of class $C^{k+3\lfloor\frac{d}2\rfloor+5,1}$ in dimension $d$ for some $k\ge 2$
    \item $f\in H^{k+3\lfloor\frac{d}2\rfloor+4}(\mathcal{U})$
    \item $\tilde{g}\in H^{k+3\lfloor\frac{d}2\rfloor+4+\frac12}(\partial\mathcal U)$
    \item $n=\nabla\psi$ with
    \begin{itemize}
        \item $\psi\in C^{k+2\lfloor \frac{d}2\rfloor+5}(\overline{\omega})$
        \item $\psi=0$ on $\partial\mathcal{U}$
        \item $\psi>0$ in $\omega$
        \item $|\nabla\psi|\ge \lambda>0$ in $\omega$
        \item $n \cdot \nu_\Omega\ge \lambda_0>0$ on $\partial\Omega$, where $\nu_\Omega$ is the outward normal vector to $\Omega$
    \end{itemize}
\end{itemize}

%% file: sections/conclu.tex
In this paper, we have investigated the extension to higher dimensions of a previously suggested penalization method for Neumann or Robin boundary conditions~\cite{Bensiali2014}. Our results include the existence and uniqueness of the solution of the penalized problem and the convergence with respect to the penalization parameter towards the solution. The existence and uniqueness have been established using the method of Droniou for non-coercive linear elliptic problems~\cite{dronioupota2002}. To prove the convergence, we developed a boundary layer approach, assuming enough regularity, for Neumann or Robin boundary conditions, adapted from the Dirichlet approach~\cite{BookBoyerFabrie}.  The primary challenge in applying the boundary layer approach to Neumann or Robin boundary conditions is that, in contrast to Dirichlet boundary conditions, the penalized problem is not coercive, making it more challenging to derive essential estimates on the remainders. To this aim, in order to derive such estimates, a thorough examination of the dual problem turned out to be pertinent: the convergence result was reduced to finding suitable supersolutions to the dual problem. While the construction of these supersolutions might be explicit in special cases: one-dimensional case, spherical symmetric case, it is more tricky in general settings. 
We proposed to conduct an additional formal boundary layer approach on the dual problem in order to construct theses supersolutions as approximated solutions of the dual problem. Finally, we have derived numerical experiments in two dimensions, using adequate finite difference schemes suitable for advection-dominated problems.  Even if finite differences are not the best choice of discretization given the discontinuity of the coefficients introduced by the penalization, these numerical simulations allowed to illustrate the theoretical results and also suggest that the convergence result may be valid in less regular settings. It would thus be interesting to weaken the regularity assumptions (ours are not necessarily sharp) in the context of the boundary layer approach or using weak convergence approaches.

Perspectives of this work include different directions:
\begin{itemize}
    \item extension of the proposed method to other model problems of interest in fluid mechanics: Stokes, Navier-Stokes equations, \ldots as it has been performed for penalized Dirichlet boundary conditions~\cite{BookBoyerFabrie},
    \item practical extension of the various data on the boundary to the complementary domain,
    \item improvement of numerical approximations using stabilized finite elements suitable for singular perturbation problems~\cite{gerdes2001hp,melenk2002hp}, more suitable for our penalized problem,
    \item application to problems raised on moving domains such as those associated with the simulation of population dynamics under climate change~\cite{berestycki2009can,roques2008population}.
\end{itemize}

\section*{Acknowledgments}

B. Bensiali thanks Centrale Marseille/I2M for the hospitality and assistance during her research stays. She also thanks Prof. Stefan Sauter for useful discussions on this work at the 11th Zurich Summer School 2021 and for the invitation to the ``Institut für Mathematik'' at the University of Zurich, which indirectly allowed her to treat the spherical symmetric case using Bessel functions (Appendix~\ref{sec:annexspecialcasessph}).

\medskip

This research was supported through computational resources of HPC-MARWAN (\href{hpc.marwan.ma}{hpc.marwan.ma}) provided by the National Center for Scientific and Technical Research (CNRST), Rabat, Morocco.

%% file: sections/appendix-Charac.tex
\section{Regularity for advection-reaction equation using the method of characteristics}\label{sec-append-Charac}

We give here the proof of Theorem~\ref{th:hyperbolicregularitybis}.

\begin{proof}

We use the method of characteristics to solve~\eqref{eq:advectionreaction}, starting by considering the constant case $\beta(x)=\alpha$, and follow the following steps.

\paragraph*{Step 1. Method of characteristics}

For simplicity, we assume $\psi$ defined in all $\RR^2\setminus \mathcal U$ and embed the domain $\omega$ inside the domain delimited by $\psi=0$ ($\partial\mathcal U$) and $\psi=c_2>\max_{\overline{\omega}} \psi$. 
The previous set is not empty (see below). In addition, the assumption $\psi\in C^{k+1}(\RR^2\setminus \mathcal U)$ and $|\nabla\psi|\ge \lambda>0$ imply $\frac{\partial\psi}{\partial x}(x,y)\ne 0$ or $\frac{\partial\psi}{\partial y}(x,y)\ne 0$, thus using the theorem of implicit functions, we deduce that $\psi=c$ is locally a curve of class $C^{k+1}$. For simplisity, we will also assume that the curves $\psi=c$ are closed.. For $(x,y)\in\omega$, we introduce
\begin{equation}\label{eq:edocharacteristics}
\begin{dcases}
\dot{X}(s)=n(X(s))=\nabla\psi(X(s))\\
X(t)=(x,y).
\end{dcases}
\end{equation}
The conditions on $\nabla\psi$ (continuous and globally Lipschitzian) enable to use the global Cauchy-Lipschitz theorem~\cite{rouviere2009petit} that ensures the existence and uniqueness in $C^1$ of $s\mapsto X(s)$ for all $s$, then denoting $\Psi(s)=\psi(X(s))$, one has
\begin{align*}
\frac{d\Psi}{ds}(s)&=\frac{\partial\psi}{\partial x}(X(s))\,\dot{X_1}(s)+\frac{\partial\psi}{\partial y}(X(s))\,\dot{X_2}(s)\\
&=|\nabla\psi(X(s))|^2.
\end{align*}
Using  the hypothesis $|\nabla\psi|\ge \lambda>0$ and integrating we get for $t\le T$, $\Psi(T)\ge \Psi(t)+(T-t) \lambda^2$  and, for $t>T$, $\Psi(T)\le \Psi(t)-(t-T) \lambda^2.$

Thus for all $(x,y)\in \omega$ (embeded between $\psi=0$ and $\psi=c_2$), there exists a unique $C^1$ curve $s\mapsto X(s)=(X_1(s),X_2(s))$ solution to~\eqref{eq:edocharacteristics} that passes through $(x,y)$ and cuts the curve $\psi=0$ ($\partial\mathcal U$) and the curve $\psi=c_2$. We deduce that for all $(x,y)\in\omega$ there exists a unique $C^1$ curve $s\mapsto X(s)=(X_1(s),X_2(s))$ solution to~\eqref{eq:edocharacteristics}, which cuts the curve $\psi=0$ ($\partial\mathcal U$) and $\partial\Omega$.

We now introduce 
\begin{equation}\label{eq:edoxit}
\begin{dcases}
\frac{\partial X}{\partial t}(\xi,t)=n(X(\xi,t))=\nabla\psi(X(\xi,t)) \\
X(\xi,0)=\gamma(\xi)
\in\partial\mathcal{U}.
\end{dcases}
\end{equation}
where $t\ge 0$ and $\xi\in[0,\Xi]$ is a parameterization of $\partial\mathcal{U}$ ($\psi=0$), i.e. $\partial\mathcal U=\{\gamma(\xi),\ \xi\in[0,\Xi]\}$. Given the assumption on $\partial\mathcal U$, $\gamma\in C^k([0,\Xi])$.
We know, from the previous reasoning, that for all $\xi\in[0,\Xi]$ the curve $t\mapsto X(\xi,t)$ will cut the curve $\psi=c_2>0$ at a time $T_{2,\xi}$ such that
\begin{align*}
c_2=\Psi(T_{2,\xi})&\ge \Psi(0)+ T_{2,\xi}\, \lambda^2\\
T_{2,\xi}&\le \frac{c_2}{\lambda^2}.
\end{align*}
Thus there exists a time $T_f<\infty$ (since $T_f\le\frac{c_2}{\lambda^2}$) such that all the characteristics starting from $\partial\mathcal{U}$ (at $t=0$) arrive at the curve $(\psi=c_2)$. We deduce that for all $\xi\in [0,\Xi]$, there exists a time $T_\xi<\infty$ such that the curve $t\mapsto X(\xi,t)$ arrives at $\partial\Omega$. Since we assume $n\cdot \nu_\Omega>0$, the curve $t\mapsto X(\xi,t)$ cannot come back inside of $\omega$ after reaching it for the first time. We thus know that by solving~\eqref{eq:edoxit} for $\xi\in[0,\Xi]$ and $t\in[0,T_\xi]$, the union of all the solutions will fill exactly $\overline{\omega}$.

Then, by denoting
\begin{align*}
\mathcal{G}(\xi,t)&=g(X(\xi,t))\\
\mathcal{W}(\xi,t)&=W(X(\xi,t))\\
\mathcal{V}(\xi)&=V(X(\xi,0))=V(\gamma(\xi)),
\end{align*}
we obtain from the equation on $W$~\eqref{eq:advectionreaction} that $\mathcal{W}(\xi,t)$ satisfies
\begin{align}\label{eq:equationWt}
\dot{\mathcal{W}}(\xi,t)&=\frac{\partial W}{\partial x}(X(\xi,t))\,\dot{X_1}(\xi,t)+\frac{\partial W}{\partial y}(X(\xi,t))\,\dot{X_2}(\xi,t)\notag\\
&=\frac{\partial W}{\partial x}(X(\xi,t))\,n_1(X(\xi,t))+\frac{\partial W}{\partial y}(X(\xi,t))\,n_2(X(\xi,t))\notag\\
&=\nabla W(X(\xi,t))\cdot n(X(\xi,t))\notag\\
&=g(X(\xi,t))-\alpha W(X(\xi,t))\notag\\
&=\mathcal{G}(\xi,t)-\alpha \mathcal{W}(\xi,t)\notag\\
\dot{\mathcal{W}}(\xi,t)+\alpha \mathcal{W}(\xi,t)&=\mathcal{G}(\xi,t),
\end{align}
whose solution is given by
\begin{equation}
\mathcal{W}(\xi,t)=k(\xi) e^{-\alpha t}+\mathcal{W}_\mathrm{part}(\xi,t),
\end{equation}
where $\mathcal{W}_\mathrm{part}(\xi,t)$ is a particular solution that we can obtain using the method of variation of constants: we look for a particular solution of the form
\begin{equation}
\mathcal{W}_\mathrm{part}(\xi,t)=k(\xi,t) e^{-\alpha t}
\end{equation}
then
\begin{align}
\dot{\mathcal{W}}_\mathrm{part}(\xi,t)&=\dot{k}(\xi,t)e^{-\alpha t}-k(\xi,t)\alpha e^{-\alpha t}\\
&=\dot{k}(\xi,t)e^{-\alpha t}-\alpha \mathcal{W}_\mathrm{part}(\xi,t).
\end{align}
We want $\mathcal{W}_\mathrm{part}(\xi,t)$ to be a solution to~\eqref{eq:equationWt}, i.e.
\begin{align}
\dot{k}(\xi,t)=\mathcal{G}(\xi,t)e^{\alpha t}.
\end{align}
We take
\begin{equation}
k(\xi,t)=\int_0^t \mathcal{G}(\xi,s)e^{\alpha s}\, ds.
\end{equation}
Finally,
\begin{equation}
\mathcal{W}(\xi,t)=(k(\xi)+k(\xi,t))e^{-\alpha t},
\end{equation}
and using the boundary condition in~\eqref{eq:advectionreaction}, $\mathcal{W}(\xi,0)=W(X(\xi,0))=W(\gamma(\xi))=V(\gamma(\xi))=\mathcal{V}(\xi)$ and we obtain
\begin{equation}
\mathcal{W}(\xi,t)=\Bigl(\mathcal{V}(\xi)+\int_0^t \mathcal{G}(\xi,s)e^{\alpha s}\, ds\Bigr)e^{-\alpha t},
\end{equation}
thus
\begin{equation}\label{eq:solWxit}
W(X(\xi,t))=\Bigl(V(\gamma(\xi))+\int_0^t g(X(\xi,s))e^{\alpha s}\, ds\Bigr)e^{-\alpha t}.
\end{equation}

Because the characteristics fill all $\overline{\omega}$, the previous formula allows formally to define $W(x,y)$ at each point $(x,y)\in\overline{\omega}$. Since we are interested in the regularity of the solution $W$ as a function of the space variables, we need to express $W(x,y)$. 

\paragraph*{Step 2. Regularity of $X(\xi,t)$ with respect to $\xi$ and $t$}

We denote by $G$ the function
\begin{equation}\label{eq:functionG}
    G\colon (\xi,t)\mapsto X(\xi,t)
\end{equation}
where for every $\xi\in[0,\Xi]$, $t\mapsto X(\xi,t)$ is the unique solution to the Cauchy problem~\eqref{eq:edoxit}.

The regularity of $G$ with respect to $\xi$ is related to the differentiable dependence on initial conditions for a Cauchy problem. Indeed, let's introduce
\begin{equation}\label{eq:edoY}
\begin{dcases}
\frac{\partial Z}{\partial t}(z_0,t)=n(Z(z_0,t))=\nabla\psi(Z(z_0,t)) \\
Z(z_0,0)=z_0
\in \RR^2\setminus\mathcal U,
\end{dcases}
\end{equation}
then since $\psi\in C^{k+1}(\overline{\omega})$, hence $n \in C^{k}(\overline{\omega})$, we deduce that the solution $(z_0,t)\mapsto Z(z_0,t)$ is $C^k$ on its domain of existence (\cite{hartman2002ordinary}, Corollary 4.1).

Using the fact that $G(\xi,t)=X(\xi,t)=Z(\gamma(\xi),t)$, we deduce, since $\gamma\in C^k([0,\Xi])$, that $G\in C^k([0,\Xi]\times\RR^+)$.

\paragraph*{Step 3. From local to global inverse function theorem}

Here we show that $G$ is invertible, and that $H=G^{-1}$ is $C^k (\overline{\omega})$.
\begin{equation}\label{eq:functionH}
    H\colon (x,y) \mapsto (\xi,t) \text{ such that } G(\xi,t)=(x,y).
\end{equation}

We make use of the inverse function theorem.
\begin{theorem}[Inverse function theorem]
Let $f\colon U\to V$ be a map between open subsets of $\RR^n$. Assume $f$ is $C^k$. If $f$ is injective on a closed subset $A\subset U$ and if the jacobian matrix of $f$ is invertible at each point of $A$, then $f$ is injective in a neighborhood $A'$ of $A$ and $f^{-1}\colon f(A')\to A'$ is $C^k$.
\end{theorem}

We fix $\xi_0\in[0,\Xi]$. By definition $G(\xi,t)=X(\xi,t)$. Therefore, the Jacobian of $G$ is given by
\begin{align*}
    JG(\xi,t)&=
    \begin{vmatrix}
    \frac{\partial {X_1}}{\partial\xi}(\xi,t) & \frac{\partial {X_1}}{\partial t}(\xi,t)\\
    \frac{\partial {X_2}}{\partial\xi}(\xi,t) & \frac{\partial {X_2}}{\partial t}(\xi,t)
    \end{vmatrix}\\
    &=\frac{\partial {X_1}}{\partial\xi}  \frac{\partial {X_2}}{\partial t} - \frac{\partial {X_1}}{\partial t}  \frac{\partial {X_2}}{\partial \xi}.
\end{align*}
At the point $(\xi_0,0)$,
\begin{align*}
JG(\xi_0,0)&=\frac{\partial {X_1}}{\partial\xi} (\xi_0,0) \frac{\partial {X_2}}{\partial t} (\xi_0,0)- \frac{\partial {X_1}}{\partial t}  (\xi_0,0)\frac{\partial {X_2}}{\partial \xi} (\xi_0,0)\\
&=\gamma_1'(\xi_0) n_2(\gamma(\xi_0))-n_1(\gamma(\xi_0)) \gamma_2'(\xi_0)\\
&=n(\gamma(\xi_0)) \cdot (\gamma_1'(\xi_0),-\gamma_2'(\xi_0)),
\end{align*}
since, by definition, $X(\xi,t)$ satisfies the characteristic ODE~\eqref{eq:edoxit}. Forall $\xi_0\in [0,\Xi]$, $JG(\xi_0,0)\ne0$ since $n$ is not tangent to $\partial\mathcal U$ (we say that $\partial\mathcal U$ is noncharacteristic). We already know that $G$ is injective on $[0,\Xi]\times\RR^+$ (by the global Cauchy-Lipschitz theorem). By the inverse function theorem, $G^{-1}\colon G(A')\to A'$ is $C^k$ where $A'$ is a neighborhood of $[0,\Xi]\times \{0\}$. We thus obtain that $G^{-1}$ is $C^k$ on a neighborhood $B$ of $\partial\mathcal U$.

Our goal is to establish a global differentiability, that is $G^{-1}\in C^k(\overline{\omega})$. To this end, we will use a step by sted advancing procedure. Let us explain the first step. We denote by $G_I^{-1}\colon B_I\mapsto [0,\Xi]\times\RR^+$ the restriction of $G^{-1}$ to $B_I$ a neighborhood of $\partial\mathcal U$. We denote by $c=\min_{\partial B_I\setminus \partial\mathcal U} \psi$, then we take $\Gamma_{II}$ to be defined by $\psi=c_{II}$ where $0>c_{II}<c$. We recall that we assumed that the curves $\psi=c$ are closed. $\Gamma_{II}\subset B_I$ and we consider $\Gamma_{II}=\{\gamma_{II}(\xi_{II}),\ \xi_{II}\in[0, \Xi_{II}]\}$ with $\gamma_{II}\in C^k$ (this follows from the regularity of $\psi$, the assumption on $\nabla\psi$ and using the theorem of implicit functions, as stated previously). We introduce $G_{II}(\xi_{II},t)=Y(\xi_{II},t)$, where $Y(\xi_{II},t)$ is the unique solution to
\begin{equation}\label{eq:edoxitII}
\begin{dcases}
\frac{\partial Y}{\partial t}(\xi_{II},t)=n(Y(\xi_{II},t))=\nabla\psi(Y(\xi_{II},t)) \\
Y(\xi_{II},0)=\gamma_{II}(\xi_{II})
\in \Gamma_{II}.
\end{dcases}
\end{equation}
Since $\Gamma_{II}$ is noncharacteristic ($n=\nabla\psi$ is normal to $\Gamma_{II}$), we obtain as previously using the inverse function theorem that $G_{II}^{-1}$ is $C^k$ on a neighborhood $B_{II}$ of $\Gamma_{II}$. For $(x,y)\in B_{II}\setminus B_I$, since
\begin{align*}
    G_{II}^{-1}(x,y)=(\xi_{II},t_{II}) \text{ such that $G_{II}(\xi_{II},t_{II})=(x,y)$}
\end{align*}
and
\begin{align*}
    G_{I}^{-1}(\gamma_{II}(\xi_{II}))=(\xi,t) \text{ such that $G_{I}(\xi,t)=\gamma_{II}(\xi_{II})$},
\end{align*}
we obtain
\begin{align}
    G^{-1}(x,y)=\Bigl([G^{-1}_{I}\bigl(\gamma_{II}([G_{II}^{-1}(x,y)]_1)\bigr)]_1,[G^{-1}_{II}(x,y)]_2+[G^{-1}_{I}\bigl(\gamma_{II}([G_{II}^{-1}(x,y)]_1)\bigr)]_2\Bigr)
\end{align}
which allows to show by composition, that $G^{-1}$ is $C^k$ on $B_I\cup B_{II}$. 

It remains to show that this process allows to visit all $\overline{\omega}$. Assume, by contradiction, that there exist unreachable curves $\Gamma_{\hat{\beta}}=\{x,\ \psi(x)=\hat{\beta}\}$ and let $\beta_*$ be the infimum of the $\hat{\beta}$ and note $\gamma_*$ the parameterization of $\Gamma_{\beta_*}$. If $\Gamma_*$ can be reached, the procedure above provides and extension and then a contradiction. If $\Gamma_*$ can not be reached, we introduce $\tilde{G}_{*}(\xi_{*},t)=Z(\xi_{II},t)$, where $Z(\xi_{II},t)$ is the unique solution to the backward Cauchy problem
\begin{equation}\label{eq:edoxit*}
\begin{dcases}
\frac{\partial Z}{\partial t}(\xi_{*},t)=-n(Z(\xi_{*},t))=-\nabla\psi(Z(\xi_{*},t)) \\
Z(\xi_{*},0)=\gamma_{*}(\xi_{*})
\in \Gamma_{*},
\end{dcases}
\end{equation}
then $\tilde{G}_{*}$ is $C^k$ on its domain of existence.

Let us fix $t_2>0$ small enough such that, for all $\xi_*\in[0,\Xi_*]$, $\tilde{G}_*(\xi_*,t_2)\in B^*$, where $B_*$ is a neighborhood of $\Gamma_*$, then $\tilde{G}_*(\xi_*,t_2)=\gamma_{\beta'}(\xi_{\beta'})$ ($\tilde{G}_*(\xi_*,t_*)$ belongs to a curve $(\psi=\beta')$ included in $B_*$ with $\beta'<\beta_*$). Since we know that $G^{-1}$ is $C^k$ on $\{x,\ \psi(x)=\beta'\}$, and
\begin{align*}
G^{-1}(\gamma_{*}(\xi_{*}))&=(\xi,t) \text{ such that $G(\xi,t)=\gamma_{*}(\xi_{*})$}\\
\tilde{G}_*(\xi_*,t_2)&=\gamma_{\beta'}(\xi_{\beta'})\\
G^{-1}(\gamma_{\beta'}(\xi_{\beta'}))&=(\xi,t_1) \text{ such that $G(\xi,t_1)=\gamma_{\beta'}(\xi_{\beta'})$},
\end{align*}
we can extend $G^{-1}$ up to $\Gamma_*$ forming 
\begin{align}
    G^{-1}(\gamma_*(\xi_*))=\Bigl([G^{-1}\bigl(\tilde{G}_*(\xi_*,t_2)\bigr)]_1,[G^{-1}\bigl(\tilde{G}_*(\xi_*,t_2)\bigr)]_2+t_2\Bigr)
\end{align}
which is $C^k$ ($\gamma_*^{-1}$ is $C^k$) and provides the contradiction.

\paragraph*{Step 4. Regularity of $W$}

We are  now able to discuss the regularity of $W$. By inversion, \eqref{eq:solWxit} writes
\begin{equation}\label{eq:solWxy}
    W(x,y)=\Bigl(V(\gamma(H_1(x,y)))+\int_0^{H_2(x,y)} g(G(H_1(x,y),s)) e^{\alpha s}\, ds \Bigr) e^{-\alpha H_2(x,y)},
\end{equation}
where $G$ and $H$ are given by~\eqref{eq:functionG} and~\eqref{eq:functionH} respectively.

Since $V$ is $C^k(\partial\mathcal U)$, $\gamma\in C^k([0,\Xi])$ and $H_1\in C^k(\overline{\omega})$, we deduce $(x,y)\mapsto V(\gamma(H_1(x,y)))$ is $C^k(\overline{\omega})$. Similarly $(x,y)\mapsto e^{-\alpha H_2(x,y)}$ is $C^k(\overline{\omega})$ by composition.

It then remains to study the regularity of the function
\begin{equation}\label{functionL}
(x,y)\mapsto L(x,y) =\int_0^{H_2(x,y)} g(G(H_1(x,y),s)) e^{\alpha s}\, ds.
\end{equation}
Let us introduce
\begin{equation}\label{eq:functionK}
(x,y,t)\mapsto K(x,y,t) =\int_0^{t} g(G(H_1(x,y),s)) e^{\alpha s}\, ds=\int_0^t h(x,y,s)\, ds.
\end{equation}
then
\begin{equation}\label{eq:functionL}
    L(x,y)=K(x,y,H_2(x,y)).
\end{equation}
Since $g\in C^k(\overline{\omega})$, $G\in C^k([0,\Xi]\times\RR^+)$ and $H\in C^k(\overline{\omega})$ then $(x,y,s)\mapsto h(x,y,s)$ is $C^k(\overline{\omega}\times\RR^+)$.

We prove by induction on $k\ge 1$ that if $h$ is $C^k(\overline{\omega}\times\RR^+)$ then $K$ is $C^k(\overline{\omega}\times\RR^+)$.

\begin{lemma}
Let $\omega$ be a bounded domain of $\RR^2$. For all $k\ge 1$, if $h\in C^k(\overline{\omega}\times\RR^+)$ then $(x,y,t)\mapsto K(x,y,t)=\int_0^t h(x,y,s)\,ds$ is also in $C^k(\overline{\omega}\times\RR^+)$ and
    \begin{align*}
        \frac{\partial^k K}{\partial x^k}(x,y,t)&=\int_0^t \frac{\partial^k h}{\partial x^k}(x,y,s)\,ds,\\
        \frac{\partial^k K}{\partial y^k}(x,y,t)&=\int_0^t \frac{\partial^k h}{\partial y^k}(x,y,s)\,ds,\\
        \frac{\partial^k K}{\partial t^k}(x,y,t)&=\frac{\partial ^{k-1}h}{\partial t^{k-1}}(x,y,t).
    \end{align*}
\end{lemma}
\begin{proof}
We use a proof by induction on $k\ge1$.

\begin{itemize}
    \item We first prove the result for $k=1$:
    Since $h\in C^1(\overline{\omega}\times\RR^+)\subset C^0(\overline{\omega}\times\RR^+)$, we obtain
    \begin{equation*}
        \frac{\partial K}{\partial t}(x,y,t)=h(x,y,t),
    \end{equation*}
    and $\frac{\partial K}{\partial t}\in C(\overline\omega\times\RR^+)$.
    Now we fix $t\in\RR^+$ and consider the function $\widetilde{K}_t$ defined on $\overline{\omega}$ by
    \begin{equation*}
        \widetilde{K}_t(x,y)=K(x,y,t).
    \end{equation*}
    We use the theorem of differentiability of a function defined by in integral to show that $\widetilde{K}_t\in C^1(\overline{\omega})$. Indeed,
    \begin{itemize}
        \item For all $s\in\RR^+$, $h$ is of class $C^1$ in the variables $(x,y)\in\overline{\omega}$.
        \item Since $\frac{\partial h}{\partial x}$ and $\frac{\partial h}{\partial y}$ are continuous on $\overline{\omega}\times\RR^+$, there exists a constant $M$ such that
        \begin{align*}
            \forall (x,y)\in\overline{\omega},\ \forall s\in [0,t], \quad &\biggl|\frac{\partial h}{\partial x}(x,y,s)\biggr|\le M,\\
            &\biggl|\frac{\partial h}{\partial y}(x,y,s)\biggr|\le M,
        \end{align*}
        and the function $s\mapsto M$ is integrable on $[0,t]$.
    \end{itemize}
    In conclusion, for all $t\in\RR^+$, $\widetilde{K}_t\in C^1(\overline{\omega})$ and
    \begin{align*}
        \frac{\partial K}{\partial x}(x,y,t)=\frac{\partial \widetilde{K}_t}{\partial x}(x,y)=\int_0^t \frac{\partial h}{\partial x}(x,y,s)\,ds,\\
        \frac{\partial K}{\partial y}(x,y,t)=\frac{\partial \widetilde{K}_t}{\partial y}(x,y)=\int_0^t \frac{\partial h}{\partial y}(x,y,s)\,ds.
    \end{align*}
    We now use the theorem of continuity of a function defined by an integral to show that $\frac{\partial K}{\partial x}$ and  $\frac{\partial K}{\partial y}$ are in $C(\overline{\omega}\times\RR^+)$. Indeed, let's detail the proof for $\frac{\partial K}{\partial x}$. We first write
    \begin{equation}
        \frac{\partial K}{\partial x}(x,y,t)=\int_0^{T_f} \frac{\partial h}{\partial x}(x,y,s) \mathds{1}_{[0,t]}(s)\,ds=\int_0^{T_f} i(x,y,t,s)\,ds.
    \end{equation}
    Since $\frac{\partial h}{\partial x}\in C(\overline{\omega}\times\RR^+)$, we obtain
    \begin{itemize}
        \item $i$ is continuous in  the variables $(x,y,t)$ for almost every $s\in [0,T_f]$, i.e.
        \begin{equation*}
            i(x,y,t,s) \underset{(x,y,t)\to(x_0,y_0,t_0)}{\longrightarrow} i(x_0,y_0,t_0,s) \quad \text{for almost every $s\in [0,T_f]$}.
        \end{equation*}
        \item There exists a constant $M$ such that
        \begin{align*}
            \forall (x,y,t)\in\overline{\omega}, \quad &|i(x,y,t,s)|\le M, \quad \text{for almost every $s\in[0,T_f]$,}
        \end{align*}
        and the function $s\mapsto M$ is integrable on $[0,T_f]$.
    \end{itemize}
    In conclusion, $K\in C^1(\overline{\omega}\times\RR^+)$ since the partial derivatives exist and are continuous on $\overline{\omega}\times\RR^+$.

    \item We assume the proposition true for a given $k\ge 1$ and prove it holds for $k+1$: If $h\in C^{k+1}(\overline\omega\times\RR^+)$ then $h\in C^k(\overline{\omega}\times\RR^+)$ and using the induction hypothesis, it follows $K\in C^k(\overline{\omega}\times\RR^+)$ and
    \begin{align*}
        \frac{\partial K}{\partial x}(x,y,t)&=\int_0^t \frac{\partial h}{\partial x}(x,y,s)\,ds,\\
        \frac{\partial K}{\partial y}(x,y,t)&=\int_0^t \frac{\partial h}{\partial y}(x,y,s)\,ds,\\
        \frac{\partial K}{\partial t}(x,y,t)&=h(x,y,t).
    \end{align*}

Since $\frac{\partial h}{\partial x}\in C^k(\overline{\omega}\times\RR^+)$ and $\frac{\partial h}{\partial y}\in C^k(\overline{\omega}\times\RR^+)$, then using again the induction hypothesis, we deduce $\frac{\partial K}{\partial x}\in C^{k}(\overline{\omega}\times\RR^+)$ and $\frac{\partial K}{\partial y}\in C^{k}(\overline{\omega}\times\RR^+)$. In addition $\frac{\partial K}{\partial t}=h\in C^k(\overline{\omega}\times\RR^+)$. In conclusion, $K\in C^{k+1}(\overline{\omega}\times\RR^+)$.
\item The proposition thus holds for all $k\ge 1$, by induction.
\end{itemize}\end{proof}

By composition, we deduce from its expression~\eqref{eq:functionL} that $L\in C^k(\overline{\omega})$. In conclusion $W\in C^k(\overline{\omega})$.

\paragraph*{Step 5. Genral case ($\beta\in C^k(\overline{\omega})$)}

Using the same approach and notations as in the proof of the constant case $\beta(x)=\alpha$ (see Appendix~\ref{sec-append-Charac}), we find by the method of characteristics that the solution of~\eqref{eq:advectionreactionbis} is given along the characteristics by
\begin{equation}\label{eq:solWxitbis}
W(X(\xi,t))=\Bigl(V(\gamma(\xi))+\int_0^t g(X(\xi,s))e^{\int_0^s \beta(X(\xi,u))\, du}\, ds\Bigr)e^{-\int_0^t \beta(X(\xi,s))\, ds},
\end{equation}
and by inversion
\begin{equation}\label{eq:solWbis}
W(x,y)=\Bigl(V\bigl(\gamma(H_2(x,y))\bigr)+\int_0^{H_2(x,y)} g(G(H_1(x,y),s))e^{\int_0^{s} \beta(G(H_1(x,y),u))\, du}\, ds\Bigr)e^{-\int_0^{H_2(x,y)} \beta(G(H_1(x,y),s))\, ds}.
\end{equation}
The proof of the regularity follows the same ideas as in the previous case ($\beta(x)=\alpha$).

\end{proof}

%% file: sections/appendixDroniouChecking.tex
\section{Hypotheses checking in Droniou's article~\cite{dronioupota2002}}\label{sec:annexDroniouChecking}

\begin{proof}
We first recall the notation and hypotheses used in \cite{dronioupota2002}. Note that these notations are local to this proof.

\paragraph*{Notation and hypotheses in \cite{dronioupota2002}}\mbox{}\\

Let $\Omega$ be a bounded domain in $\mathbb{R}^{N}(N \geqslant 2)$ with a Lipschitz continuous boundary. We denote by $\mathbf{n}$ the unit normal to $\partial \Omega$ outward to $\Omega$ and by $\sigma$ the measure on $\partial \Omega$.

If $\Gamma$ is a measurable subset of $\partial \Omega, W_{\Gamma}^{1, q}(\Omega)$ is the space of all functions in $W^{1, q}(\Omega)$ (the usual Sobolev space) the trace of which is null on $\Gamma ;$ it is endowed with the same norm as $W^{1, q}(\Omega)$, that is to say $\|v\|_{W^{1, q}(\Omega)}=\|v\|_{L^{q}(\Omega)}+\||\nabla v|\|_{L^{q}(\Omega)}$. When $q=2$, we denote as usual $W^{1,2}=H^{1}$. 

We take, $N_{*}=N$ when $N \geqslant 3$, and, $N_{*} \in{]} 2, \infty[$ when $N=2$.

We make the following hypotheses on the data.
\begin{itemize}
\item $\Gamma_{d}$ and $\Gamma_{f}$ are measurable subsets of $\partial \Omega$ such that $\sigma\left(\Gamma_{d} \cap \Gamma_{f}\right)=0$ and $\partial \Omega=\Gamma_{d} \cup \Gamma_{f}$, \hfill \textbf{[Hyp (8)]}
\item $A\colon \Omega \rightarrow M_{N}(\mathbb{R})$ is a measurable matrix-valued function which satisfies:\\
$\exists \alpha_{A}>0$ s.t. $A(x) \xi \cdot \xi \geqslant \alpha_{A}|\xi|^{2} \quad$ for a.e. $x \in \Omega$, for all $\xi \in \mathbb{R}^{N}$,\\
$\exists \Lambda_{A}>0$ s.t. $\|A(x)\| \leqslant \Lambda_{A} \quad$ for a.e. $x \in \Omega$\\
(where, for $\left.M \in M_{N}(\mathbb{R}),\|M\|:=\sup \left\{|M \xi|, \xi \in \mathbb{R}^{N},|\xi|=1\right\}\right)$,
 \hfill \textbf{[Hyp (9)]}
\item $b \in L^{N_{*} / 2}(\Omega), \quad b \geqslant 0 \text { a.e. on } \Omega$, \hfill \textbf{[Hyp (10)]} 
\item $\lambda \in L^{N_{*}-1}(\partial \Omega), \quad \lambda \geqslant 0 \ \sigma \text {-a.e. on } \partial \Omega$, \hfill \textbf{[Hyp (11)]} 
\item $\mathbf{v} \in\left(L^{N_{*}}(\Omega)\right)^{N}$, \hfill \textbf{[Hyp (12)]}
\item $L \in\left(H_{\Gamma_{d}}^{1}(\Omega)\right)^{\prime}$ \hfill \textbf{[Hyp (13)]}\\
(recall that $N_{*}=N$ when $N \geqslant 3$ and that $N_{*} \in{]} 2, \infty[$ when $N=2$ ). 
\item The non-convection parts of equation~\eqref{eq:droniouproblemweak} (Eq. (4) in Droniou's article) are supposed to be coercive, that is to say:\\
$\exists b_{0}>0, \exists E \subset \Omega$ such that $b \geqslant b_{0}$ on $E$,\\
$\exists \lambda_{0}>0, \exists S \subset \Gamma_{f}$ such that $\lambda \geqslant \lambda_{0}$ on $S$ and either\\
$\sigma\left(\Gamma_{d}\right)>0$ or $|E|>0$ or $\sigma(S)>0$. \hfill \textbf{[Hyp (14)]}
\end{itemize}

\paragraph*{Hypotheses checking}\mbox{}\\

The penalized problem~\eqref{eq:penalizedproblemweak} is of the form~\eqref{eq:droniouproblemweak} and the dual problem~\eqref{eq:dualpenalizedproblemweak} of the form~\eqref{eq:droniouproblemweakdual} by taking $N=d$ and
\begin{itemize}
    \item $\Gamma_d=\partial\Omega$, $\Gamma_f=\emptyset$. Thus $\Gamma_d\cap\Gamma_f=\emptyset$ and $\Gamma_d\cup\Gamma_f=\partial\Omega$. \hfill \textbf{[Hyp (8) OK]}
    
    \item $A=Id$. \hfill \textbf{[Hyp (9) OK]}
    
    \item $b=1+\frac{\chi}\varepsilon\alpha$. Thus $b\in L^\infty(\Omega) \subset L^{N_*/2}(\Omega)$, and $b\ge0$. \hfill \textbf{[Hyp (10) OK]}
    
    \item $\lambda=0$. Thus $\lambda\in L^{N_*-1}(\partial\Omega)$, and $\lambda\ge0$ on $\partial\Omega$.  \hfill \textbf{[Hyp (11) OK]}
    
    \item For the penalized problem, $L=(1-\chi)f+\frac\chi\varepsilon g$. Thus $L\in L^2(\Omega)\hookrightarrow H^{-1}(\Omega)=(H^1_0(\Omega))'$.\\
    For the dual problem, $L$ is assumed in $H^{-1}(\Omega)$. \hfill \textbf{[Hyp (13) OK]}
    
    \item ${\bf v}=\frac\chi\varepsilon n$
\end{itemize}
where $N_*=N$ (dimension) when $N\ge 3$ and $N_*\in{]}2,\infty[$ when $N=2$.

Also, the non-convection parts of the problem are coercive ($E=\Omega$, $b_0=1$ and $S=\emptyset$). \hfill \textbf{[Hyp (14) OK]}

Finally, the remaining hypothesis to be checked is ${\bf v}\in (L^{N_*}(\Omega))^N$. We have ${\bf v}=0$ in $\mathcal U$ and ${\bf v}=\frac{n}\varepsilon$ in $\omega$, where $n$ is an extension of the normal unit vector $\tilde{n}$ defined on $\partial\mathcal U$. $n$ is assumed to be in $H^1(\omega)$. 

In practice, we are interested in dimension $N=2$ or $3$. Because $\omega$ is a bounded Lipschiz domain, Sobolev embeddings (see Subsection~\ref{sec:sobolevembeddings}) give the following:

\begin{itemize}
    \item If $N=2$, then $n\in H^1(\omega)\hookrightarrow L^3(\omega)$ \hfill \textbf{[Hyp (12) OK]}

    \item If $N=3$, then $n\in H^1(\omega)\hookrightarrow L^{p*}(\omega)$ where $\frac1{p^*}=\frac12-\frac13=\frac16$. Thus $n\in L^6(\omega)\subset L^3(\omega)$ \hfill \textbf{[Hyp (12) OK]}
\end{itemize}

In conclusion, there exists a unique solution $u_\varepsilon$ to the penalized problem~\eqref{eq:penalizedproblemweak} and a unique solution $w_\varepsilon$ to the dual problem~\eqref{eq:dualpenalizedproblemweak} in dimension $2$ or $3$ if $n\in H^1(\omega)$. If we can choose the lifting $n\in L^d(\omega)$, then the hypothesis (12) holds also in any dimension $N=d\ge 4$.

The proof in~\cite{dronioupota2002} starts by proving an existence result for~\eqref{eq:droniouproblemweakdual}, using an approximated solution and fixed point arguments (Leray-Schauder topological degree~\cite{deimlingnonlinear,le2018fixed}). This existence result is then used to prove an a priori estimate on the solution of~\eqref{eq:droniouproblemweak} that leads to an existence result for~\eqref{eq:droniouproblemweak}. Finally, the uniqueness results for both problems are obtained using the linearity of these equations and a duality argument. A simplified proof of the existence of the solution of the dual problem when $n\in L^\infty(\omega)$ can be found in~\cite{droniou2002}.

\end{proof}

%% file: sections/appendixspecialcases.tex
\section{Explicit construction of supersolutions in special cases}\label{sec:annexspecialcases}

\subsection{One-dimensional case}

The one-dimensional case was studied in~\cite{Bensiali2014}. Here, we present an alternative proof for the convergence of the penalization method based on the boundary layer approach presented in this paper for the general case. The advantage of the boundary layer approach is that it is generalizable in higher dimension, unlike the approach used in~\cite{Bensiali2014} which was based on the explicit computation of the solution.

We thus consider the following one-dimensional problem with Neumann or Robin boundary conditions at $x=1$
\begin{equation}\label{eq:1d}
    \begin{dcases}
    -u''+u=f \quad \text{in $\mathcal{U}={]}1,2[$}\\
    -u'(1)+\alpha u(1)=g(1), \quad u(2)=0,
    \end{dcases}
\end{equation}
where $f\in L^2(]1,2[)$, $g(1)\in\RR$ and $\alpha\ge0$ are given. The corresponding penalized problem reads
\begin{equation}\label{eq:penalized1d}
    \begin{dcases}
    -u_\varepsilon''+u_\varepsilon+\frac\chi\varepsilon (-u_\varepsilon'+\alpha u_\varepsilon -g(1))=(1-\chi)f \quad \text{in $\Omega={]}0,2[$}\\
    u_\varepsilon(0)=0, \quad u_\varepsilon(2)=0,
    \end{dcases}
\end{equation}
where $\varepsilon>0$ is a small parameter (the penalization parameter), and $\chi$ is the characteristic function of the obstacle $\omega={]}0,1[$. Using the notations of this paper, here $\varphi(x)=x$ and $n(x)=-1=\nabla\psi(x)$ in $\omega$ where $\psi(x)=1-x$. Following the approach used in Section~\ref{sec:convergenceBL}, the question of convergence of the asymptotic expansion reduces to finding suitable weight functions $p_\varepsilon$ and $q_\varepsilon$ satisfying~\eqref{eq:pqNd} which reduces in the considered one-dimensional case to
\begin{equation}\label{eq:pq}
    \begin{dcases}
    -q_\varepsilon''+\frac{q_\varepsilon'}{\varepsilon}+q_\varepsilon=b_\varepsilon\ge 0 & \text{in $\omega=(0,1)$}\\
    -p_\varepsilon''+p_\varepsilon=a_\varepsilon\ge 0 & \text{in $\mathcal{U}=(1,2)$}\\
    p_\varepsilon(1)=q_\varepsilon(1)\\
    p_\varepsilon'(1)=q_\varepsilon'(1)-\frac{q_\varepsilon(1)}\varepsilon\\
    p_\varepsilon \ge \beta>0 & \text{in $\mathcal{U}=(1,2)$}\\
    q_\varepsilon\ge \beta>0 & \text{in $\omega=(0,1)$}\\
    \|q_\varepsilon\|_{L^\infty(0,\delta)}\le C \text{ and }\|q_\varepsilon\|_{L^\infty(\delta,1)}\le \frac{C}\varepsilon,
    \end{dcases}
\end{equation}
where $0<\delta<1$.

In this one-dimensional case, we can exhibit suitable supersolutions satisfying~\eqref{eq:pq}: we take
\begin{equation}
q_\varepsilon(x)=1+\frac1\varepsilon e^{\frac{x-1}\varepsilon}-\frac1\varepsilon e^{-\frac1\varepsilon} \quad \text{in $\omega=(0,1)$},
\end{equation}
so that
\begin{equation}\label{eq:q}
    q_\varepsilon(x)\ge 1-\frac1\varepsilon e^{-\frac1\varepsilon} \to 1 \ \text{when $\varepsilon \to 0$}
\end{equation}
and
\begin{align}
    -q_\varepsilon''+\frac{q_\varepsilon'}{\varepsilon}+q_\varepsilon=q_\varepsilon \ge 0.
\end{align}

To satisfy the transmission conditions, we choose
\begin{equation}\label{eq:p}
p_\varepsilon(x)=q_\varepsilon(1)+\Bigl(q_\varepsilon'(1)-\frac{q_\varepsilon(1)}\varepsilon\Bigr) (x-1) \quad \text{in $\mathcal U=(1,2)$},
\end{equation}
so that
\begin{equation}
    p_\varepsilon(2)=q_\varepsilon(1)+q_\varepsilon'(1)-\frac{q_\varepsilon(1)}\varepsilon=\Bigl(1-\frac1\varepsilon\Bigr)\Bigl(1+\frac1\varepsilon-\frac1\varepsilon e^{-\frac1\varepsilon}\Bigr)+\frac1{\varepsilon ^2} =1+\frac1{\varepsilon^2}e^{-\frac1\varepsilon}\to 1 \ \text{when $\varepsilon\to0$}
\end{equation}
and
\begin{equation}
    -p_\varepsilon''+p_\varepsilon=p_\varepsilon \ge \min (p_\varepsilon(1),p_\varepsilon(2))= \min (q_\varepsilon(1),p_\varepsilon(2))\ge 0 \quad \text{for $\varepsilon$ small enough}.
\end{equation}

Thus the strict positivity of $p_\varepsilon$ and $q_\varepsilon$ for sufficiently small $\varepsilon>0$. Finally, for all $x\in (0,\delta)$ with $\delta<1$,
\begin{align*}
    |q_\varepsilon(x)| \le 1+\frac1\varepsilon e^{\frac{\delta-1}\varepsilon}-\frac1\varepsilon e^{-\frac1\varepsilon} \le C
\end{align*}
for sufficiently small $\varepsilon>0$, and, for $x\in(\delta,1)$,
\begin{align*}
    |q_\varepsilon(x)| \le 1+\frac1\varepsilon -\frac1\varepsilon e^{-\frac1\varepsilon} \le \frac{C}{\varepsilon}
\end{align*}
for sufficiently small $\varepsilon$, then the desired estimates on $\|q_\varepsilon\|_{L^\infty}$.  This completes the existence of suitable weight functions $p_\varepsilon$, $q_\varepsilon$ in the one-dimensional case.

\begin{remark}
    We thus recover the results obtained in~\cite{Bensiali2014} using a boundary layer approach. In fact, we obtain a finer description of what happens in the obstacle domain, which was not carried out in~\cite{Bensiali2014}.
\end{remark}

\begin{figure}[h!]
	\centering
	\includegraphics[scale=0.4]{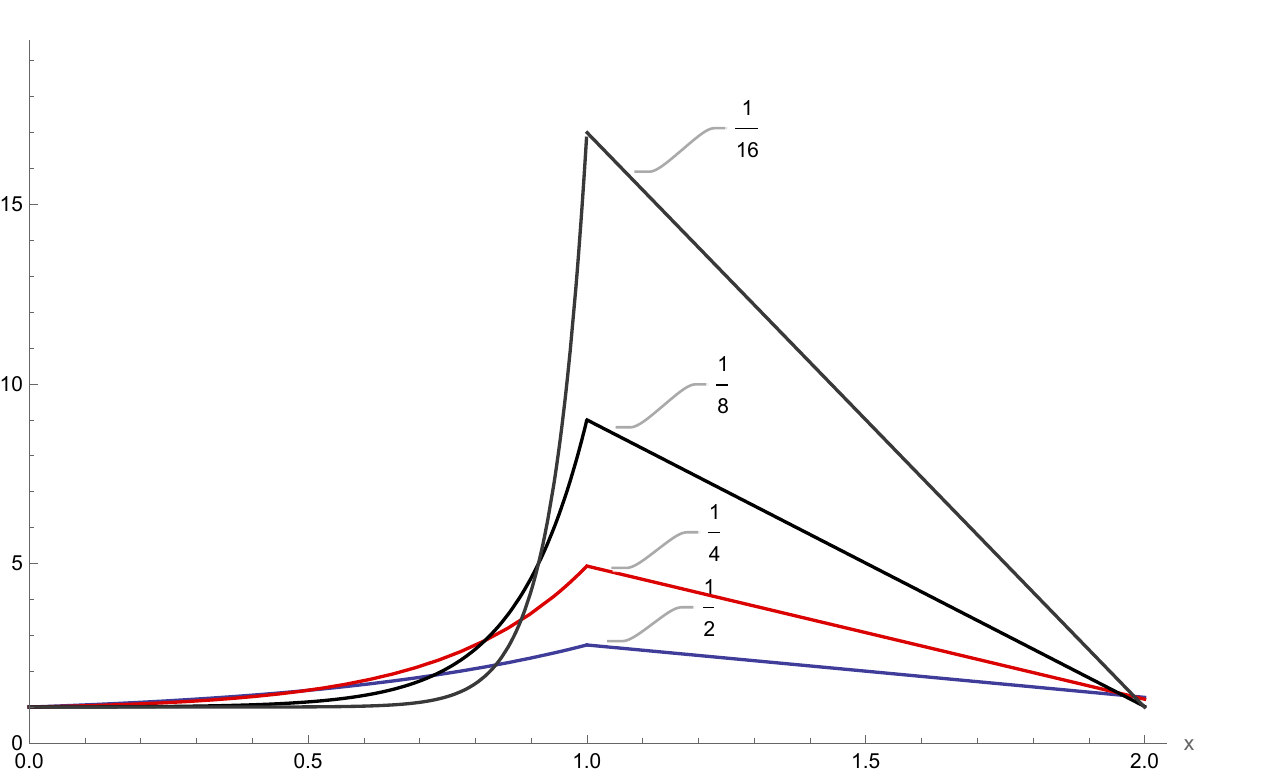}
	\caption{Plot of suitable supersolutions $q_\varepsilon$~\eqref{eq:q} in $(0,1)$ and $p_\varepsilon$~\eqref{eq:p} in $(1,2)$ satisfying~\eqref{eq:pq} in the one-dimensional case, for different values of $\varepsilon$.}
	\label{fig:annex_p_q_1d}
\end{figure}

\subsection{Spherical symmetric case (in 2D)}\label{sec:annexspecialcasessph}

\input{sections/subsections/BLspherical}

%% file: sections/subsections/BLspherical.tex
We consider here $\mathcal{U}=B(0,1)$ and $\Omega=B(0,2)$ in 2D. In this case, we can choose $n=\vec{e}_r=\nabla\psi$ where $\psi(x,y)=\sqrt{x^2+y^2}-1=r-1$ and look for weight functions~$p_\varepsilon$, $q_\varepsilon$ satisfying~\eqref{eq:pqNd} that are radial
\begin{align*}
p_\varepsilon(x,y)=p_\varepsilon(r) \quad \text{and} \quad
q_\varepsilon(x,y)=q_\varepsilon(r).
\end{align*}
Using polar coordinates,
\begin{equation}
\begin{dcases}
\Delta f=\frac1{r} \frac{\partial}{\partial r} \Bigl(r\frac{\partial f}{\partial r}\Bigr)+\frac1{r^2}\frac{\partial^2 f}{\partial\theta^2}=\frac{\partial^2 f}{\partial r^2}+\frac1{r} \frac{\partial f}{\partial r}+\frac1{r^2}\frac{\partial^2 f}{\partial\theta^2}\\
\mathrm{div}(A)=\frac1{r} \frac{\partial}{\partial r}\Bigl(r A_r\Bigr)+\frac1{r} \frac{\partial A_\theta}{\partial\theta}=\frac{\partial A_r}{\partial r}+\frac{A_r}{r}+\frac1{r} \frac{\partial A_\theta}{\partial\theta}.
\end{dcases}
\end{equation}
Thus, \eqref{eq:pqNd} rewrites in the radial symmetric case (see for instance~\cite{sauter2021heterogeneous})
\begin{equation}\label{eq:pqNdspherical}
    \begin{dcases}
    - p_\varepsilon''-\frac1{r} p_\varepsilon'+p_\varepsilon=a_\varepsilon(r)\ge 0 & \text{in $]0,1[$}\\
    - q_\varepsilon''-\Bigl(\frac{1}r+\frac1\varepsilon\Bigr)q_\varepsilon'+\Bigl(1-\frac1{\varepsilon r}\Bigr)q_\varepsilon=b_\varepsilon(r)\ge 0 & \text{in $]1,2[$}\\
    p_\varepsilon'(0)=0&\\
    p_\varepsilon(1)=q_\varepsilon(1) & \\
    p_\varepsilon'(1)=q_\varepsilon'(1)+\frac{q_\varepsilon(1)}\varepsilon & \\
    p_\varepsilon \ge \beta>0 & \text{in $]0,1[$}\\
    q_\varepsilon\ge \beta>0 & \text{in $]0,2[$}\\
    \|q_\varepsilon\|_{L^\infty(1+\delta,2)}\le C \text{ and }\|q_\varepsilon\|_{L^\infty(1,1+\delta)}\le \frac{C}\varepsilon.
    \end{dcases}
\end{equation}

We construct by hand suitable weight functions satisfying the previous conditions~\eqref{eq:pqNdspherical}: we take
\begin{equation}\label{eq:solpradial}
p_\varepsilon(r)=\frac{1}{\varepsilon} J_0(ir)
\end{equation}
where $J_0$ is the Bessel function of the first kind,
then $p_\varepsilon$ satisfies
\begin{equation}
\begin{dcases}
- p_\varepsilon''-\frac1{r} p_\varepsilon'+p_\varepsilon=0 & \text{in $]0,1[$}\\
p_\varepsilon'(0)=0\\
p_\varepsilon(r)\ge \frac{1}{\varepsilon} J_0(0)=\frac{1}{\varepsilon}>0 &\text{in $]0,1[$}
\end{dcases}
\end{equation}
using the properties of the Bessel functions~\cite{abramovitz1972handbook}. And we take $q_\varepsilon$ of the form 
\begin{equation}\label{eq:solqradial}
q_\varepsilon(r)=\frac{d_\varepsilon}r + \frac{e_\varepsilon}{\varepsilon} e^{-\frac{r-1}\varepsilon}.
\end{equation}
The transmission conditions at $r=1$ give, since $\frac{d}{dr}\bigl(J_0(i r)\bigr)=-i J_1(ir)$,
\begin{equation}\label{eq:coeffsolqradial}
\begin{dcases}
d_\varepsilon=\frac{i J_1(i)}{\varepsilon-1}\to -i J_1(i)>0\\
e_\varepsilon=J_0(i)+\frac{\varepsilon i J_1(i)}{\varepsilon-1}\to J_0(i)>0,
\end{dcases}
\end{equation}
when $\varepsilon\to 0^+$. Therefore, for sufficiently small $\varepsilon$,
\begin{equation}
q_\varepsilon(r)\ge \frac{d_\varepsilon}r \ge \frac{C}{2}>0 \qquad \text{in $]1,2[$}.
\end{equation}
One then has
\begin{align}
    - q_\varepsilon''-\Bigl(\frac{1}r+\frac1\varepsilon\Bigr)q_\varepsilon'+\Bigl(1-\frac1{\varepsilon r}\Bigr)q_\varepsilon&=q_\varepsilon+d_\varepsilon \Bigl(-\Bigl(\frac1r\Bigr)''-\frac1r\Bigl(\frac1r\Bigr)'\Bigr) & \\
    &=q_\varepsilon+d_\varepsilon \Bigl(\frac{-1}{r^3}\Bigr) & \\
    &=\frac{e_\varepsilon}\varepsilon e^{-\frac{r-1}\varepsilon}+d_\varepsilon \frac{r^2-1}{r^3} \ge0 & \text {in $]1,2[$}
\end{align}
for sufficiently small $\varepsilon$. The previous calculations follow from the fact that $r\mapsto \frac1r$ solves
$$-u'-\frac1r u=0$$
and $r\mapsto e^{-\frac{r-1}\varepsilon}$ solves
$$- u''-\Bigl(\frac{1}r+\frac1\varepsilon\Bigr)u'-\frac1{\varepsilon r}u=0.$$

Finally, the estimates on $q_\varepsilon$ in~\eqref{eq:pqNdspherical} are also satisfied since, for sufficiently small $\varepsilon$,
\begin{equation}
|q_\varepsilon(r)|\le C+C' \frac{1}\varepsilon e^{-\frac{\delta}\varepsilon} \to 0, \quad \text{for $r\in{]1+\delta},2[$}
\end{equation}
and 
\begin{equation}
|q_\varepsilon(r)|\le C+C' \frac{1}\varepsilon \le \frac{C}\varepsilon, \quad \text{for $r\in{]1},1+\delta[$}.
\end{equation}

\begin{remark}
    The approach used here could also be used in the one-dimensional case. The approach could also be adapted to the spherical symmetric case in $n$ dimensions. 
\end{remark}

\begin{remark}
    The choices made for $q_\varepsilon$ in the special cases (1D or spherical symmetric case) will be understood in the following appendix where we use a formal boundary layer approach for the dual problem.
\end{remark}


\begin{figure}[!h]
	\centering
	\begin{subfigure}[b]{0.49\textwidth}
		\centering
		\includegraphics[width=\textwidth]{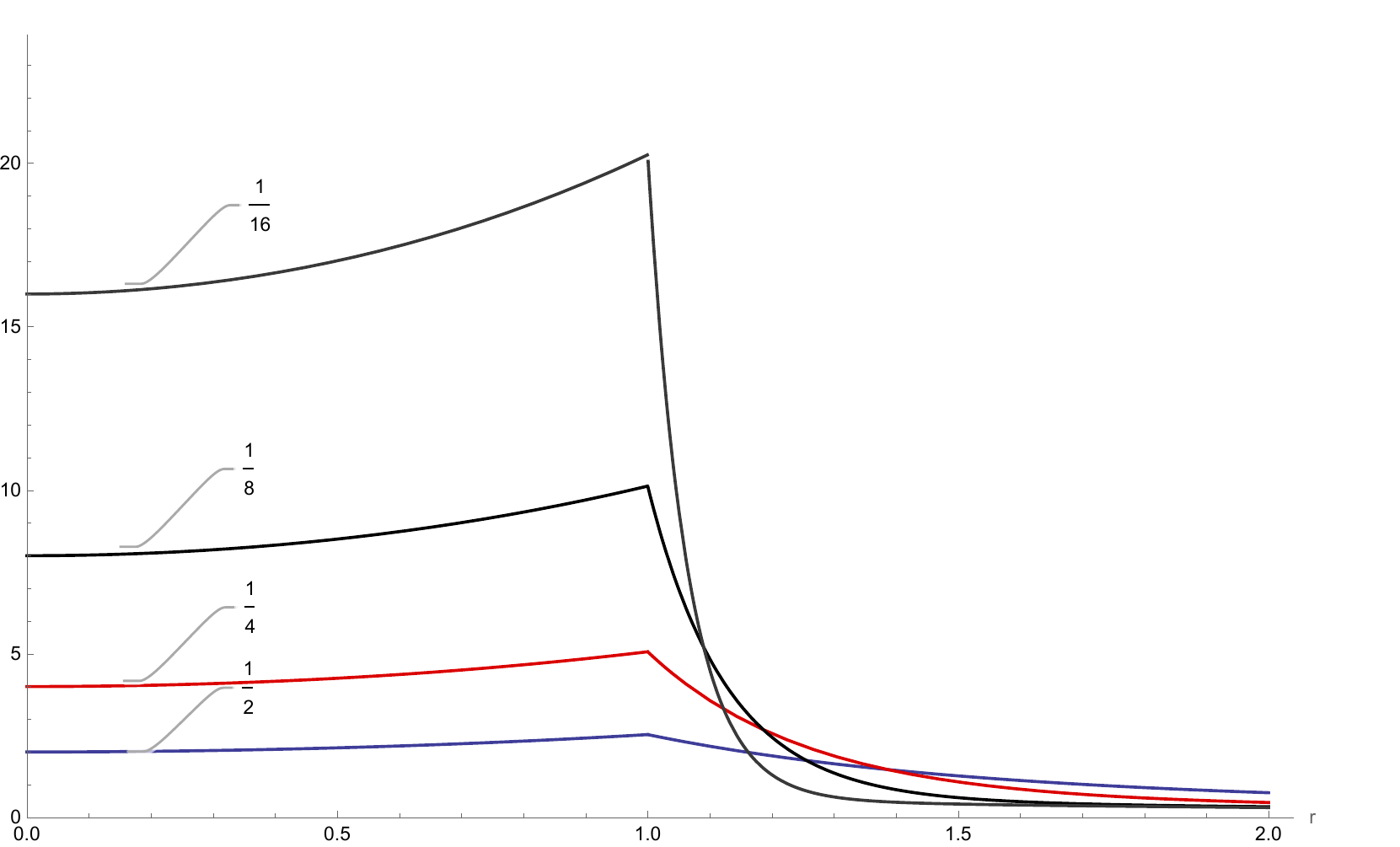}
		\caption{Plots of $p_\varepsilon(r)$ and $q_\varepsilon(r)$}
		\label{fig:p_q_sph_2d}
	\end{subfigure}
	\begin{subfigure}[b]{0.49\textwidth}
		\centering
		\includegraphics[width=\textwidth]{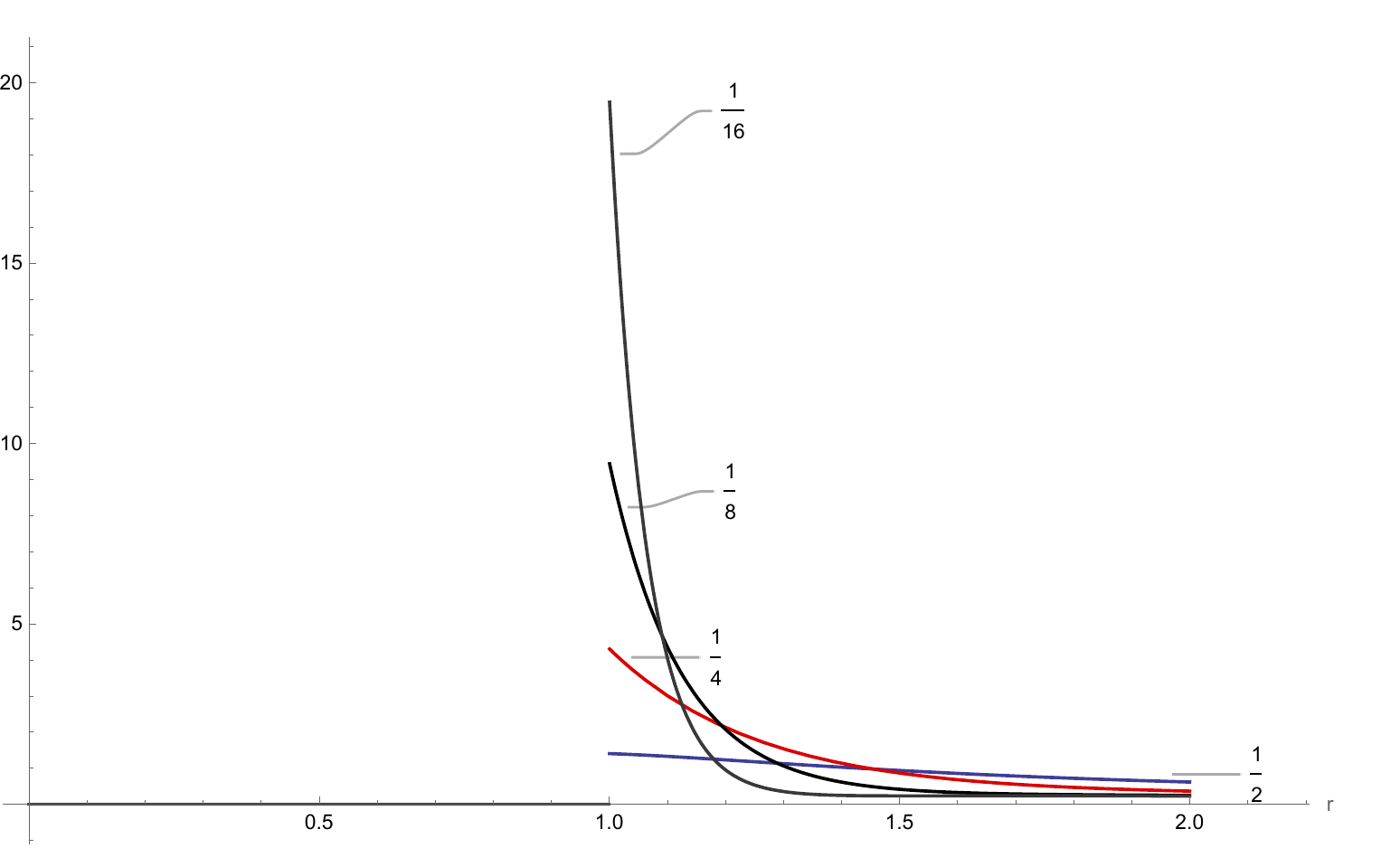}
		\caption{Plots of $a_\varepsilon(r)$ and $b_\varepsilon(r)$}
		\label{fig:Eq_p_q_sph_2d}
	\end{subfigure}
	\caption{Plot of suitable supersolutions $p_\varepsilon(r)$~\eqref{eq:solpradial} in $(0,1)$ and $q_\varepsilon(r)$~\eqref{eq:solqradial} in $(1,2)$ satisfying~\eqref{eq:pqNdspherical} in the spherical symmetric case (2D), for different values of $\varepsilon$.}
	\label{fig:p_q_sph_2d_Eq}
\end{figure}

%% file: sections/appendixDual.tex
\section{Formal boundary layer approach for the dual problem}\label{sec:annexBLdual}

We are going to use again a boundary layer approach for the dual problem
    \begin{subnumcases}{}
    -\Delta p_\varepsilon+p_\varepsilon=a_\varepsilon=\frac{a^{-1}(x)}\varepsilon+a^0(x)+\ldots\ge 0 & \text{in $\mathcal{U}$}\label{eq:pNd}\\ 
    -\Delta q_\varepsilon-\frac{1}{\varepsilon}\mathrm{div}(q_\varepsilon n)+q_\varepsilon=b_\varepsilon=\frac{b^{-1}(x)}\varepsilon+b^0(x)+\ldots\ge 0 & \text{in $\omega$}\label{eq:qNd}\\
    p_\varepsilon=q_\varepsilon & \text{on $\partial\mathcal U$}\label{eq:transmission1pqNd}\\
    \frac{\partial p_\varepsilon}{\partial\nu}=\frac{\partial q_\varepsilon}{\partial\nu}+\frac{q_\varepsilon}\varepsilon & \text{on $\partial\mathcal U$}\label{eq:transmission2pqNd}\\    q_\varepsilon=q_0(x)+\varepsilon q_1(x)+\ldots& \text{ on $\partial\Omega$}\label{eq:dirichletqNd}
    \end{subnumcases}
The reader might think that we are going around in circles, but the goal here is not to estabish a convergence for the dual problem, but simply to use the boundary layer approach in order to construct an approximate solution which we wish to be a supersolution and satisfy the expected properties in~\eqref{eq:pqNd}. 

Based on numerical experiments, we consider the following ansatz for the solution (boundary layer near $\partial\mathcal U$):
\begin{equation}\label{eq:dualNdansatz}
    \begin{dcases}
    p_\varepsilon(x)=\frac1\varepsilon P^{-1}(x)+P^0(x)+\varepsilon P^1(x)+\ldots,\\
    q_\varepsilon(x)=\frac1\varepsilon Q^{-1}\Bigl(x,\frac{\psi(x)}\varepsilon\Bigr)+Q^0\Bigl(x,\frac{\psi(x)}\varepsilon\Bigr)+\varepsilon Q^1\Bigl(x,\frac{\psi(x)}\varepsilon\Bigr)+\ldots
    \end{dcases}
\end{equation}
where $\psi$ is somehow the distance to $\partial\mathcal U$, and the profile terms inside $\omega$ have the form 
\begin{equation}\label{eq:dualNdboundarylayer}
    \begin{dcases}
    Q^i(x,z)=\overline{Q}^i(x)+\widetilde{Q}^i(x,z)\\
    \text{where} \ \forall i,k\ge0, \ \partial_z^k \widetilde{Q}^i \xrightarrow[z \to +\infty]{} 0, \quad \forall x\in\omega,
    \end{dcases}
\end{equation}

Using the same notations as before, for a function $(x,z)\mapsto Q(x,z)$, the function $x\mapsto q(x)=Q\bigl(x,\frac{\psi(x)}\varepsilon\bigr)$ satisfies
\begin{equation}\label{eq:dualderivativesNd}
    \begin{dcases}
    \nabla q=\nabla Q+\frac{1}\varepsilon Q_z\nabla\psi\\ 
    \Delta q=\Delta Q+\frac{2}\varepsilon  \nabla Q_z\cdot\nabla\psi+\frac1\varepsilon \Delta\psi Q_z+\frac{1}{\varepsilon^2} |\nabla\psi|^2Q_{zz},
    \end{dcases}
\end{equation}
where the derivatives of $q$ are evaluated at $x$ while the derivatives of $Q$ are evaluated at $\bigl(x,\frac{\psi(x)}{\varepsilon}\bigr)$.

For simplicity, we assume $\psi$ to be sufficiently regular and $n=\nabla\psi$ in $\omega$, we also assume $\nabla\psi\ne 0$ in $\omega$. In this cas, using
\begin{align*}
\mathrm{div}(q_\varepsilon n)&=q_\varepsilon \mathrm{div}(n)+n\cdot\nabla q_\varepsilon\\
&=q_\varepsilon \Delta\psi +\nabla q_\varepsilon\cdot\nabla \psi,
\end{align*}
Eq.~\eqref{eq:qNd} reduces to
\begin{equation}\label{eq:dualEqqNd}
-\Delta q_\varepsilon-\frac{1}{\varepsilon}\nabla q_\varepsilon\cdot\nabla \psi-\frac{1}{\varepsilon}\Delta\psi q_\varepsilon+q_\varepsilon=b_\varepsilon=\frac{b^{-1}(x)}\varepsilon+b^0(x)+\ldots\ge 0 \quad \text{in $\omega$}.
\end{equation}

\subsection{Determination of profiles}

\subsubsection{Asymptotic expansion of Eq.~\eqref{eq:dualEqqNd} inside the obstacle $\omega$}

\begin{itemize}
    \item Order $\varepsilon^{-3}$: identifying the terms corresponding to the power $-3$ of $\varepsilon$ leads to 
    $$-|\nabla\psi|^2 Q^{-1}_{zz}-|\nabla\psi|^2Q^{-1}_z =0,$$
    i.e.
    $$-Q^{-1}_{zz}-Q^{-1}_z =0.$$
    Using the decomposition~\eqref{eq:dualNdboundarylayer} one obtains
    \begin{equation}\label{eq:Qtildem1Nd}
        -\widetilde{Q}_{zz}^{-1}-\widetilde{Q}_z^{-1}={0} \quad \text{in $\omega\times\RR^+$}.
    \end{equation}
    
    \item Order $\varepsilon^{-2}$:
    $$-|\nabla\psi|^2 Q^0_{zz}-\Delta\psi Q^{-1}_z-2 \nabla Q^{-1}_z\cdot\nabla\psi-Q^0_z|\nabla\psi|^2-\nabla Q^{-1}\cdot\nabla\psi- \Delta\psi Q^{-1}=0,$$
    which, by using hypothesis~\eqref{eq:dualNdboundarylayer} and taking the limit when $z\to +\infty$, leads to
    $$
        -\nabla \overline{Q}^{-1}\cdot \nabla\psi-\Delta\psi\overline{Q}^{-1} =0 \quad \text{in $\omega$},
    $$
    i.e.
    \begin{equation}\label{eq:Qbarm1Nd}
    \mathrm{div}(\overline{Q}^{-1} n)=0 \quad \text{in $\omega$},
    \end{equation}
    and by difference, we obtain
        \begin{equation}\label{eq:Qtilde0Nd}
        -|\nabla\psi|^2\widetilde{Q}_{zz}^0-\Delta\psi\widetilde{Q}_{z}^{-1} -2 \nabla \widetilde{Q}^{-1}_z\cdot\nabla\psi-\widetilde{Q}^0_z|\nabla\psi|^2-\nabla\widetilde{Q}_z^{-1}\cdot \nabla\psi-\Delta\psi \widetilde{Q}^{-1}=0 \quad \text{in $\omega\times\RR^+$}.
    \end{equation}
    \item Order $\varepsilon^{-1}$:
    $$-|\nabla\psi|^2 Q^1_{zz}-\Delta\psi Q^0_{z} -2 \nabla Q^0_z\cdot \nabla\psi-\Delta Q^{-1}-Q^1_z |\nabla\psi|^2-\nabla Q^0\cdot \nabla\psi-\Delta\psi Q^0+Q^{-1}=b^{-1},$$
    and again using the same reasoning as above, by using hypothesis~\eqref{eq:dualNdboundarylayer} and $z\to +\infty$, we obtain
    $$
        -\Delta\overline{Q}^{-1}-\nabla\overline{Q}^0\cdot\nabla\psi-\Delta\psi \overline{Q}^0+\overline{Q}^{-1}=b^{-1} \quad \text{in $\omega$},
    $$
    i.e.
    \begin{equation}\label{eq:Qbar0Nd}
    \nabla\overline{Q}^0\cdot\nabla\psi+\Delta\psi \overline{Q}^0=\mathrm{div}(\overline{Q}^0 n)=-\Delta\overline{Q}^{-1}+\overline{Q}^{-1}-b^{-1} \quad \text{in $\omega$},
    \end{equation}
    and by difference,
        \begin{equation}\label{eq:Qtilde1Nd}
        -|\nabla\psi|^2\widetilde{Q}_{zz}^1-\Delta\psi \widetilde{Q}_{z}^0-2 \nabla\widetilde{Q}^0_z\cdot\nabla\psi-\Delta\widetilde{Q}^{-1}-\widetilde{Q}^1_z |\nabla\psi|^2-\nabla\widetilde{Q}^0\cdot\nabla\psi-\Delta\psi \widetilde{Q}^0+\widetilde{Q}^{-1}=0 \quad \text{in $\omega\times\RR^+$}.
    \end{equation}
    \item Order $\varepsilon^0$: from
    $$-|\nabla\psi|^2Q^2_{zz}-\Delta\psi Q^1_{z}-2 \nabla Q^1_z\cdot\nabla\psi-\Delta Q^0-Q^2_z |\nabla\psi|^2-\nabla Q^1\cdot\nabla\psi -\Delta\psi Q^1+Q^0=b^0,$$
    we deduce once again as before,
    $$
        -\Delta\overline{Q}^0-\nabla\overline{Q}^1\cdot\nabla\psi-\Delta\psi \overline {Q}^1+\overline{Q}^0=b^0 \quad \text{in $\omega$},
    $$
    i.e.
    \begin{equation}\label{eq:Qbar1Nd}
    \nabla\overline{Q}^1\cdot\nabla\psi+\Delta\psi \overline{Q}^1=\mathrm{div}(\overline{Q}^1 n)=-\Delta\overline{Q}^{0}+\overline{Q}^{0}-b^{0} \quad \text{in $\omega$},
    \end{equation}
    and by difference,
        \begin{equation}\label{eq:Qtilde2Nd}
        -|\nabla\psi|^2\widetilde{Q}_{zz}^2-\Delta\psi \widetilde{Q}_{z}^1-2 \nabla\widetilde{Q}^1_z\cdot\nabla\psi-\Delta\widetilde{Q}^0-\widetilde{Q}^2_z|\nabla\psi|^2-\nabla\widetilde{Q}^1\cdot \nabla\psi-\Delta\psi\widetilde{Q}^1+\widetilde{Q}^0=0 \quad \text{in $\omega\times\RR^+$}.
    \end{equation}
    \item Order $\varepsilon^1$: from
    $$-|\nabla\psi|^2Q^3_{zz}-\Delta\psi Q^2_{z}-2 \nabla Q^2_z\cdot\nabla\psi-\Delta Q^1-Q^3_z |\nabla\psi|^2-\nabla Q^2\cdot\nabla\psi -\Delta\psi Q^2+Q^1=b^1,$$
    we deduce once again as before,
    $$
        -\Delta\overline{Q}^1-\nabla\overline{Q}^2\cdot\nabla\psi-\Delta\psi \overline {Q}^2+\overline{Q}^2=b^1 \quad \text{in $\omega$},
    $$
    i.e.
    \begin{equation}\label{eq:Qbar2Nd}
    \nabla\overline{Q}^2\cdot\nabla\psi+\Delta\psi \overline{Q}^2=\mathrm{div}(\overline{Q}^2 n)=-\Delta\overline{Q}^{1}+\overline{Q}^{1}-b^{1} \quad \text{in $\omega$},
    \end{equation}
    and by difference,
        \begin{equation}\label{eq:Qtilde3Nd}
        -|\nabla\psi|^2\widetilde{Q}_{zz}^3-\Delta\psi \widetilde{Q}_{z}^2-2 \nabla\widetilde{Q}^2_z\cdot\nabla\psi-\Delta\widetilde{Q}^1-\widetilde{Q}^3_z|\nabla\psi|^2-\nabla\widetilde{Q}^2\cdot \nabla\psi-\Delta\psi\widetilde{Q}^2+\widetilde{Q}^1=0 \quad \text{in $\omega\times\RR^+$}.
    \end{equation}
    
\end{itemize}

\subsubsection{Asymptotic expansion of Eq.~\eqref{eq:qNd} inside the fluid domain $\mathcal{U}$}

\begin{itemize}
    \item Order $\varepsilon^{-1}$:
    \begin{equation}\label{eq:Pm1Nd}
        -\Delta P^{-1}+P^{-1}=a^{-1}(x) \quad \text{in $\mathcal{U}$}.
    \end{equation}
    \item Order $\varepsilon^{0}$:
    \begin{equation}\label{eq:P0Nd}
        -\Delta P^0+P^0=a^0(x) \quad \text{in $\mathcal{U}$}.
    \end{equation}
    \item Order $\varepsilon^{1}$:
    \begin{equation}\label{eq:P1Nd}
        -\Delta P^1+P^1=a^1(x) \quad \text{in $\mathcal{U}$}.
    \end{equation}
\end{itemize}

\subsubsection{Asymptotic expansion of Eq.~\eqref{eq:transmission1pqNd} (transmission condition on $\partial\mathcal{U}$)}

We obtain simply
\begin{align*}
Q^{-1}(x,0)&=P^{-1}(x)\\
Q^{0}(x,0)&=P^{0}(x)  \qquad \text{on $\partial\mathcal U$}\\
Q^{1}(x,0)&=P^{1}(x)
\end{align*}
i.e.
\begin{align}\label{eq:expansiontransmission1pqNd}
\overline{Q}^{-1}(x)+\widetilde{Q}^{-1}(x,0)&=P^{-1}(x)\\
\overline{Q}^{0}(x,0)+\widetilde{Q}^{0}(x,0)&=P^{0}(x)  \qquad\text{on $\partial\mathcal U$}.\\
\overline{Q}^{1}(x,0)+\widetilde{Q}^{1}(x,0)&=P^{1}(x)
\end{align}

\subsubsection{Asymptotic expansion of Eq.~\eqref{eq:transmission2pqNd} (transmission condition on $\partial\mathcal{U}$)} We have $$\nabla p_\varepsilon\cdot\nu=\nabla q_\varepsilon\cdot\nu+\frac{q_\varepsilon}\varepsilon=\nabla q_\varepsilon\cdot\nu+\frac{p_\varepsilon}\varepsilon \quad \text{on $\partial\mathcal U$}.$$

\begin{itemize}
    \item Order $\varepsilon^{-2}$
    $$0=Q^{-1}_z \nabla\psi\cdot \nu+P^{-1}=Q^{-1}_z+P^{-1}$$ 
    
    \item Order $\varepsilon^{-1}$
    $$\nabla P^{-1}\cdot\nu=Q^{0}_z \nabla\psi\cdot \nu+\nabla Q^{-1}\cdot\nu+P^{0}=Q^{0}_z +\nabla Q^{-1}\cdot\nu+P^{0}$$ 

    \item Order $\varepsilon^{0}$
    $$\nabla P^{0}\cdot\nu=Q^{1}_z \nabla\psi\cdot \nu+\nabla Q^{0}\cdot\nu+P^{1}=Q^{1}_z +\nabla Q^{0}\cdot\nu+P^{1}.$$
\end{itemize}
Thus
\begin{align}\label{eq:expansiontransmission2pqNd}
\widetilde{Q}^{-1}_z+\overline{Q}^{-1}+\widetilde{Q}^{-1}=0&\\
\nabla P^{-1}\cdot\nu=\widetilde{Q}^{0}_z +\nabla \overline{Q}^{-1}\cdot\nu+\nabla \widetilde{Q}^{-1}\cdot\nu+\overline{Q}^{0}+\widetilde{Q}^{0} & \qquad\text{on $\partial\mathcal U\times\{z=0\}$.}\\
\nabla P^{0}\cdot\nu=\widetilde{Q}^{1}_z +\nabla \overline{Q}^{0}\cdot\nu+\nabla \widetilde{Q}^{0}\cdot\nu+\overline{Q}^{1}+\widetilde{Q}^{1}&
\end{align}

\subsubsection{Asymptotic expansion of Eq.~\eqref{eq:dirichletqNd} (Dirichlet boundary conditions on $\partial\Omega$)}

By neglecting $\widetilde{Q}^i\bigl(x,\frac{\psi(x)}\varepsilon\bigr)\to 0$ when $\varepsilon\to0$ (we will find later an exponential decrease), we obtain

\begin{equation}\label{eq:expansiondirichletqNd}
\begin{dcases}
    \overline{Q}^{-1}(x)=0 &\\
    \overline{Q}^0(x)=q_0 &  \text{on $\partial\Omega$}.\\
    \overline{Q}^1(x)=q_1 &
    \end{dcases}
\end{equation}

\subsection{Formal resolution of the profile equations}

\subsubsection{Determination of $Q^{-1}$ and $P^{-1}$}
From~\eqref{eq:Qbarm1Nd} and~\eqref{eq:expansiondirichletqNd}, we obtain
\begin{equation}\label{eq:solQbarm1Nd}
    \overline{Q}^{-1}=0 \quad \text{in $\omega$}.
\end{equation}

From~\eqref{eq:Qtildem1Nd} and the fact that $\widetilde{Q}^{-1}$ tends to $0$ as $z\to+\infty$ , we obtain
\begin{equation*}
\widetilde{Q}^{-1}(x,z)=A(x) e^{-z} \quad \text{in $\omega\times \RR^+$}.
\end{equation*}
The condition~\eqref{eq:expansiontransmission2pqNd} rewrites
$$\widetilde{Q}^{-1}_z(x,0)+\widetilde{Q}^{-1}(x,0)=0 \quad \text{on $\partial\mathcal U$}$$
and is satisfied since $\widetilde{Q}^{-1}$ satisfy everywhere
\begin{equation}\label{eq:Qm1tildeEq}
\widetilde{Q}^{-1}_z(x,z)=-\widetilde{Q}^{-1}(x,z) \quad \text{in $\omega\times\RR^+$}.
\end{equation}

From~\eqref{eq:Pm1Nd} and~\eqref{eq:expansiontransmission1pqNd}, we obtain that $P^{-1}$ is solution to
\begin{equation}\label{eq:solPm1Nd}
    \begin{dcases}
        -\Delta P^{-1}+P^{-1}=a^{-1} & \text{in $\mathcal{U}$}\\
        P^{-1}(x)=\widetilde{Q}^{-1}(x,0)=A(x) & \text{on $\partial\mathcal{U}$},
    \end{dcases}
\end{equation}
and a possible choice for $\widetilde{Q}^{-1}$ is given by
\begin{equation}\label{eq:solQtildem1Nd}
\widetilde{Q}^{-1}(x,z)=\underline{P}^{-1}(x) e^{-z}\quad \text{in $\omega\times\RR^+$}, 
\end{equation}
where $\underline{P}^{-1}$ is a simultaneous lifting in $\omega$ of the trace of $P^{-1}$  and of that of $\frac{\partial P^{-1}}{\partial\nu}$ on $\partial\mathcal U$ and having the same regularity as $P^{-1}$.

Notice that $A(x)$ on $\partial\mathcal{U}$ is not yet determined at this stage. It will be using the equations on $P^0,Q^0$.

\subsubsection{Determination of $Q^0$ and $P^0$}
Similarly, from~\eqref{eq:Qbar0Nd} and~\eqref{eq:expansiondirichletqNd}, we obtain that $\overline{Q}^0$ is solution to
\begin{equation}\label{eq:solQ0Nd}
    \begin{dcases}
        \mathrm{div}(\overline{Q}^0 n)=-\Delta \overline{Q}^{-1}+\overline{Q}^{-1}-b^{-1}=-b^{-1} & \text{in $\omega$}\\
        \overline{Q}^0(x)=q_0(x) & \text{on $\partial\Omega$}.
    \end{dcases}
\end{equation}

From~\eqref{eq:Qtilde0Nd}, we have
\begin{align*}
        -|\nabla\psi|^2\widetilde{Q}_{zz}^0-|\nabla\psi|^2\widetilde{Q}^0_z&=\Delta\psi\widetilde{Q}_{z}^{-1} +2 \nabla \widetilde{Q}^{-1}_z\cdot\nabla\psi+\nabla\widetilde{Q}_z^{-1}\cdot \nabla\psi+\Delta\psi \widetilde{Q}^{-1}\\
        &=-\nabla\widetilde{Q}_z^{-1}\cdot \nabla\psi\\
        &=-\nabla\underline{P}^{-1}\cdot \nabla\psi\, e^{-z}\\
        -\widetilde{Q}_{zz}^0-\widetilde{Q}^0_z&=-\frac{\nabla\underline{P}^{-1}\cdot \nabla\psi}{|\nabla\psi|^2}\, e^{-z}=F(x)\, e^{-z},
\end{align*}
using~\eqref{eq:Qm1tildeEq} and the obtained expression of $\widetilde{Q}^{-1}$~\eqref{eq:solQtildem1Nd}. The solution which tends to $0$ as $z\to+\infty$ is given by
\begin{equation}\label{eq:solQ0tildeNd}
    \widetilde{Q}^0(x,z)=B(x) e^{-z}+F(x) z e^{-z} \quad \text{in $\omega\times\RR^+$},
\end{equation}
where $B(x)$ is to be determined.

Using~\eqref{eq:expansiontransmission2pqNd}, we have
\begin{align*}
\nabla P^{-1}\cdot\nu=\widetilde{Q}^{0}_z(x,0) +\nabla \overline{Q}^{-1}\cdot\nu+\nabla \widetilde{Q}^{-1}\cdot\nu+\overline{Q}^{0}(x)+\widetilde{Q}^{0}(x,0) & \qquad\text{on $\partial\mathcal U$.}
\end{align*}
i.e.
\begin{align*}
\widetilde{Q}^{0}_z(x,0)+\widetilde{Q}^{0}(x,0)&=\nabla P^{-1}\cdot\nu-\nabla \overline{Q}^{-1}\cdot\nu-\overline{Q}^{0}(x) & \text{on $\partial\mathcal U$}\\
&=(\nabla P^{-1}\cdot\nu)_{|\mathcal U}-(\nabla \underline{P}^{-1}\cdot\nu)_{|\omega}-\overline{Q}^{0}(x)& \text{on $\partial\mathcal U$}\\
&=-\overline{Q}^{0}(x)& \text{on $\partial\mathcal U$}.
\end{align*}
Using the expression found for $\widetilde{Q}^0$, we have
\begin{align*}
\widetilde{Q}^{0}_z(x,z)+\widetilde{Q}^{0}(x,z)&=-B(x) e^{-z}+F(x) e^{-z}-F(x) z e^{-z}+B(x) e^{-z}+F(x) z e^{-z}\\
&=F(x)\, e^{-z}.
\end{align*}
Thus, 
\begin{align*}
F(x)&=-\overline{Q}^{0}(x)& \text{on $\partial\mathcal U$}\\
-\frac{\nabla\underline{P}^{-1}\cdot \nabla\psi}{|\nabla\psi|^2}&=-\overline{Q}^{0}(x)& \text{on $\partial\mathcal U$}\\
\frac{\partial P^{-1}}{\partial\nu}&=\overline{Q}^{0}& \text{on $\partial\mathcal U$}.
\end{align*}
since $\nabla\psi=\nu$ on $\partial\mathcal U$. This fixes $A(x)$ on $\partial\mathcal U$, since we obtain that $P^{-1}$ is solution to
\begin{equation}\label{eq:solPm1Ndbis}
    \begin{dcases}
        -\Delta P^{-1}+P^{-1}=a^{-1} & \text{in $\mathcal{U}$}\\
        \frac{\partial P^{-1}}{\partial\nu}=\overline{Q}^{0} & \text{on $\partial\mathcal{U}$},
    \end{dcases}
\end{equation}
and $A(x)$ is then the trace of $P^{-1}$ on $\partial\mathcal U$.

From~\eqref{eq:P0Nd} and~\eqref{eq:expansiontransmission1pqNd}, we obtain that $P^{0}$ is solution to
\begin{equation}\label{eq:solP0Nd}
    \begin{dcases}
        -\Delta P^{0}+P^{0}=a^{0} & \text{in $\mathcal{U}$}\\
        P^{0}(x)=\widetilde{Q}^{0}(x,0)+\overline{Q}^0(x)=B(x)+\overline{Q}^0(x) & \text{on $\partial\mathcal{U}$}.
    \end{dcases}
\end{equation}
Thus we can take $B(x)=\underline{P}^0-\overline{Q}^0$ in $\omega$, where $\underline{P}^{0}$ is a simultaneous lifting in $\omega$ of the trace of $P^{0}$  and of that of $\frac{\partial P^{0}}{\partial\nu}$ on $\partial\mathcal U$ and having the same regularity as $P^{0}$. The equation satisfied by $P^0$ will be determined using the equations on $P^1,Q^1$.

\subsubsection{Determination of $Q^1$ and $P^1$}
From~\eqref{eq:Qbar1Nd} and~\eqref{eq:expansiondirichletqNd}, we obtain that $\overline{Q}^1$ is solution to
\begin{equation}\label{eq:solQ1Nd}
    \begin{dcases}
        \mathrm{div}(\overline{Q}^1 n)=-\Delta \overline{Q}^{0}+\overline{Q}^{0}-b^0 & \text{in $\omega$}\\
        \overline{Q}^1(x)=q_1(x) & \text{on $\partial\Omega$}.
    \end{dcases}
\end{equation}

From~\eqref{eq:Qtilde1Nd} and the obtained expression of $\widetilde{Q}^{0}$~\eqref{eq:solQ0tildeNd}, we have
\begin{align*}
        -|\nabla\psi|^2\widetilde{Q}_{zz}^1-|\nabla\psi|^2\widetilde{Q}^1_z &=\Delta\psi \widetilde{Q}_{z}^0+2 \nabla\widetilde{Q}^0_z\cdot\nabla\psi+\Delta\widetilde{Q}^{-1}+\nabla\widetilde{Q}^0\cdot\nabla\psi+\Delta\psi \widetilde{Q}^0-\widetilde{Q}^{-1}\\
        &=\Delta\psi F(x)\, e^{-z}+(\Delta \underline{P}^{-1}-\underline{P}^{-1}) e^{-z}+2 \nabla(F(x) e^{-z}-\widetilde{Q}^0)\cdot\nabla\psi+\nabla\widetilde{Q}^0\cdot\nabla\psi\\
        &=\Delta\psi F(x)\, e^{-z}+(\Delta \underline{P}^{-1}-\underline{P}^{-1}) e^{-z}-\nabla\widetilde{Q}^0\cdot\nabla\psi+2\nabla F(x) \cdot\nabla\psi e^{-z}\\
        &=\Delta\psi F\, e^{-z}+(\Delta \underline{P}^{-1}-\underline{P}^{-1}) e^{-z}-\nabla B\cdot\nabla\psi\, e^{-z}-\nabla F\cdot\nabla\psi\, ze^{-z}+2\nabla F \cdot\nabla\psi\, e^{-z}\\
        -\widetilde{Q}_{zz}^1-\widetilde{Q}^1_z&=\frac{\Delta\psi F+\Delta \underline{P}^{-1}-\underline{P}^{-1}-\nabla B\cdot\nabla\psi+2\nabla F \cdot\nabla\psi}{|\nabla\psi|^2}\, e^{-z} - \frac{\nabla F\cdot\nabla\psi}{|\nabla\psi|^2}\, z e^{-z}\\
        &=G(x)\, e^{-z} +H(x)\, z e^{-z}.
\end{align*}
 The solution which tends to $0$ as $z\to+\infty$ is given by
\begin{equation*}
    \widetilde{Q}^1(x,z)=C(x) e^{-z}+(G(x)+H(x)) z e^{-z}+\frac{H(x)}2 z^2 e^{-z} \quad \text{in $\omega\times\RR^+$},
\end{equation*}
where $C(x)$ is to be determined.

Using~\eqref{eq:expansiontransmission2pqNd}, we have
\begin{align*}
\nabla P^{0}\cdot\nu=\widetilde{Q}^{1}_z +\nabla \overline{Q}^{0}\cdot\nu+\nabla \widetilde{Q}^{0}\cdot\nu+\overline{Q}^{1}+\widetilde{Q}^{1} & \qquad\text{on $\partial\mathcal U$.}
\end{align*}
i.e.
\begin{align*}
\widetilde{Q}^{1}_z(x,0)+\widetilde{Q}^{1}(x,0)&=\nabla P^{0}\cdot\nu-\nabla \overline{Q}^{0}\cdot\nu -\nabla \widetilde{Q}^{0}(x,0)\cdot\nu-\overline{Q}^{1}& \text{on $\partial\mathcal U$}\\
&=\nabla P^{0}\cdot\nu-\nabla \overline{Q}^{0}\cdot\nu -\nabla B\cdot\nu-\overline{Q}^{1}& \text{on $\partial\mathcal U$}.
\end{align*}
Using the expression found for $\widetilde{Q}^1$, we have
\begin{align*}
[\widetilde{Q}^{1}_z(x,z)+\widetilde{Q}^{1}(x,z)]_{z=0}&=-C(x)+(G(x)+H(x))+C(x)\\
&=G(x)+H(x).
\end{align*}
Thus
\begin{align*}
G(x)+H(x)&=\frac{\partial P^{0}}{\partial \nu}-\frac{\partial \overline{Q}^{0}}{\partial \nu} -\nabla B\cdot\nu-\overline{Q}^{1}& \text{on $\partial\mathcal U$}\\
\Delta\psi F+\Delta \underline{P}^{-1}-\underline{P}^{-1}-\nabla B\cdot \nu+\nabla F \cdot \nu&=\frac{\partial P^{0}}{\partial \nu}-\frac{\partial \overline{Q}^{0}}{\partial \nu} -\nabla B\cdot\nu-\overline{Q}^{1}& \text{on $\partial\mathcal U$}
\end{align*}
and it follows
\begin{align*}
\frac{\partial P^{0}}{\partial\nu}&=\Delta\psi (-\overline{Q}^0)+\Delta \underline{P}^{-1}-\underline{P}^{-1}+\nabla F \cdot \nu+\frac{\partial \overline{Q}^{0}}{\partial \nu} +\overline{Q}^{1}& \text{on $\partial\mathcal U$}\\
&=2\frac{\partial \overline{Q}^{0}}{\partial \nu}+b^{-1}+\Delta \underline{P}^{-1}-\underline{P}^{-1}+\nabla F \cdot \nu +\overline{Q}^{1}& \text{on $\partial\mathcal U$,}
\end{align*}
since $\overline{Q}^0$ satisfies $\overline{Q}^0\Delta\psi+\nabla\overline{Q}^0\cdot\nabla \psi=-b^{-1}$, where we recall that
$$F(x)=-\frac{\nabla\underline{P}^{-1}\cdot \nabla\psi}{|\nabla\psi|^2} \quad\text{in $\omega$}.$$

Therefore, we obtain that $P^{0}$ is solution to
\begin{equation}\label{eq:solP0Ndbis}
    \begin{dcases}
        -\Delta P^{0}+P^{0}=a^{0} & \text{in $\mathcal{U}$}\\
        \frac{\partial P^{0}}{\partial\nu}= 2\frac{\partial \overline{Q}^{0}}{\partial \nu}+b^{-1}+\Delta \underline{P}^{-1}-\underline{P}^{-1}+\nabla F \cdot \nu +\overline{Q}^{1}& \text{on $\partial\mathcal{U}$},
    \end{dcases}
\end{equation}
and $B(x)$ is then well defined by $\underline{P}^0-\overline{Q}^0$.

%% file: main.bbl
\begin{thebibliography}{10}

\bibitem{abramovitz1972handbook}
M~Abramovitz and IA~Stegun.
\newblock Handbook of mathematical functions, vol. 55.
\newblock {\em National Bureau of Standards Applied Mathematics Series,
  Washington, DC}, 1972.

\bibitem{adams2003sobolev}
Robert~A Adams and John~JF Fournier.
\newblock {\em Sobolev spaces}.
\newblock Elsevier, 2003.

\bibitem{AE23}
G.~Allaire and A.~ERN.
\newblock {\em Optimisation et contr\^ole}.
\newblock Cours de l'Ecole Polytechnique. 2023.

\bibitem{ABF99}
P.~Angot, C.H. Bruneau, and P.~Fabrie.
\newblock A penalization method to take into account obstacles in
  incompressible viscous flows.
\newblock {\em Numer. Math.}, 81:497--520, 1999.

\bibitem{angot1999analysis}
Philippe Angot.
\newblock Analysis of singular perturbations on the brinkman problem for
  fictitious domain models of viscous flows.
\newblock {\em Mathematical methods in the applied sciences},
  22(16):1395--1412, 1999.

\bibitem{bardos1970problemes}
Claude Bardos.
\newblock Probl{\`e}mes aux limites pour les {\'e}quations aux d{\'e}riv{\'e}es
  partielles du premier ordre {\`a} coefficients r{\'e}els; th{\'e}or{\`e}mes
  d'approximation; application {\`a} l'{\'e}quation de transport.
\newblock In {\em Annales scientifiques de l'{\'E}cole normale sup{\'e}rieure},
  volume~3, pages 185--233, 1970.

\bibitem{Bensiali2014}
B.~Bensiali, G.~Chiavassa, and J.~Liandrat.
\newblock Penalization of {Robin} boundary conditions.
\newblock {\em Appl. Numer. Math.}, 96:134--152, 2015.

\bibitem{berestycki2009can}
Henri Berestycki, Odo Diekmann, Cornelis~J Nagelkerke, and Paul~A Zegeling.
\newblock Can a species keep pace with a shifting climate?
\newblock {\em Bulletin of mathematical biology}, 71(2):399--429, 2009.

\bibitem{BookBoyerFabrie}
F.~Boyer and P.~Fabrie.
\newblock {\em Mathematical tools for the study of incompressible Navier-Stokes
  equations and related models}, volume 183 of {\em Applied Mathematical
  Sciences}.
\newblock Springer, 2010.

\bibitem{BookBrezis}
H.~Brezis.
\newblock {\em Functional analysis, Sobolev spaces and partial differential
  equations}.
\newblock Universitext. Springer, 2010.

\bibitem{ciarlet1998introduction}
P.G. Ciarlet.
\newblock {\em Introduction {\`a} l'analyse num{\'e}rique matricielle et {\`a}
  l'optimisation}.
\newblock Collection Math{\'e}matiques appliqu{\'e}es pour la ma{\^\i}trise.
  Dunod, 1998.

\bibitem{dalibard2018mathematical}
Anne-Laure Dalibard and Laure Saint-Raymond.
\newblock {\em Mathematical study of degenerate boundary layers: A large scale
  ocean circulation problem}, volume 253.
\newblock American Mathematical Society, 2018.

\bibitem{deimlingnonlinear}
K.~Deimling.
\newblock Nonlinear functional analysis (springer, berlin, 1985).

\bibitem{dronioupota2002}
J.~Droniou.
\newblock Non-coercive linear elliptic problems.
\newblock {\em Potential Analysis}, 2002.

\bibitem{droniou2002}
J.~Droniou and C.~Imbert.
\newblock Solutions de viscosité et solutions variationnelles pour {EDP}
  non-linéaires.
\newblock {\em Cours de DEA, Département de Mathématiques, Montpellier II},
  2002.

\bibitem{ern2021finite}
Alexandre Ern and Jean-Luc Guermond.
\newblock {\em Finite Elements III: First-Order and Time-Dependent PDEs},
  volume~74.
\newblock Springer Nature, 2021.

\bibitem{evans2022partial}
Lawrence~C Evans.
\newblock {\em Partial differential equations}, volume~19.
\newblock American Mathematical Society, 2022.

\bibitem{E10}
L.C. Evans.
\newblock {\em Partial differential equations}.
\newblock American Mathematical Society, RI. second edition 2010.

\bibitem{gerdes2001hp}
Klaus Gerdes, Jens~Markus Melenk, Christoph Schwab, and Dominik Sch{\"o}tzau.
\newblock The hp-version of the streamline diffusion finite element method in
  two space dimensions.
\newblock {\em Mathematical Models and Methods in Applied Sciences},
  11(02):301--337, 2001.

\bibitem{gie2018singular}
Gung-Min Gie, Makram Hamouda, Chang-Yeol Jung, and Roger~M Temam.
\newblock {\em Singular perturbations and boundary layers}.
\newblock Springer, 2018.

\bibitem{hartman2002ordinary}
Philip Hartman.
\newblock {\em Ordinary differential equations}.
\newblock SIAM, 2002.

\bibitem{le2018fixed}
H.~Le~Dret.
\newblock Fixed point theorems and applications.
\newblock In {\em Nonlinear Elliptic Partial Differential Equations}, pages
  47--68. Springer, 2018.

\bibitem{lions2006perturbations}
Jacques-Louis Lions.
\newblock {\em Perturbations singuli{\`e}res dans les probl{\`e}mes aux limites
  et en contr{\^o}le optimal}, volume 323.
\newblock Springer, 2006.

\bibitem{melenk2002hp}
Jens~M Melenk.
\newblock {\em hp-Finite element methods for singular perturbations}.
\newblock Number 1796. Springer Science \& Business Media, 2002.

\bibitem{MI05}
R.~Mittal and G.~Iaccarino.
\newblock Immersed boundary methods.
\newblock {\em Annu. Rev. Fluid Mech.}, 37:239--261, 2005.

\bibitem{Pe02}
C.~Peskin.
\newblock The immersed boundary method.
\newblock {\em Acta Numerica}, 11:479--517, 2002.

\bibitem{RAB07}
I.~Ramiere, P.~Angot, and M.~Belliard.
\newblock A fictitious domain approach with spread interface for elliptic
  problems with general boundary conditions.
\newblock {\em Comput. Methods Appl. Mech. Eng.}, 196:766--781, 2007.

\bibitem{roques2008population}
Lionel Roques, Alain Roques, Henri Berestycki, and Andr{\'e} Kretzschmar.
\newblock A population facing climate change: joint influences of {Allee}
  effects and environmental boundary geometry.
\newblock {\em Population Ecology}, 50(2):215--225, 2008.

\bibitem{rouviere2009petit}
Fran{\c{c}}ois Rouvi{\`e}re.
\newblock {\em Petit guide de calcul diff{\'e}rentiel: {\`a} l'usage de la
  licence et de l'agr{\'e}gation}.
\newblock Cassini, 2009.

\bibitem{sakurai2019volume}
Teluo Sakurai, Katsunori Yoshimatsu, Naoya Okamoto, and Kai Schneider.
\newblock Volume penalization for inhomogeneous neumann boundary conditions
  modeling scalar flux in complicated geometry.
\newblock {\em Journal of Computational Physics}, 390:452--469, 2019.

\bibitem{sauter2021heterogeneous}
Stefan Sauter and C{\'e}line Torres.
\newblock The heterogeneous {H}elmholtz problem with spherical symmetry:
  Green’s operator and stability estimates.
\newblock {\em Asymptotic Analysis}, 125(3-4):289--325, 2021.

\bibitem{Sch15}
K.~Schneider.
\newblock Immersed boundary methods for numerical simulation of confined fluid
  and plasma turbulence in complex geometries: a review.
\newblock {\em J. Plasma Phys.}, 6, 2015.

\bibitem{thirumalaisamy2022handling}
Ramakrishnan Thirumalaisamy, Neelesh~A Patankar, and Amneet Pal~Singh Bhalla.
\newblock Handling {Neumann} and {Robin} boundary conditions in a fictitious
  domain volume penalization framework.
\newblock {\em Journal of Computational Physics}, 448:110726, 2022.

\end{thebibliography}
